\documentclass[11pt, a4paper, oneside]{amsart}

\usepackage[english]{babel}
\usepackage{amsmath, amsthm, amsfonts, mathrsfs, amssymb}
\usepackage{mathtools}
\mathtoolsset{centercolon}
\usepackage{mathrsfs}
\usepackage{amscd}
\usepackage{bm}

\usepackage[shortlabels]{enumitem}
\usepackage[colorlinks, citecolor = blue, urlcolor={red}]{hyperref}

\usepackage{faktor}

\allowdisplaybreaks
\setlist[itemize]{leftmargin=25pt}
\setlist[enumerate]{leftmargin=25pt}

\newtheorem{theorem}{Theorem}
\newtheorem{corollary}[theorem]{Corollary}
\newtheorem{lemma}[theorem]{Lemma}
\newtheorem{proposition}[theorem]{Proposition}

\theoremstyle{remark}
\newtheorem{remark}[theorem]{Remark}

\theoremstyle{definition}
\newtheorem{definition}[theorem]{Definition}
\newtheorem{example}[theorem]{Example}

\numberwithin{theorem}{section}
\numberwithin{equation}{section}

\def\N{{\mathbb N}}
\def\Z{{\mathbb Z}}

\def\R{{\mathbb R}}
\def\C{{\mathbb C}}

\def\T{{\mathbb T}}

\renewcommand{\P}{{\mathbb P}}

\renewcommand{\vec}[1]{\bm{#1}}
\newcommand{\Dom}{\mathcal{O}}
\newcommand{\norm}[1]{\nrm{#1}}
\newcommand{\CO}{C}

\DeclareFontFamily{U}{mathx}{\hyphenchar\font45}
\DeclareFontShape{U}{mathx}{m}{n}{<5> <6> <7> <8> <9> <10> <10.95> <12> <14.4> <17.28> <20.74> <24.88> mathx10}{}
\DeclareSymbolFont{mathx}{U}{mathx}{m}{n}
\DeclareFontSubstitution{U}{mathx}{m}{n}
\DeclareMathAccent{\widecheck}{0}{mathx}{"71}

\newcommand{\mbb}{\mathbb}

\newcommand{\mc}{\mathcal}
\newcommand{\ms}{\mathscr}
\newcommand{\mf}{\mathfrak}
\newcommand{\mrm}{\mathrm}

\DeclarePairedDelimiter\abs{\lvert}{\rvert}
\DeclarePairedDelimiter\brac[]
\DeclarePairedDelimiter\cbrace\{\}
\DeclarePairedDelimiter\ha()
\DeclarePairedDelimiter{\ip}\langle\rangle
\DeclarePairedDelimiter{\nrm}\lVert\rVert

\newcommand{\nrmb}[1]{\bigl\|#1\bigr\|}

\newcommand{\hab}[1]{\bigl(#1\bigr)}
\newcommand{\cbraceb}[1]{\bigl\{#1\bigr\}}
\newcommand{\ipb}[1]{\bigl\langle#1\bigr\rangle}
\newcommand{\bracb}[1]{\bigl[#1\bigr]}

\newcommand{\nrms}[1]{\Bigl\|#1\Bigr\|}
\newcommand{\abss}[1]{\Bigl|#1\Bigr|}
\newcommand{\has}[1]{\Bigl(#1\Bigr)}
\newcommand{\cbraces}[1]{\Bigl\{#1\Bigr\}}

\newcommand{\bracs}[1]{\Bigl[#1\Bigr]}

\newcommand{\dd}{\hspace{2pt}\mathrm{d}}

\newcommand{\ee}{\mathrm{e}}

\DeclareMathOperator{\KB}{KB}
\DeclareMathOperator{\RNP}{RNP}
\DeclareMathOperator{\re}{Re}
\DeclareMathOperator{\ind}{\mathbf{1}}
\newcommand{\one}{\ind}

\begin{document}

\title[A discrete framework for the interpolation of Banach spaces]{A discrete framework for the interpolation of Banach spaces}

\author{Nick Lindemulder}
\address[N. Lindemulder]{Institute of Analysis \\
Karlsruhe Institute of Technology \\
Englerstra{\ss}e 2 \\
76131 Karlsruhe\\
Germany}
\email{nick.lindemulder@gmail.com}

\author{Emiel Lorist}
\address[E. Lorist]{Delft Institute of Applied Mathematics\\
Delft University of Technology \\ P.O. Box 5031\\ 2600 GA Delft\\The
Netherlands} \email{e.lorist@tudelft.nl}




\thanks{The second author was supported by the Vidi subsidy 639.032.427 of the Netherlands Organisation for Scientific Research (NWO) and the Academy of Finland through grant no. 336323}

\begin{abstract}
We develop a discrete framework for the interpolation of Banach spaces, which contains the well-known real and complex interpolation methods, but also more recent methods like the Rademacher, $\gamma$- and $\ell^q$-interpolation methods. Our framework is based on a sequential structure imposed on a Banach space, which allows us to deduce properties of interpolation methods from properties of sequential structures.

Our framework has a formulation modelled after both the real and the complex interpolation methods. This enables us to extend various results, previously known only for either the real or the complex interpolation method, to all interpolation methods that fit into our framework. As applications, we prove an interpolation result for analytic operator families and an interpolation result for intersections.
\end{abstract}

\maketitle
\setcounter{tocdepth}{1}
{
\hypersetup{linkcolor=black}
\tableofcontents
}

\section{Introduction}
Interpolation of bounded linear operators on Banach spaces is a widely used technique in analysis, in which key roles are played by the real and complex interpolation methods.
Besides these well-known methods, there is a wealth of other interpolation methods used in applications. Our main goal is to develop an interpolation method for Banach spaces that are equipped with a space of sequences, which is motivated by applications of the Rademacher interpolation method \cite{KKW06}, $\gamma$-interpolation method \cite{SW06} and $\ell^q$-interpolation method \cite{Ku15} to the study of partial differential equations. Before turning to our abstract framework, we will first explain these motivating applications.

\subsection{New interpolation methods for partial differential equations}
The scales of Besov spaces and Triebel-Lizorkin spaces play prominent roles in function space theory and its applications to the study of partial differential equations (PDEs). Many classical spaces, such as Sobolev spaces and variants with fractional smoothness, are contained in them.
As both scales admit a description through Littlewood-Paley decompositions, their theories have many similarities.
However, there are also major differences.
Whereas the Besov space $B^{s}_{p,q}(\R^d)$ arises as the real interpolation space between  the Lebesgue space $L^p(\R^d)$ and the Sobolev space $W^{m}_{p}(\R^d)$,  the Triebel-Lizorkin space $F^{s}_{p,q}(\R^d)$ is in general not an interpolation space between $L^p(\R^d)$ and $W^{m}_{p}(\R^d)$ in the sense of classical interpolation theory.
This issue was, to some extend, overcome by Kunstmann in \cite{Ku15}, going beyond the realm of classical interpolation.
More explicitly, it was shown that $F^{s}_{p,q}(\R^d)$ can be obtained from $L^p(\R^d)$ and  $W^{m}_{p}(\R^d)$ by the newly introduced $\ell^q$-interpolation method.
The philosophy in \cite{Ku15} is that, roughly speaking, for the scale of Triebel-Lizorkin spaces, the $\ell^q$-method plays the same role as the real interpolation method plays for the scale of Besov spaces.

The $\ell^q$-interpolation method is formally defined for couples of Banach spaces that are equipped with an isometric embedding $J:X \to E$ into some Banach function space $E$. The idea behind this is that one can make sense of expressions like $\norm{(\sum_j|x_j|^q)^{1/q}}_{X}$ through the embedding $J$.
In this setting $\ell^q$-interpolation is, loosely speaking, obtained by modifying the discrete mean method for real interpolation by Lions and Peetre \cite{LP64}, moving the $\ell^q$-sequence space from ``the outside'' to ``the inside''.

Our interest in the $\ell^q$-interpolation method comes from trace theory for parabolic boundary value problems.
It is a classical application of  interpolation theory that in the maximal regularity approach to evolution equations, the space of initial values is given as a real interpolation space.  In the concrete case of maximal $L^p$-$L^q$-regularity for parabolic boundary value problems, this real interpolation space can be identified as Besov space.
The treatment of boundary values gets quite involved on the function space theoretic side in the case that $p \neq q$ (see e.g.\ \cite{DHP07,HL19,Li17b,LV20,PS16,Wei02}),
which is a case that is of crucial importance, as has become apparent through recent advances in the maximal $L^p$-$L^q$-regularity approach to quasi-linear PDEs based on a
novel systematic approach in critical spaces \cite{AV20a,AV20b,HNVW23, PW18,PSW18}.
For e.g.\ the heat equation with Dirichlet boundary condition on a domain $\Dom$ in $\R^d$ and an time interval $J=(0,T)$, the boundary value has to be in the intersection space
\begin{equation}\label{eq:intro:space_BV}
F^{\delta}_{p,q}(J;L^q(\partial\mathscr{O})) \cap L^p(J;F^{2\delta}_{q,q}(\partial\mathscr{O})), \qquad \delta = 1-\frac{1}{2q}.
\end{equation}
The appearances of Triebel-Lizorkin spaces suggest a connection between $\ell^q$-interpolation and the space of boundary values in analogy to the treatment of initial values.

However, in the development of the theory of $\ell^q$-interpolation, it turns out that the setting of Banach spaces equipped with an isometric embedding into a Banach function space as in \cite{Ku15} has some limitations.
It is, for instance, not broad enough to work with Littlewood-Paley decompositions, spaces on domains (defined as quotients) and anisotropic mixed-norm spaces related to~\eqref{eq:intro:space_BV}.
In these applications the relevant $\ell^q$-space of sequences comes naturally with the given function space and is dependent on the context. Trying to fit these applications into a generalization of $\ell^q$-interpolation seems to make matters unnecessarily complicated.

The issue with Littlewood-Paley decompositions and anisotropic mixed-norm spaces related to \eqref{eq:intro:space_BV} is the that $\ell^q$-space should not be placed on ``the inside'', but ``somewhere in between''. More specifically, given an inhomogeneous Littlewood-Paley decomposition $\varphi = (\varphi_{n})_{n \in \N}$ of $\R^d$, the Sobolev space $W^m_p(\R^d)$ admits the corresponding Littlewood-Paley decomposition
\begin{align*}
\nrm{f}_{W^m_p(\R^d)} &\eqsim_{p,d,m} \nrm{(2^{nm}\varphi_n*f)_{n \in \N}}_{L^{p}(\R^d;\ell^2(\N))}.
\end{align*}
This decomposition provides an isomorphic embedding into the Banach function space $L^{p}(\R^d;\ell^2(\N))$ and the natural $\ell^q$-structure $L^{p}(\R^d;\ell^2(\N;\ell^q(\Z)))$ induces an $\ell^q$-structure on $W^m_p(\R^d)$ through this embedding.
However, in this way one does not obtain the natural $\ell^q$-structure $W^m_p(\R^d;\ell^q(\Z))$.
Indeed, for $q \in (1,\infty)$, the Littlewood-Paley decomposition of the vector-valued Sobolev space $W^m_p(\R^d;\ell^q(\Z))$ takes the form
\begin{align*}
\nrm{F}_{W^m_p(\R^d;\ell^q(\Z))} &\eqsim_{p,d,m,q} \nrm{(2^{nm}\varphi_n*F)_{n \in \N}}_{L^{p}(\R^d;\ell^q(\Z;\ell^2(\N)))}.
\end{align*}
This illustrates that in order to exploit the power of harmonic analysis, one needs to leave the current framework of $\ell^q$-interpolation from \cite{Ku15}.

\bigskip

The above suggests to build the $\ell^q$-interpolation framework in a setting of Banach spaces equipped with an $\ell^q$-space of sequences, instead of having the construction of the sequence space built into the framework.
It turns out that, for a major part of the theory, the space of sequences does not even need to resemble some kind of $\ell^q$-structure, which opens the door for us to set up a very general theory of interpolation.

The advantage of such a general theory of interpolation is that it allows us to develop  other modern interpolation methods, such as the Rademacher and $\gamma$-methods, simultaneously.
These two methods are closely connected to the $H^\infty$-calculus of sectorial operators \cite{KKW06,KLW19}.
This functional calculus can be thought of as an extension of the spectral theory of self-adjoint operators on Hilbert spaces. It was initially developed by McIntosh
and collaborators \cite{CDMY96,Mc86}, motivated by the longstanding Kato square root problem, which was
eventually solved in \cite{AHLT02} using techniques inspired by the $H^\infty$-calculus.
For a general introduction to the $H^\infty$-calculus we refer the reader to \cite{HNVW17,We06}.

Verifying the boundedness of the  $H^\infty$-calculus for concrete operators can get quite involved.
A powerful tool for this are comparison principles, which allow one to transfer the property of having a bounded $H^\infty$-calculus from one sectorial operator to another. Such principles, based on the Rademacher and $\gamma$-interpolation methods, have been developed in e.g.\ \cite{KKW06,KLW19, KW17}.

\subsection{The sequentially structured interpolation method}
Motivated by the preceding discussion, we will develop a method for the interpolation of Banach spaces that are equipped with a space of sequences.
This will allow us to simultaneously develop the Rademacher method, $\gamma$-method and $\ell^q$-method and its variants. Moreover, we will treat known results for the real and complex interpolation methods in a unified fashion.

The  starting point for our framework is the discrete mean method for real interpolation by Lions and Peetre.
In order to describe their method, let $(X_0,X_1)$ be a compatible couple of Banach spaces and $\theta \in (0,1)$.  In \cite{LP64} Lions and Peetre introduced the real interpolation spaces $(X_0,X_1)_{\theta,p_0,p_1}$ for $p_0,p_1 \in [1,\infty]$ as the space of all $x \in X_0+X_1$ such that
\begin{equation}\label{eq:realmeanmethod}
  \nrm{x}_{(X_0,X_1)_{\theta,p_0,p_1}} =   \inf \,\max_{j=0,1} \,\nrmb{(\ee^{k(j-\theta)} x_k)_{k\in \Z} }_{\ell^{p_j}(\Z;X_j)}<\infty,
\end{equation}
where the infimum is taken over all  sequences $(x_k)_{k \in \Z}$ in $X_0 \cap X_1$  such that $\sum_{k \in \Z} x_k = x$ with convergence in $X_0+X_1$. These spaces are isomorphic to the real interpolation spaces $(X_0,X_1)_{\theta,p}$, defined using the K-functional, where $\frac1p = \frac{1-\theta}{p_0}+\frac{\theta}{p_1}$.
In this paper we will study the spaces defined by the right hand-side of \eqref{eq:realmeanmethod}, in which we replace the sequence spaces $\ell^{p_j}(\Z;X_j)$ by a \emph{sequence structure} $\mf{S}_j$ for $j=0,1$.

 A sequence structure $\mf{S}$ on a Banach space $X$ is a translation invariant Banach space of $X$-valued sequences such that
\begin{equation*}
  \ell^1(\Z;X) \hookrightarrow \mf{S} \hookrightarrow \ell^\infty(\Z;X)
\end{equation*}
contractively. Given a sequence structure $\mf{S}_j$ on $X_j$, we set $\mc{X}_j:=[X_j,\mf{S}_j]$ for $j=0,1$ and define the space $(\mc{X}_0,\mc{X}_1)_\theta$ as the space of all $x \in X_0+X_1$ for which
\begin{equation*}
  \nrm{x}_{(\mc{X}_0,\mc{X}_1)_\theta} :=    \inf \,\max_{j=0,1} \,\nrmb{(\ee^{k(j-\theta)} x_k)_{k\in \Z} }_{\mf{S}_j}<\infty,
\end{equation*}
where the infimum is taken over all  sequences $(x_k)_{k \in \Z}$ in $X_0 \cap X_1$  such that $\sum_{k \in \Z} x_k = x$ with convergence in $X_0+X_1$.

For specific choices of $\mf{S}_j$ for $j=0,1$, this framework includes, for example,  the following interpolation methods:
\begin{enumerate}[(i)]
  \item The real interpolation method, using the sequence spaces $\ell^{p_j}(\Z;X_j)$.
  \item \label{it:complexintro} The lower and upper complex interpolation methods \cite{Ca64}, using the space of Fourier coefficients of functions in $C(\T;X_j)$ and measures in $\Lambda^{\infty}(\T;X_j)$ respectively.
  \item The Rademacher and $\gamma$-interpolation methods \cite{KKW06,SW06}, using the random sequence spaces $\varepsilon^p(\Z;X_j)$ and $\gamma^p(\Z;X_j)$.
  \item \label{it:ellqintro} The $\ell^q$-interpolation method \cite{Ku15}, using the spaces $X_j(\ell^q(\Z))$.
  \item The $\alpha$-interpolation method \cite{KLW19} for a global Euclidean structure~$\alpha$, using the spaces $\alpha(\Z;X_j)$.
\end{enumerate}
In the literature there exist many works that unify various interpolation methods, see, for example,
 the generalized interpolation spaces by Williams \cite{Wi71}, the minimal and maximal methods of interpolation by Janson \cite{Ja81b},
the method of orbits by Ovchinnikov \cite{Ov84} and
the general real interpolation method (see \cite{BK91, BS88} and the references therein).
The closest work to our approach is the unified framework for commutator estimates by Cwikel, Kalton, Milman and Rochberg \cite{CKMR02}, for which a detailed comparison will be given in Remark \ref{remark:CKMR02}.
A major difference between these prior works and our framework is that our assumptions do not necessarily give rise to an interpolation functor in the classical sense. This relaxation from prior works allows us to include e.g. the $\ell^q$-interpolation method and related methods in our framework, which is crucial for future applications to trace theory for parabolic boundary value problems.

One of the merits of our approach is that properties of all interpolation methods fitting in our framework can be studied simultaneously. Questions regarding e.g. density of $X_0 \cap X_1$, interpolation of operators, duality, embeddings between different interpolation methods, reiteration and change of basis are reduced to properties of the associated sequence structures (see Sections \ref{section:ssi},   \ref{section:interpolationofoperators} and \ref{section:reiteration}).
 While these results are well-known for the real and complex interpolation methods, our general theorems provide a wealth of new results for e.g. the less thoroughly developed Rademacher, $\gamma$-,  $\ell^q$- and $\alpha$-interpolation methods. Moreover, some of these results were phrased as open problems for the method developed in \cite{CKMR02}, see \cite[p.662]{Ka16b}.

Our framework also explains quite clearly the need for additional assumptions in certain results for concrete interpolation methods. For example, the real interpolation method is self-dual for any compatible couple of Banach spaces (duality of the \emph{sequence} spaces $\ell^p(\Z;X)$), whereas one needs an additional geometric assumption for the complex interpolation method to be self-dual (duality of the \emph{function} spaces $C(\T;X)$).

\bigskip

As we noted before, the complex interpolation method fits into our framework by using e.g. the space of Fourier coefficients of functions in $C(\T;X_j)$. While this formulation of the complex interpolation method is well-known (see Cwikel \cite{Cw78}), it is not the original one introduced by Calder\'on \cite{Ca64}. This raises the question what the relation between the classical formulation of the complex interpolation method and our framework is. It turns out that our framework admits a complex formulation, which yields a complex formulation of all previously mentioned interpolation methods (see Section \ref{sec:complex_formulations}).
 This in particular means that, from our viewpoint, the real and complex interpolation methods are not inherently real or complex.
 These interpolation methods are rather living on opposite sides
of the Fourier transform.

Since our interpolation framework admits a real and complex formulation, results that were previously only known for either the real or the complex method, can now be extended to all interpolation methods that fit in our framework. A prime example of this observation is the fact that we are able to deduce a version of the interpolation of analytic families of operators of Stein \cite{St56} for our interpolation framework (see Theorem \ref{theorem:Steininterpolation1}). This theorem is well-known for the complex interpolation method and was proven for the $\gamma$-interpolation method in \cite{SW06}. For the specific case of real interpolation, we used similar ideas in a continuous setting in \cite{LL21c}. We remark that for the interpolation framework developed in \cite{CKMR02}, Stein interpolation was phrased as an open problem \cite[p.662]{Ka16b}.

As an example of the generalization of a result only known for the real interpolation method, we will extend Peetre's result on the interpolation of intersections  \cite{Pe74} to our interpolation framework. In particular, we will show that under suitable assumptions on the sequence structures  $\mc{X}$, $\mc{Y}$ and $\mc{Z}$ (see Theorem \ref{thm:interpol_intersection}), one has  $(\mc{X},\mc{Y})_\theta \cap (\mc{X},\mc{Z})_\theta = (\mc{X},\mc{Y} \cap \mc{Z})_\theta$. In the specific case of $\ell^q$-interpolation, this result yields intersection representations for Triebel--Lizorkin spaces.

\subsection{Open questions}
Besides the basic properties of our sequentially structured interpolation method in Section \ref{section:ssi}, the selection of topics in interpolation theory covered in this article
is based on the application of our theory to trace theory of parabolic boundary value problems and $H^\infty$-calculus. There are of course many other topics that would be interesting to study
in our framework. A non-exhaustive list of such topics is given below.
\begin{itemize}
  \item In \cite{CKMR02} Cwikel, Kalton, Milman and Rochberg build their interpolation framework to study commutator estimates. Studying such estimates in our setting, and thus obtaining commutator estimates for e.g. the $\ell^q$-interpolation method, could be quite interesting.
  \item In \cite{Sn73} {\v{S}ne\u{\i}berg} proved the stability of the invertibility of operators on the complex interpolation scale. His result was proven on the real interpolation scale by Zafran \cite{Za80}. In \cite{AKM20} this result was extended to the general interpolation framework of \cite{CKMR02}. It would be interesting to study these results in our setting as well.
  \item Wolff's reiteration theorem \cite{Wo82} (see also \cite{JNP84}) roughly states that if $X_1,X_2,X_3,X_4$ are Banach spaces such that
      \begin{itemize}
        \item $X_2$ is a real or complex interpolation space between $X_1$ and $X_3$,
        \item $X_3$ is a real or complex interpolation space between $X_2$ and $X_4$,
      \end{itemize}
      then $X_2$ and $X_3$ are also real, respectively complex, interpolation spaces between $X_1$ and $X_4$. A sequentially structured proof of this theorem could provide new insights in this area.
  \item If an operator $T$ is compact from $X_0$ to $Y_0$ and bounded from $X_1$ to $Y_1$, one may wonder whether $T$ is also compact from an intermediate space between $X_0$ and $X_1$ to an intermediate space between $Y_0$ and $Y_1$. For the real interpolation method this was answered affirmatively by Cwikel \cite{Cw92}, with an alternative proof by Cobos, K\"uhn and Schonbek \cite{CBS92}. For the complex interpolation method this is a long standing open problem, for which a breakthrough partial solution was given by Cwikel and Kalton \cite{CK95}. A, probably very hard, open question is whether one could obtain such an interpolation of compactness result in our framework.
\end{itemize}

\subsection{Structure}
\begin{itemize}
  \item In Section \ref{section:sequencestructures} we will introduce sequence structures and some of their basic properties.
  \item In Section \ref{section:ssi} we introduce the sequentially structured interpolation method and study all the basic properties, i.e. density of $X_0 \cap X_1$, equivalent norms, duality, embeddings between different interpolation methods and changes of basis. Most of the proofs are adaptations from the corresponding results for the real interpolation method. To make this section suitable for a first introduction to interpolation theory, we include full details.
  \item Section \ref{sec:complex_formulations} is our first section with major, new results. It contains complex formulations of the sequentially structured interpolation method modelled after both Calder\'on's lower and upper complex interpolation methods.
  \item In Section \ref{section:interpolationofoperators} we discuss the interpolation of operators and analytic operator families using the sequentially structured interpolation method.
  \item In Section \ref{section:reiteration} we give a general reiteration theorem for the sequentially structured interpolation method, which we make more concrete for the real, complex and $\gamma$-interpolation methods.
  \item In Section \ref{section:intersections} we generalize a result by Peetre on the interpolation of intersections  to our interpolation framework.
\end{itemize}
Results for specific interpolation methods, like the Rademacher, $\gamma$- and $\ell^q$-interpolation methods, will typically be put in examples and are scattered throughout the text. The examples should therefore not be overlooked and, in some sense, contain the main concrete results of this paper.

\subsection{Notation and conventions}
We denote by $\T$ the one-dimensional torus
$$
\T =  \faktor{\R}{2\pi\Z} \simeq S^1=\cbraceb{ e^{it} : t \in [-\pi,\pi)} \subseteq \C,
$$
which we often identify with $[-\pi,\pi)$ equipped with the Lebesgue measure.

For a Banach space $X$ we denote by $\ell^0(\Z;X)$ the space of all $X$-valued sequences $\vec{x}=(x_k)_{k \in \Z}$ equipped with the topology of pointwise convergence. The subspace of $\ell^0(\Z;X)$ consisting of all finitely nonzero sequences is denoted by $c_{00}(\Z;X)$.

For two topological vector spaces $X$ and $Y$, we will write $X=Y$ to state that $X$ and $Y$ are isomorphic, unless explicitly specified otherwise.

Given a normed space $X$ that is a linear subspace of a vector space $\mf{X}$, we will view the norm $\nrm{\,\cdot\,}_{X}$ on $X$ as an extended norm on $\mf{X}$ by setting $\nrm{x}_{X}=\infty$ for $x \in \mf{X} \setminus X$.

 By $\lesssim_{a,b,\ldots}$ we mean that there is a constant $C>0$ depending on $a$, $b$, $\ldots$ such that inequality holds and by $\eqsim_{a,b,\ldots}$ we mean that $\lesssim_{a,b,\ldots}$ and $\gtrsim_{a,b,\ldots}$ hold.

\section{Sequence structures}\label{section:sequencestructures}

A \emph{sequence structure} on a Banach space $X$ is a Banach space $\mf{S} \subseteq \ell^0(\Z;X)$ that is translation invariant and satisfies
\begin{equation}\label{eq:ell_1_infty_sandwich}
  \ell^1(\Z;X) \hookrightarrow \mf{S} \hookrightarrow \ell^\infty(\Z;X)
\end{equation}
contractively. Equivalently, $\mf{S}$ is a sequence structure on $X$ if
the following hold:
\begin{align}
\label{eq:sequence1}   \nrm{(\ldots,0,x,0,\ldots)}_{\mf{S}} &= \nrm{x}_X, && x \in X,\\
\label{eq:sequence2}   \nrm{(x_{k+n})_{k \in \Z}}_{\mf{S}} &= \nrm{\vec{x}}_{\mf{S}}, && \vec{x} \in \mf{S}, \,n \in \Z,\\
\label{eq:sequence3}  \nrm{x_n}_X & \leq \nrm{\vec{x}}_{\mf{S}}, && \vec{x} \in \mf{S}, \, n \in \Z.
\end{align}
The pair $\mc{X}=[X,\mf{S}]$ is called a \emph{{sequentially} structured Banach space}.
In the rest of the paper, we will use  the shorthand notation $\mc{X}$, $\mc{Y}$ and $\mc{Z}$ to denote the sequentially structured Banach spaces $\mc{X}=[X,\mf{S}]$, $\mc{Y}=[Y,\mf{T}]$ and $\mc{Z}=[Z,\mf{U}]$, respectively, with a similar convention for indexed variants.

For each $n \in \N$ we define the \emph{Ces\`aro operator} $\CO_n$ on $\ell^0(\Z;X)$ by
\begin{equation}\label{eq:Cesaro_operator}
\CO_n\vec{x} := \frac{1}{n+1}\sum_{m=0}^n (\ldots, 0,x_{-m},\ldots,x_m,0,\ldots), \qquad \vec{x} \in \ell^0(\Z;X).
\end{equation}
If we have 
\begin{align}
\label{eq:sequence4b}\sup_{n \in \N}\, \nrm{\CO_n\vec{x}}_{\mf{S}}\leq \nrm{\vec{x}}_{\mf{S}}, \qquad \vec{x} \in \mf{S},
\end{align}
we call $\mf{S}$ (and $\mc{X}$) \emph{Ces\`aro bounded} and if we additionally have
\begin{align}
\label{eq:sequence4}\lim_{n \to \infty} \CO_n\vec{x}=\vec x, \qquad \vec{x} \in \mf{S},
\end{align}
we call $\mf{S}$ (and $\mc{X}$)  \emph{Ces\`aro convergent}. Note that \eqref{eq:sequence4b} and \eqref{eq:sequence4} hold in particular if
\begin{align}\label{eq:sequence5b}
\sup_{n \in \N} \,\nrmb{(\ldots, 0,x_{-n},\ldots,x_n,0,\ldots) }_{\mf{S}} &\leq \nrm{\vec{x}}_{\mf{S}},  &&\vec{x} \in \mf{S},\\
\lim_{n \to \infty} (\ldots, 0,x_{-n},\ldots,x_n,0,\ldots) &= \vec x,  &&\vec{x} \in \mf{S},\label{eq:sequence5}
\end{align}
respectively. In most concrete examples we will be able to check \eqref{eq:sequence5b} and \eqref{eq:sequence5}, but we use the slightly more general assumptions \eqref{eq:sequence4b} and \eqref{eq:sequence4} to make our results applicable to the complex interpolation method.

 If $\mf{S}$ is a Ces\`aro convergent sequence structure, then  $c_{00}(\Z;X)$ is dense in $\mf{S}$. Conversely we have the following:

\begin{lemma}\label{lemma:c00dense}
  Let $\mf{S}$ be an  Ces\`aro bounded sequence structure on a Banach space $X$ and suppose that $c_{00}(\Z;X)$ is dense in $\mf{S}$. Then $\mf{S}$ is Ces\`aro convergent.
\end{lemma}

\begin{proof}
  As $\mf{S}$ is  Ces\`aro bounded, we have
  $\nrm{\CO_n}_{\mf{S} \to \mf{S} }\leq 1$ for all $n \in \N$. Moreover note that $\CO_n\vec{x} \to \vec{x}$ as $n \to \infty$ for all $\vec{x} \in c_{00}(\Z;X)$, so the lemma follows by density.
\end{proof}

For a sequence structure $\mf{S}$ on a Banach space $X$ and $a \in (0,\infty)$ we define the weighted space $\mf{S}(a) \subseteq \ell^0(\Z;X)$ as the Banach space
\begin{align*}
\mf{S}(a) &:= \left\{ \vec{x} \in \ell^{0}(\Z;X) : (a^k x_k)_{k \in \Z} \in \mf{S} \right\}
\end{align*}
with norm $ \nrm{\vec{x}}_{\mf{S}(a)}
:= \nrm{(a^k x_k)_{k \in \Z}}_{\mf{S}}.$
Note that $\mf{S}$ satisfies \eqref{eq:sequence4b} or \eqref{eq:sequence4} if and only if $\mf{S}(a)$ satisfies \eqref{eq:sequence4b} or \eqref{eq:sequence4}, respectively.

We will now give some examples of sequence  structures.
\begin{example}\label{example:sequencestructures} Let $X$ be a Banach space.
  \begin{enumerate}[(i)]
    \item \label{it:ssreal} Let $p \in [1,\infty]$. Then $\ell^p(\Z;X)$ is a sequence structure on $X$, which is Ces\`aro convergent if $p<\infty$ and Ces\`aro bounded if $p=\infty$.
    \item \label{it:sscomplex} Let $p \in [1,\infty)$ and define
    \begin{equation*}
      \widehat{L}^p(\T;X):= \cbraceb{\widehat{f}: f \in L^p(\T;X)},
    \end{equation*}
        with norm
\begin{equation*}
  \nrm{\widehat{f}\hspace{2pt}}_{\widehat{L}^p(\T;X)} := \frac{1}{(2\pi)^{1/p}}\nrm{f}_{L^p(\T;X)},
\end{equation*}
    where we use the $2\pi$-periodic Fourier transform
    \begin{equation*}
  \widehat{f}(k):= \frac{1}{2 \pi} \int_{\T} f(t)\ee^{-ikt}\dd t, \qquad  k \in \Z.
\end{equation*}
    Then $\widehat{L}^p(\T;X)$ is a Ces\`aro convergent sequence structure on $X$. Similarly, we define the Ces\`aro convergent sequence structure $\widehat{C}(\T;X)$ on $X$ as the space of all sequences $\widehat{f}$ for $f \in C(\T;X)$ with norm
$
  \nrm{\widehat{f}\hspace{2pt}}_{\widehat{C}(\T;X)} := \nrm{f}_{C(\T;X)}.
$
    \item \label{it:sscomplexdual} Let $p \in [1,\infty]$ and let $\Lambda^p(\T;X)$ be the space of vector-valued measures $\nu$ such that the Radon--Nikod\'ym derivative of $\abs{\nu}$ with respect to the Lebesgue measure is in $L^p(\T)$ (see \cite[Chapter 2]{Pi16} for an introduction).
    We define
    \begin{equation*}
      \widehat{\Lambda}^p(\T;X):= \cbraceb{\widehat{f}: f \in \Lambda^p(\T;X)},
    \end{equation*}
        with norm
\begin{equation*}
  \nrm{\widehat{\mu}\hspace{2pt}}_{\widehat{\Lambda}^p(\T;X)} := \frac{1}{(2\pi)^{1/p}}\nrm{\mu}_{\Lambda^p(\T;X)},
\end{equation*}
    where we use the $2\pi$-periodic Fourier transform
    \begin{equation*}
  \widehat{\mu}(k):= \frac{1}{2\pi} \int_{\T} \ee^{- ikt}\dd \mu(t), \qquad  k \in \Z.
\end{equation*}
    Then $\widehat{\Lambda}^p(\T_;X)$ is an Ces\`aro bounded sequence structure on $X$. 
    \item \label{it:ssrademacher} Let $(\varepsilon_k)_{k \in \Z}$ be a sequence of independent Rademachers on a probability space $(\Omega,\P)$ and fix $p \in [1,\infty)$. Define $\varepsilon^p(\Z;X)$ as the space of all $\vec{x} \in \ell^0(\Z;X)$ such that $\sum_{k\in \Z} \varepsilon_kx_k$ converges in $L^p(\Omega;X)$ with norm
        \begin{equation*}
         \nrm{\vec{x}}_{\varepsilon^p(\Z;X)} :=  \nrmb{\sum_{k\in \Z} \varepsilon_kx_k}_{L^p(\Omega;X)} = \,\sup_{n \in \N}\, \nrmb{\sum_{k = -n}^n \varepsilon_kx_k}_{L^p(\Omega;X)}.
        \end{equation*}
    Then $\varepsilon^p(\Z;X)$ is a Ces\`aro convergent sequence structure on $X$.
 \item \label{it:ssgauss} We define $\gamma^p(\Z;X)$ similarly to $\varepsilon^p(\Z;X)$, with a sequence of independent Gaussians $(\gamma_k)_{k \in \Z}$ instead of Rademachers $(\varepsilon_k)_{k \in \Z}$. Then $\gamma^p(\Z;X)$ is a Ces\`aro convergent sequence structure on $X$.
    \item \label{it:sslattice} Suppose that $X$ is a Banach lattice and let $q \in [1,\infty)$. Define
    $X(\ell^q_+(\Z))$ as the space of  all $\vec{x} \in \ell^0(\Z;X)$ for which the norm
     \begin{equation*}
         \nrm{\vec{x}}_{X(\ell_+^q(\Z))} :=  \sup_{n \in \N} \,\nrms{ \has{\sum_{k=-n}^n \abs{x_k}^q}^{1/q}}_X
        \end{equation*}
        is finite, where $\ha{\sum_{k=-n}^n \abs{x_k}^q}^{1/q}$ is defined through the Krivine calculus (see \cite[Theorem 1.d.1]{LT79}). Then $X(\ell^q_+(\Z))$ is an Ces\`aro bounded sequence structure on $X$. We let $X(\ell^q(\Z))$ be the closure of $c_{00}(\Z;X)$ in  $X(\ell^q_+(\Z))$, which is a Ces\`aro convergent sequence structure on $X$. 

        We define $X(\ell^\infty(\Z))=X(\ell^\infty_+(\Z))$ as the space of  all $\vec{x} \in \ell^0(\Z;X)$ for which the norm
     \begin{equation*}
         \nrm{\vec{x}}_{X(\ell^\infty(\Z))} :=  \sup_{n \in \N} \,\nrms{ \sup_{\abs{k} \leq n} \abs{x_k}}_X
        \end{equation*}
        is finite, which is a Ces\`aro bounded sequence structure on $X$. We let $X(c_0(\Z))$ be the closure of $c_{00}(\Z;X)$ in $X(\ell^\infty(\Z))$, which is a Ces\`aro convergent sequence structure on $X$.
    \item \label{it:sseuclidean} In \cite{KLW19} the notion of a \emph{Euclidean structure} $\alpha$ on $X$ is introduced. The space $\alpha(\Z;X)$ defined in \cite[Section 3.3]{KLW19} is a  Ces\`aro convergent sequence structure on $X$. These spaces are a generalization of the $\gamma^p(\Z;X)$-spaces in \ref{it:ssgauss}.
  \end{enumerate}
\end{example}

\begin{proof}
\ref{it:ssreal} is clear.
 For \ref{it:sscomplex} we note that for $x \in X$
  \begin{equation*}
    \nrm{(\ldots,0,x,0,\ldots)}_{\widehat{L}^p(\T;X)} = \frac{1}{(2\pi)^{1/p}}\nrm{\ind_{\T}\cdot x}_{L^p(\T;X)} = \nrm{x}_X,
  \end{equation*}
  and for  $\vec{x} \in \widehat{L}^p(\T;X)$, $f \in L^p(\T;X)$ such that $\widehat{f} = \vec{x}$, $n \in \Z$ and $\epsilon = \pm 1$, writing $e_n(t) = \ee^{int}$, we have
\begin{align*}
 &\nrm{(x_{k+n})_{k \in \Z}}_{\widehat{L}^p(\T;X)} = \frac{1}{(2\pi)^{1/p}}\nrm{ t\mapsto e_{- n}(t) \cdot f( t)}_{L^p(\T;X)}  = \nrm{\vec{x}}_{\widehat{L}^p(\T;X)},\\
   &\nrm{x_n}_X= \frac{1}{(2\pi)^{1+1/p}}\nrm{ e_n *f}_{L^p(\T;X)}
 \leq \frac{1}{(2\pi)^{1/p}}\nrm{f}_{L^p(\T;X)} = \nrm{\vec{x}}_{\widehat{L}^p(\T;X)}.
\end{align*}
 Moreover, since the Fej\'er kernel is an approximate identity (see e.g. \cite[Proposition 3.1.10]{Gr14a}), \eqref{eq:sequence4} holds. Similar reasoning also shows that $\widehat{C}(\T;X)$ is a Ces\`aro convergent sequence structure and $\widehat{\Lambda}^p(\T;X)$ in \ref{it:sscomplexdual} is an Ces\`aro bounded sequence structure.
 For \ref{it:ssrademacher}, \eqref{eq:sequence1} and \eqref{eq:sequence2} are clear, \eqref{eq:sequence3} follows from \cite[Proposition 6.1.5]{HNVW17} and \eqref{eq:sequence5} is a consequence of the convergence of $\sum_{k\in \Z} \varepsilon_kx_k$ in $L^p(\Omega;X)$. \ref{it:ssgauss} follows similarly.
The spaces in \ref{it:sslattice} are Banach spaces by \cite[Section 1.d]{LT79}, the sequence structure properties and \eqref{eq:sequence4b} are clear and \eqref{eq:sequence4} follows from Lemma \ref{lemma:c00dense}. Finally, \ref{it:sseuclidean} is discussed in detail in \cite{KLW19}.
 \end{proof}

\section{Sequentially structured interpolation}\label{section:ssi}
 We call a pair of Banach spaces $(X_0,X_1)$ a \emph{compatible couple} if both $X_0$ and $X_1$ are continuously embedded into a Hausdorff topological vector space $\mf{X}$. We define the Banach spaces
\begin{align*}
  X_0\cap X_1 &:= \cbrace{x \in \mf{X}:x \in X_0 \text{ and } x \in X_1}, \\
  X_0+X_1 &:= \cbrace{x \in \mf{X}: x=x_0+x_1 \text{ with } x_0 \in X_0,x_1 \in X_1},
\end{align*}
with norms
\begin{align*}
  \nrm{x}_{X_0\cap X_1} &:= \max_{j=0,1}\,\nrm{x}_{X_j},\\
  \nrm{x}_{X_0+X_1}&:= \inf\cbraceb{\nrm{x_0}_{X_0}+\nrm{x_1}_{X_1}: x=x_0+x_1, x_0 \in X_0, x_1 \in X_1}.
\end{align*}
We call a Banach space $X$ an \emph{intermediate space} with respect to $(X_0,X_1)$ if we have continuous embeddings
\begin{equation*}
  X_0\cap X_1 \hookrightarrow X \hookrightarrow X_0+X_1.
\end{equation*}
If this is the case we define $X^\circ$ as the closure of $X_0 \cap X_1$ in $X$.

Let $\mc{X}_j=[X_j,\mf{S}_j]$ be sequentially structured Banach spaces for $j=0,1$. We call the pair $(\mc{X}_0,\mc{X}_1)$ a compatible couple of sequentially structured Banach spaces if the pair $(X_0,X_1)$ is a compatible couple of Banach spaces.
We will define a discrete interpolation method for $(\mc{X}_0,\mc{X}_1)$, which is modelled after the mean method for real interpolation by Lions and Peetre \cite{LP64}.

\begin{definition}\label{definition:sequenceinterpolation}
Let $(\mc{X}_0,\mc{X}_1)$ be a compatible couple of sequentially structured Banach spaces and let $\theta \in (0,1)$.
For $x \in X_0+X_1$ we set
\begin{align*}
  \nrm{x}_{(\mc{X}_0,\mc{X}_1)_\theta}
  & := \inf \,\nrm{\vec{x}}_{\mf{S}_0(\ee^{-\theta})\cap \mf{S}_1(\ee^{1-\theta})},
\end{align*}
where the infimum is taken over all  sequences $\vec{x} \in \mf{S}_0(\ee^{-\theta}) \cap \mf{S}_1(\ee^{1-\theta})$ such that $\sum_{k \in \Z} x_k = x$ with convergence in $X_0+X_1$.
We define
\begin{equation*}
  (\mc{X}_0,\mc{X}_1)_\theta := \cbraceb{x \in X_0+X_1:\nrm{x}_{(\mc{X}_0,\mc{X}_1)_\theta}< \infty},
\end{equation*}
with norm $\nrm{\,\cdot\,}_{(\mc{X}_0,\mc{X}_1)_\theta}$.
\end{definition}

\begin{remark}\label{remark:automaticconvergenceX0X1}
  Note that the convergence of $\sum_{k \in \Z}x_k$ in $X_0+X_1$ is automatic for $\vec{x} \in \mf{S}_0(\ee^{-\theta}) \cap \mf{S}_1(\ee^{1-\theta})$, since we have by \eqref{eq:sequence1} and \eqref{eq:sequence3} that
\begin{align*}
  \sum_{k \in \Z} \nrm{x_k}_{X_0+X_1} &\leq \sum_{k \in \Z_{\leq 0}} \ee^{k\theta} \nrm{\ee^{-k\theta}x_k}_{X_0} + \sum_{k \in \Z_{>0}} \ee^{-k(1-\theta)} \nrm{\ee^{k(1-\theta)}x_k}_{X_1}\\
  &\leq  \nrm{\vec{x}}_{\mf{S}_0(\ee^{-\theta}) \cap \mf{S}_1(\ee^{1-\theta})}\has{ \sum_{k \in \Z_{\leq 0}}\ee^{k\theta}  +   \sum_{k \in \Z_{>0}} \ee^{-k(1-\theta)}}<\infty.
\end{align*}
\end{remark}

Before turning to the properties of our interpolation method, let us connect this definition to already existing interpolation methods using the sequence structures from Example~\ref{example:sequencestructures}.

\begin{example} \label{example:interpolationmethods}
Let $(\mc{X}_0,\mc{X}_1)$ be a compatible couple of sequentially structured Banach spaces and let $\theta \in (0,1)$.
\begin{enumerate}[(i)]
  \item \label{it:ssireal} If $\mf{S}_j = \ell^{p_j}(\Z;X_j)$ for $j=0,1$ with $p_0,p_1 \in [1,\infty]$, then
  \begin{equation*}
    (\mc{X}_0,\mc{X}_1)_\theta  = (X_0,X_1)_{\theta,p_0,p_1},
  \end{equation*}
where $(X_0,X_1)_{\theta,p_0,p_1}$ is the Lions--Peetre mean method \cite{LP64}. In particular, for $\frac1p = \frac{1-\theta}{p_0}+\frac{\theta}{p_1}$ we have
\begin{equation*}
  (\mc{X}_0,\mc{X}_1)_\theta  =(X_0,X_1)_{\theta,p},
\end{equation*}
where $(X_0,X_1)_{\theta,p}$ denotes the real interpolation method.
\item \label{it:ssicomplex} Let $p_0,p_1 \in [1,\infty]$. If
\begin{equation*}
  \mf{S}_j= \begin{cases}
     \widehat{L}^{p_j}(\T;X_j)&\text{ if $p_j\in [1,\infty)$},\\
     \widehat{C}(\T;X_j) &\text{ if $p_j = \infty$},
  \end{cases}
\end{equation*}
for  $j=0,1$, then
  \begin{equation*}
    (\mc{X}_0,\mc{X}_1)_\theta  = \brac{X_0,X_1}_{\theta},
  \end{equation*}
where $\brac{X_0,X_1}_{\theta}$ denotes the complex interpolation method \cite{Ca64}. If $\mf{S}_j=\widehat{\Lambda}^{p_j}(\T;X_j)$
for  $j=0,1$, then
  \begin{equation*}
    (\mc{X}_0,\mc{X}_1)_\theta  = \brac{X_0,X_1}^{\theta},
  \end{equation*}
where $\brac{X_0,X_1}^{\theta}$ denotes Calder\'on's upper complex interpolation method \cite{Ca64}. We will prove these identities in Examples \ref{example:complexmethod} and  \ref{example:complexmethod;upper}.
\item \label{it:ssirademacher}If $\mf{S}_j = \varepsilon^p(\Z;X_j)$ for $j=0,1$ and $p \in [1,\infty)$, then
\begin{equation*}
  (\mc{X}_0,\mc{X}_1)_\theta  = (X_0,X_1)_{\theta,\varepsilon},
\end{equation*}
where $(X_0,X_1)_{\theta,\varepsilon}$ denotes the Rademacher interpolation method \cite{KKW06}.
\item \label{it:ssigaussian} If $\mf{S}_j = \gamma^p(\Z;X_j)$ for $j=0,1$ and $p \in [1,\infty)$, then
\begin{equation*}
  (\mc{X}_0,\mc{X}_1)_\theta  = (X_0,X_1)_{\theta,\gamma},
\end{equation*}
where $(X_0,X_1)_{\theta,\gamma}$ denotes the $\gamma$-interpolation method \cite{SW06,SW09}.
\item \label{it:ssilattice} Suppose that there is a Banach lattice $E_j$ and an isometric embedding  $J \colon X_j \to E_j$ for $j=0,1$. Let $q_0,q_1\in [1,\infty]$. If  $\mf{S}_j$ is the space of all $\vec{x} \in \ell^0(\Z;X_j)$ such that
    \begin{equation*}
      \nrm{\vec{x}}_{\mf{S}_j} := \nrm{(Jx_k)_{k\in \Z}}_{E_j(\ell^{q_j}(\Z))}<\infty,
    \end{equation*}
then
 \begin{equation*}
  (\mc{X}_0,\mc{X}_1)_\theta  = (X_0,X_1)_{\theta,\ell^{q_0}, \ell^{q_1}},
\end{equation*}
 where $(X_0,X_1)_{\theta,\ell^{q_0},\ell^{q_1}}$ denotes the $\ell^q$-interpolation method introduced in \cite{Ku15}. We define $(X_0,X_1)_{\theta,c_0,c_0}$ and $(X_0,X_1)_{\theta,\ell^{q_0}_+,\ell^{q_1}_+}$ similarly.
\item \label{it:ssiEuclidean} The Rademacher, $\gamma$- and $\ell^q$-interpolation methods were inspired by the $\alpha$-interpolation method for a global Euclidean structure $\alpha$ introduced in \cite{KLW19}, a preprint of which circulated since the beginning of the 2000s. If $\mf{S}_j = \alpha(\Z;X_j)$ for $j= 0,1$, then
    \begin{equation*}
  (\mc{X}_0,\mc{X}_1)_\theta  = (X_0,X_1)_{\theta,\alpha},
\end{equation*}
 where $(X_0,X_1)_{\theta,\alpha}$ denotes the $\alpha$-interpolation method. For details we refer to \cite[Section 3.3 and 3.4]{KLW19}.
\end{enumerate}
\end{example}

The real and complex interpolation methods in Example \ref{example:interpolationmethods}\ref{it:ssireal}-\ref{it:ssicomplex} were already studied by Peetre  \cite{Pe71} from this viewpoint, where also the problem was posed to study these methods in an abstract, translation invariant setting as we do in the present article (see \cite[Probl\`eme 1]{Pe71}).

\begin{remark}\label{remark:CKMR02}
In \cite{CKMR02} Cwikel, Kalton, Milman and Rochberg gave a different approach to \cite[Probl\`eme 1]{Pe71} in order to study commutator estimates, using so-called \emph{pseudo-$\Z$-lattices} instead of sequence structures. Comparing our basic assumptions to the pseudo-$\Z$-lattices in \cite{CKMR02}, we note:
\begin{itemize}
  \item We assume translation invariance, while this is not a requirement for a pseudo-$\Z$-lattice. Various results in \cite{CKMR02} do additionally assume either translation invariance or decay of the operator norm of translation operators.
  \item We assume contractivity of coordinate projections, while this is not a requirement for a pseudo-$\Z$-lattice. Various results in \cite{CKMR02} rely on \emph{Laurent compatibility}, which is implied by the contractivity of coordinate projections.
  \item A pseudo-$\Z$-lattice is assumed to be an interpolation functor, whereas our sequentially structured interpolation framework does not insist on defining an interpolation functor in the classical sense.
\end{itemize}
While the assumptions on a sequence structure are rather different from the assumptions on  a pseudo-$\Z$-lattices, most of the main examples satisfying either set of assumptions also satisfy the other set. The key example that fits our framework, but not the framework of \cite{CKMR02}, is the $\ell^q$-interpolation method in Example \ref{example:interpolationmethods}\ref{it:ssilattice} and its variants discussed in the introduction. These methods will play a major role in our applications to trace theory for parabolic boundary value problems in future work.
\end{remark}

\begin{remark}
  There are also various generalizations of real interpolation based on the $K$-functional, replacing $L^q(\R)$ or $\ell^q(\Z)$ by a Banach function or sequence space $E$, see e.g. \cite{CK17, Nil82, BK91, BS88} and the references therein. We note that these generalizations are disjoint from our approach, as, on the one hand, the assumptions on $E$ can be more lenient than our assumptions on $\mf{S}$, but, on the other hand, the complex, Rademacher, $\gamma$- and $\ell^q$-interpolation methods do not fit into such a framework.
\end{remark}

\subsection{Basic properties}
Having motivated the sequentially structured interpolation method with some examples, let us turn to a few basic properties. We start by proving that $(\mc{X}_0,\mc{X}_1)_\theta$ is an intermediate space with respect to $X_0$ and $X_1$.

\begin{proposition}\label{proposition:Banachembeddings}
  Let $(\mc{X}_0,\mc{X}_1)$ be a compatible couple of sequentially structured Banach spaces and let $\theta \in (0,1)$. Then $(\mc{X}_0,\mc{X}_1)_\theta$ is a Banach space and we have
  continuous embeddings
  \begin{equation*}
  X_0 \cap X_1 \hookrightarrow (\mc{X}_0,\mc{X}_1)_\theta \hookrightarrow X_0 + X_1.
\end{equation*}
\end{proposition}

The key ingredient in the proof of Proposition~\ref{proposition:Banachembeddings} is the observation that $(\mc{X}_0,\mc{X}_1)_\theta$ can be realized as a quotient, which for future reference will be convenient to record as the following remark.

\begin{remark}\label{rmk:proposition:Banachembeddings}
Let $(\mc{X}_0,\mc{X}_1)$ be a compatible couple of sequentially structured Banach spaces and let $\theta \in (0,1)$. By Remark~\ref{remark:automaticconvergenceX0X1},
\begin{equation*}
T \colon \mf{S}_0(\ee^{-\theta}) \cap \mf{S}_1(\ee^{1-\theta}) \to X_{0}+X_1,
\end{equation*}
given by $\vec{x} \mapsto \sum_{k \in \Z}x_{k}$ is a well-defined, bounded linear operator.
Note that $\mrm{ran}(T) = (\mc{X}_0,\mc{X}_1)_\theta$ and that $T$ induces a well-defined isometric linear isomorphism
\begin{equation*}
\widetilde{T}:\faktor{\mf{S}_0(\ee^{-\theta}) \cap \mf{S}_1(\ee^{1-\theta})}{\ker(T)} \to (\mc{X}_0,\mc{X}_1)_\theta
\end{equation*}
given by $\brac{\vec{x}} \mapsto T\vec{x}$.
\end{remark}

\begin{proof}[Proof of Proposition~\ref{proposition:Banachembeddings}]
Let $T$ and $\widetilde{T}$ be as in Remark \ref{rmk:proposition:Banachembeddings}. Then $\ker(T)$ is a closed linear subspace of the Banach space $\mf{S}_0(\ee^{-\theta}) \cap \mf{S}_1(\ee^{1-\theta})$.
Therefore, the quotient space $\brac{\mf{S}_0(\ee^{-\theta}) \cap \mf{S}_1(\ee^{1-\theta})}/\ker(T)$, which is isometrically isomorphic to $(\mc{X}_0,\mc{X}_1)_\theta$ through $\widetilde{T}$, is a Banach space.

We will now check the continuous embeddings. Take $x \in X_0\cap X_1$, then we have by \eqref{eq:sequence1}
\begin{equation*}
  \nrm{x}_{(\mc{X}_0,\mc{X}_1)_\theta} \leq \max_{j=0,1}\, \nrm{(\ldots,0,x,0,\ldots)}_{\mf{S}_j} \leq   \nrm{x}_{X_0\cap X_1}.
\end{equation*}
For the second embedding let $x \in (\mc{X}_0,\mc{X}_1)_\theta$ and take a sequence  $\vec{x} \in \mf{S}_0(\ee^{-\theta}) \cap \mf{S}_1(\ee^{1-\theta})$ such that $T\vec{x}=\sum_{k \in \Z} x_k = x$.
Then, by the boundedness of $T$, we have
\begin{equation*}
\nrm{x}_{X_0+X_1} = \nrm{T\vec{x}}_{X_0+X_1} \lesssim \nrm{\vec{x}}_{\mf{S}_0(\ee^{-\theta}) \cap \mf{S}_1(\ee^{1-\theta})}.
\end{equation*}
The second embedding now follows by taking the infimum over all such $\vec{x}$.
\end{proof}

In Definition \ref{definition:sequenceinterpolation} we can actually improve the maximum of the $\mf{S}_j$-norms to a log-convex combination of the $\mf{S}_j$-norms, as we will show next.

\begin{lemma}\label{lemma:logconvex}
  Let $(\mc{X}_0,\mc{X}_1)$ be a compatible couple of sequentially structured Banach spaces and let $\theta \in (0,1)$.
For $x \in (\mc{X}_0,\mc{X}_1)_\theta$ we have
\begin{align*}
  \nrm{x}_{(\mc{X}_0,\mc{X}_1)_\theta} \eqsim \inf\,\nrm{\vec{x}}_{\mf{S}_0(\ee^{-\theta})}^{1-\theta}\cdot \nrm{\vec{x}}_{\mf{S}_1(\ee^{1-\theta})}^\theta,
\end{align*}
where the infimum is taken over all  sequences $\vec{x} \in \mf{S}_0(\ee^{-\theta}) \cap \mf{S}_1(\ee^{1-\theta})$ such that $\sum_{k \in \Z} x_k = x$ with convergence in $X_0+X_1$.
In particular,
\begin{equation*}
  \nrm{x}_{(\mc{X}_0,\mc{X}_1)_\theta} \lesssim_\theta \nrm{x}_{X_0}^{1-\theta} \nrm{x}_{X_1}^\theta, \qquad x \in X_0 \cap X_1.
\end{equation*}
\end{lemma}

\begin{proof}
We only need to  show the inequality ``$\lesssim$''.
Let $\vec{x} \in \mf{S}_0(\ee^{-\theta}) \cap \mf{S}_1(\ee^{1-\theta})$ such that $\sum_{k \in \Z} x_k = x$ with convergence in $X_0+X_1$, set
$
  a_j:= \nrm{\vec{x}}_{\mf{S}_j(\ee^{j-\theta})}$ for $j =0,1
$
and let $n \in \Z$ such that $\ee^n\leq \frac{a_0}{a_1} \leq \ee^{n+1}$. Then the sequence $\vec{y}:=(x_{k-n})_{k \in \Z}$ satisfies $\sum_{k \in \Z} y_k = x$ and by \eqref{eq:sequence2} we have
\begin{equation*}
 \nrm{\vec{y}}_{\mf{S}_j(\ee^{j-\theta})} =  \nrmb{(\ee^{k(j-\theta)}x_{k-n})_{k \in \Z}}_{\mf{S}_j} = \ee^{n(j-\theta)} a_j, \qquad j=0,1.
\end{equation*}
Therefore
\begin{align*}
  \nrm{x}_{(\mc{X}_0,\mc{X}_1)_\theta} &\leq \max_{j=0,1} \, \ee^{n(j-\theta)} a_j= a_0^{1-\theta}a_1^\theta \cdot  \max_{j=0,1}\, \hab{\ee^{-n}\frac{ a_0}{a_1}}^{\theta-j}\leq \ee \cdot a_0^{1-\theta}a_1^\theta.
\end{align*}
Taking the infimum over all such $\vec{x}$ concludes the proof of the equivalence.
\end{proof}

As a consequence of Lemma \ref{lemma:logconvex}, we obtain that $X_0\cap X_1$ is dense in $(\mc{X}_0,\mc{X}_1)_\theta$ if either $\mf{S}_0$ or $\mf{S}_1$ is Ces\`aro convergent, which without the aid of Lemma \ref{lemma:logconvex} would require a Ces\`aro convergence assumption on both $\mf{S}_0$ and $\mf{S}_1$.

\begin{corollary}\label{corollary:denseX0X1}
  Let $(\mc{X}_0,\mc{X}_1)$ be a compatible couple of Ces\`aro bounded sequentially structured Banach spaces and let $\theta \in (0,1)$. If either $\mf{S}_0$ or $\mf{S}_1$ is Ces\`aro convergent, then $X_0\cap X_1$ is dense in $(\mc{X}_0,\mc{X}_1)_\theta$.
\end{corollary}

\begin{proof}
Fix $x \in (\mc{X}_0,\mc{X}_1)_\theta$ and $\vec{x} \in \mf{S}_0(\ee^{-\theta}) \cap \mf{S}_1(\ee^{1-\theta})$ such that $\sum_{k \in \Z} x_k = x$ with convergence in $X_0+X_1$.
For $n \in \N$ we define $\vec{y}^n := \CO_n\vec{x}$ and $y^n := \sum_{k \in \Z} y_k^n$, where $\CO_n$ is the Ces\`aro operator given in \eqref{eq:Cesaro_operator}.
Then $y^n \in X_0\cap X_1$ and by Lemma \ref{lemma:logconvex}
\begin{equation*}
  \nrm{x-y^n}_{(\mc{X}_0,\mc{X}_1)_\theta} \lesssim \nrmb{\vec{x}-\vec{y}^n}_{\mf{S}_0(e^{-\theta})}^{1-\theta}\cdot \nrmb{\vec{x}-\vec{y}^n}_{\mf{S}_1(e^{1-\theta})}^\theta.
\end{equation*}
Now, by \eqref{eq:sequence4} and \eqref{eq:sequence4}, one of the terms on the right-hand side converges to $0$ as $n \to \infty$ and the other stays bounded, finishing the proof.
\end{proof}

\subsection{Finite approximation}
Using Corollary \ref{corollary:denseX0X1}, we can often reduce considerations to elements of $X_0\cap X_1$. For such elements, we can simplify our arguments even further by only considering finitely nonzero sequences in Definition \ref{definition:sequenceinterpolation}. Moreover, in applications it is sometimes useful to restrict arguments to a dense subspace of $X_0\cap X_1$, for example the space of simple functions when $X_0$ and $X_1$ are $L^p$-spaces.

The following approximation lemma will play a key role in many of our subsequent results. In the setting of the interpolation framework in \cite{CKMR02}, a related result was obtained in \cite{Iv12}.

\begin{lemma}\label{lemma:finiterep}
  Let $(\mc{X}_0,\mc{X}_1)$ be a compatible couple of Ces\`aro bounded sequentially structured Banach spaces, let $\breve{X}$ be a dense subspace of $X_0\cap X_1$ and let $\theta \in (0,1)$.
For $x \in \breve{X}$ we have
\begin{align*}
  \nrm{x}_{(\mc{X}_0,\mc{X}_1)_\theta}
  & \eqsim \inf \,\nrm{\vec{x}}_{\mf{S}_0(\ee^{-\theta})\cap \mf{S}_1(\ee^{1-\theta})},
\end{align*}
where the infimum is taken over all $\vec{x} \in c_{00}(\Z;\breve{X})$ such that $\sum_{k \in \Z} x_k = x$.
\end{lemma}

\begin{proof} Since we are taking the infimum over a smaller collection of sequences than the collection of sequences in Definition \ref{definition:sequenceinterpolation}, we only need to show that the infimum can be dominated by $\nrm{x}_{(\mc{X}_0,\mc{X}_1)_\theta}$.

\textbf{Step 1:} We will first consider the special case that $\breve{X}=X_0 \cap X_1$.
Let $x \in X_0 \cap X_1$ and take $\vec{y} \in \mf{S}_0(\ee^{-\theta}) \cap \mf{S}_1(\ee^{1-\theta})$ such that
$\sum_{k \in \Z} y_k = x$ with convergence in $X_0+X_1$ and
\begin{equation}\label{eq:finiterepnorm}
  \nrm{\vec{y}}_{\mf{S}_0(\ee^{-\theta})\cap \mf{S}_1(\ee^{1-\theta})} \leq 2 \, \nrm{x}_{(\mc{X}_0,\mc{X}_1)_\theta}.
\end{equation}
For $n \in \N$ let $C_n$ be the Ces\'aro operator as in \eqref{eq:Cesaro_operator} and define $ \vec{z}^n := \CO_n\vec{y}$.
For $m \in \Z$ set
  \begin{align*}
    {z}^{n,+}_m &:= \frac{1}{n+1}\sum_{k={m+1}}^\infty y_{k} =  \frac{1}{n+1} \has{x- \sum_{k=-\infty}^m y_{k}},&& \\
    z^{n,-}_m &:= \frac{1}{n+1} \sum_{k=-\infty}^{-m-1} y_{k} = \frac{1}{n+1} \has{x- \sum_{k=m}^\infty y_{k}},
  \end{align*}
  which converge in $X_0+X_1$. Note that for $M \in \Z$ we have by \eqref{eq:sequence3}
  \begin{align*}
    \sum_{k =-\infty}^M \nrm{y_k}_{X_0} &\leq   \nrmb{(\ee^{-k\theta}y_k)_{k \in \Z}}_{\mf{S}_0} \cdot \sum_{k =-\infty}^M \ee^{k\theta} <\infty,\\
     \sum_{k =M}^\infty \nrm{y_k}_{X_1} &\leq   \nrmb{(\ee^{k(1-\theta)}y_k)_{k \in \Z}}_{\mf{S}_1}\cdot \sum_{k =M}^\infty  \ee^{-k(1-\theta)} <\infty.
  \end{align*}
  Since $x \in X_0\cap X_1$, we deduce that $z^{n,+}_m, z^{n,-}_m \in X_0\cap X_1$ for all $m \in \Z$. Thus, defining $ \vec{w}^n := \vec{z}^n + \vec{w}^{n,+}+\vec{w}^{n,-}$
with
  \begin{equation*}
    {w}^{n,\pm}_m := \begin{cases}
      z^{n,\pm}_m &\text{if }m=\pm 0,\cdots,\pm n\\
      0&\text{otherwise}
    \end{cases}, \qquad m \in \Z,
  \end{equation*}
we have that $\vec{w}^n \in c_{00}(\Z;X_0\cap X_1)$ and $\sum_{k \in \Z} w_k^n = x$ with convergence in $X_0+X_1$. Since $\mf{S}_0$ and $\mf{S}_1$ are Ces\`aro bounded, using \eqref{eq:finiterepnorm} we find
\begin{equation*}
   \nrm{\vec{z}^n}_{\mf{S}_0(\ee^{-\theta})\cap \mf{S}_1(\ee^{1-\theta})} \leq \nrm{\vec{y}}_{\mf{S}_0(\ee^{-\theta})\cap \mf{S}_1(\ee^{1-\theta})} \leq 2 \, \nrm{x}_{(\mc{X}_0,\mc{X}_1)_\theta}.
\end{equation*}
Moreover, using \eqref{eq:sequence1}, \eqref{eq:sequence3} and \eqref{eq:finiterepnorm}, we have
  \begin{align*}
    \nrm{\vec{w}^{n,+}}_{\mf{S}_0(\ee^{-\theta})} 
    &\leq \frac{1}{n+1} \sum_{m=0}^n \ee^{-m\theta} \has{\nrm{x}_{X_0} +  \sum_{k=-\infty}^m  \nrm{y_k}_{X_0}}\\
    &\leq \frac{1}{n+1} \cdot \has{ \frac{\ee^\theta\,\nrm{x}_{X_0}}{\ee^\theta-1} + \sum_{m=0}^n \ee^{-m\theta} \sum_{k=-\infty}^m \ee^{k\theta}  \nrm{(\ee^{-\ell\theta}y_\ell)_{\ell \in \Z}}_{\mf{S}_0}}\\
    &\leq \frac{1}{n+1} \cdot \frac{\ee^\theta\,\nrm{x}_{X_0}}{\ee^\theta-1}  +  \frac{2\ee^{\theta}}{\ee^\theta-1} \cdot \nrm{x}_{(\mc{X}_0,\mc{X}_1)_\theta}
  \end{align*}
  and similarly
   \begin{align*}
    \nrm{\vec{w}^{n,-}}_{\mf{S}_0(\ee^{-\theta})}
    &\leq \frac{1}{n+1} \sum_{m=0}^{n} \ee^{m\theta} {\sum_{k=-\infty}^{-m-1}  \nrm{y_k}_{X_0}}\\
    &\leq \frac{1}{n+1} \sum_{m=0}^n \ee^{m\theta} \sum_{k=-\infty}^{-m-1} \ee^{k\theta}  \nrm{(\ee^{-\ell\theta}y_\ell)_{\ell \in \Z}}_{\mf{S}_0}\\
    &\leq \frac{2}{\ee^\theta-1} \cdot \nrm{x}_{(\mc{X}_0,\mc{X}_1)_\theta}.
  \end{align*}
Doing similar computations for $\mf{S}_1$, one also obtains
\begin{align*}
   \nrm{\vec{w}^{n,+}}_{\mf{S}_1(\ee^{1-\theta})} &\leq \frac{2}{\ee^{1-\theta}-1} \cdot \nrm{x}_{(\mc{X}_0,\mc{X}_1)_\theta}\\
   \nrm{\vec{w}^{n,-}}_{\mf{S}_1(\ee^{1-\theta})} &\leq \frac{1}{n+1} \cdot \frac{\ee^{1-\theta}\,\nrm{x}_{X_1}}{\ee^{1-\theta}-1}  +  \frac{2\ee^{1-\theta}}{\ee^{1-\theta}-1} \cdot \nrm{x}_{(\mc{X}_0,\mc{X}_1)_\theta}
\end{align*}
  Thus, taking
  \begin{equation*}
    n +1\geq \max \cbraces{\frac{\nrm{x}_{X_0}}{\nrm{x}_{(\mc{X}_0,\mc{X}_1)_\theta}}, \frac{\nrm{x}_{X_1}}{\nrm{x}_{(\mc{X}_0,\mc{X}_1)_\theta}}}
  \end{equation*}
  we obtain
\begin{equation*}
\max_{j=0,1}\, \nrm{\vec{w}^n}_{\mf{S}_j(\ee^{j-\theta})}\lesssim_\theta \nrm{x}_{(\mc{X}_0,\mc{X}_1)_\theta}.
\end{equation*}

\textbf{Step 2:} We derive the general case, for which we fix $x \in \breve{X}$. Then, in particular, $x \in X_0 \cap X_1$ and by Step 1 there thus exists $\vec{w} \in c_{00}(\Z;X_0 \cap X_1)$ with $\sum_{k \in \Z}w_k=x$ and
\begin{equation}\label{eq:ell1approxdense;first_case}
\max_{j=0,1}\, \nrm{\vec{w}}_{\mf{S}_j(\ee^{j-\theta})}\lesssim_\theta \nrm{x}_{(\mc{X}_0,\mc{X}_1)_\theta}.
\end{equation}
Now take a $\vec{v}\in c_{00}(\Z;\breve{X})$ such that
\begin{equation}\label{eq:ell1approxdense}
  \sum_{k \in \Z} \max\cbrace{\ee^{-k\theta}, \ee^{k(1-\theta)}}\nrm{v_k-w_k}_{X_0\cap X_1} \leq \nrm{x}_{(\mc{X}_0,\mc{X}_1)_\theta}.
\end{equation}
Then $\sum_{m\in \Z} w_m - v_m   = x-\sum_{m \in \Z}v_m \in \breve{X}$, thus setting
\begin{equation*}
  x_k := \begin{cases}
    v_k, &k \in \Z \setminus\cbrace{0},\\
    v_0 + \sum_{m\in \Z}  w_m- v_m, &k=0,
  \end{cases}
\end{equation*}
we have that  $\vec{x} \in c_{00}(\Z;\breve{X})$, $\sum_{k \in \Z} x_k = x$ and, using \eqref{eq:sequence1}, \eqref{eq:ell1approxdense;first_case} and \eqref{eq:ell1approxdense}, we have
\begin{align*}
  \max_{j=0,1}\, \nrm{\vec{x}}_{\mf{S}_j(\ee^{j-\theta})}&\leq \max_{j=0,1}\, \nrm{\vec{v} - \vec{w}}_{\mf{S}_j(\ee^{j-\theta})} +\max_{j=0,1} \, \nrm{\vec{w}}_{\mf{S}_j(\ee^{j-\theta})} \\&\hspace{1cm} + \nrmb{\sum_{m\in \Z}  w_m- v_m}_{X_0\cap X_1}\\
  &\lesssim_\theta \nrm{x}_{(\mc{X}_0,\mc{X}_1)_\theta},
\end{align*}
which proves the lemma.
\end{proof}

Let $(\mc{X}_0,\mc{X}_1)$ be a couple of sequentially structured Banach spaces. Recall that for $j=0,1$ we defined $X_j^\circ$ and $(\mc{X}_0,\mc{X}_1)_\theta^\circ$ as the closure of $X_0\cap X_1$ in $X_j$ and $(\mc{X}_0,\mc{X}_1)_\theta$ respectively. Let $\mf{S}_j^\circ$ denote the closure of $c_{00}(\Z;X^\circ_j)$ in $\mf{S}_j$. Then we note that
$$
(\mc{X}_0^\circ,\mc{X}_1^\circ) := ([X_0^\circ,\mf{S}_0^\circ],[X_1^\circ,\mf{S}_1^\circ])
$$
is a compatible couple of sequentially structured Banach spaces.

\begin{proposition}\label{proposition:Xcirc}
  Let $(\mc{X}_0,\mc{X}_1)$ be a compatible couple of Ces\`aro bounded sequentially structured Banach spaces and let $\theta \in (0,1)$. Then we have $(\mc{X}_0,\mc{X}_1)_\theta^\circ = (\mc{X}_0^\circ,\mc{X}_1^\circ)_\theta.$ If additionally either $\mf{S}_0$ or $\mf{S}_1$ is  Ces\`aro convergent, then $(\mc{X}_0,\mc{X}_1)_\theta = (\mc{X}_0^\circ,\mc{X}_1^\circ)_\theta.$
\end{proposition}

\begin{proof}
Since $\mc{X}_0^\circ$ is Ces\`aro convergent by Lemma \ref{lemma:c00dense}, we note that $X_0\cap X_1$ is dense in $(\mc{X}_0^\circ,\mc{X}_1^\circ)_\theta$ by Corollary~\ref{corollary:denseX0X1}. Thus, for the first claim it suffices to prove a norm equivalence for $x \in X_0\cap X_1$. This norm equivalence is a direct consequence of Lemma \ref{lemma:finiterep}. The second claim follows from  Corollary \ref{corollary:denseX0X1}.
\end{proof}

\subsection{Alternative mean method formulation}
 We based the sequentially structured interpolation method on the discrete second Lions--Peetre mean method \cite{LP64}. In this subsection we will consider a formulation of the sequentially structured interpolation method based on the first discrete mean method of Lions--Peetre. In the framework of Cwikel, Kalton, Milman and Rochberg, a related theorem was proven in \cite[Section 8]{CKMR02}

\begin{theorem}\label{theorem:meanmethod}
Let $(\mc{X}_0,\mc{X}_1)$ be a compatible couple of sequentially structured Banach spaces and let $\theta \in (0,1)$.
For $x \in X_0+X_1$ we have
\begin{align*}
  \nrm{x}_{(\mc{X}_0,\mc{X}_1)_\theta} &\eqsim_\theta \inf{\sum_{j=0,1}\,\nrm{\vec{x}^j}_{ \mf{S}_j(\ee^{j-\theta})}}=\nrm{(\ldots,x,x,x,\ldots)}_{\mf{S}_0(\ee^{-\theta})+ \mf{S}_1(\ee^{1-\theta})},
\end{align*}
where the infimum is taken over all  $\vec{x}^0 \in \mf{S}_0(\ee^{-\theta})$ and $\vec{x}^1 \in \mf{S}_1(\ee^{1-\theta})$ such that $x=x_k^0+x_k^1$ for all $k \in \Z$.
\end{theorem}

\begin{proof}
  Take $x \in (\mc{X}_0,\mc{X}_1)_\theta$ and choose an $\vec{x} \in \mf{S}_0(\ee^{-\theta})\cap \mf{S}_1(\ee^{1-\theta})$  such that $\sum_{k\in \Z}x_k =x$ with convergence in $X_0 +X_1$  and
  \begin{equation}\label{eq:j = 0estimate}
   \nrm{\vec{x}}_{\mf{S}_0(\ee^{-\theta})\cap \mf{S}_1(\ee^{1-\theta})} \leq 2 \,\nrm{x}_{(\mc{X}_0,\mc{X}_1)_\theta}.
     \end{equation} For $n \in \Z$ define $y^0_n:= \sum_{k=-\infty}^n x_k$ and $y^1_n:= \sum_{k=n+1}^\infty x_k$. Then it is clear that $y^0_k+y_1^k = x$ for all $k \in \Z$. Moreover, for $j = 0,1$ we have by \eqref{eq:sequence1} and \eqref{eq:sequence3} that for any $n_1,n_2 \in \Z$ with $n_1<n_2$
    \begin{align*}
      \nrmb{\sum_{k=n_1}^{n_2} x_k}_{X_j} &\leq \sum_{k=n_1}^{n_2} \ee^{-k(j-\theta)}\nrm{\ee^{k(j-\theta)}x_k}_{X_j}\\
      &\leq\nrmb{(\ee^{k(j-\theta)}x_k)_{k \in \Z }}_{\mf{S}_j} \sum_{k=n_1}^{n_2} \ee^{-k(j-\theta)}\\
      &\leq 2 \,\nrm{x}_{(\mc{X}_0,\mc{X}_1)_\theta}\sum_{k=n_1}^{n_2} \ee^{-k(j-\theta)}.
    \end{align*}
For $j = 0$ this tends to zero when $n_1,n_2 \to -\infty$ and for $j = 1$ this tends to zero when $n_1,n_2 \to \infty$, which implies that $y^0_k \in X_0$ and $y^1_k \in X_1$ for all $k \in \Z$.
Using \eqref{eq:sequence2}, we furthermore have
\begin{align*}
 \nrm{\vec{y}^0}_{\mf{S}_0(\ee^{-\theta})}
 &=\nrms{  \hab{\sum_{k=-\infty}^n \ee^{-n\theta}x_{k} }_{n \in \Z}}_{\mf{S}_0}\\
  &=\nrms{  \hab{\sum_{k\in \Z_{\leq 0}}\ee^{k\theta} \ee^{-(n+k)\theta}x_{n+k} }_{n \in \Z}}_{\mf{S}_0}\\
 &\leq \sum_{k \in \Z_{\leq 0}}\ee^{k\theta} \,\nrmb{(\ee^{-(n+k)\theta}x_{n+k})_{n \in \Z}}_{\mf{S}_0}
 = \frac{\ee^\theta}{\ee^\theta-1} \,\nrm{\vec{x}}_{\mf{S}_0(\ee^{-\theta})}.
\end{align*}
Combined with a similar estimate in $\mf{S}_1(\ee^{1-\theta})$ and \eqref{eq:j = 0estimate}, this proves that
\begin{equation*}
   \inf{\sum_{j=0,1}\,\nrm{\vec{x}^j}_{ \mf{S}_j(\ee^{j-\theta})}}\lesssim_\theta \nrm{x}_{(\mc{X}_0,\mc{X}_1)_\theta}.
\end{equation*}

Conversely, take $x \in X_0 + X_1$ such that the right-hand side is less than some $C>0$. Choose $\vec{x}^0 \in \mf{S}_0(\ee^{-\theta})$ and $\vec{x}^1 \in \mf{S}_1(\ee^{1-\theta})$ such that $x_k^0+x_k^1 = x$ for all $k \in \Z$ and
 \begin{align}\label{eq:sequencemean}
\sum_{j=0,1}\,\nrm{\vec{x}^j}_{ \mf{S}_j(\ee^{j-\theta})} \leq 2\, C.
\end{align}
For $n \in \Z$ define $y_n := x^0_{n+1}-x^0_n =-(x^1_{n+1}-x^1_n) \in X_0\cap X_1$. Then for $k_1,k_2 \in \Z$ with $k_1<k_2$ we have
\begin{equation*}
  \sum_{n=k_1}^{k_2} y_n = x^0_{k_2+1}-x_{k_1}^0= -( x^1_{k_2+1}-x_{k_1}^1).
\end{equation*}
 Moreover, by \eqref{eq:sequence3} and \eqref{eq:sequencemean}, we have for $j = 0,1$
 \begin{equation*}
   \nrm{x^j_k}_{X_j} \leq 2\cdot \ee^{-k(j-\theta)} \,C, \qquad k \in \Z,
 \end{equation*}
 and therefore
 $$x = x-\lim_{k_1 \to -\infty} x^0_{k_1} - \lim_{k_2 \to \infty} x^1_{k_2+1} = \lim_{k_1,k_2 \to \infty}\sum_{n =-k_1}^{k_2} y_n = \sum_{n \in \Z} y_n$$
 with convergence in $X_0+X_1$. Using  \eqref{eq:sequencemean} we furthermore have
 \begin{align*}
   \nrm{x}_{(\mc{X}_0,\mc{X}_1)_\theta}&\leq\max_{j=0,1}\,\nrmb{(\ee^{n(j-\theta)}y_n)_{n \in \Z}}_{\mf{S}_j} \\&\leq \max_{j=0,1}\,(1+\ee^{-(j-\theta)})\nrmb{(\ee^{k(j-\theta)}x_k^j)_{k \in \Z}}_{\mf{S}_j} \leq 2(1+\ee) C,
 \end{align*}
 which finishes the proof.
\end{proof}

\subsection{Duality}\label{subsection:duality}
 For a sequentially structured Banach space $\mc{X}= [X,\mf{S}]$ we will write $\mc{X^*}:= [X^*,\mf{S}^*]$. As a preparation for the duality of the sequentially structured interpolation method, we note the following lemma.

\begin{lemma}\label{lemma:dualstructure}
  Let $\mc{X}$ be a Ces\`aro convergent sequentially structured Banach space. Then $\mc{X}^*$ is a Ces\`aro bounded sequentially structured Banach space.
\end{lemma}

\begin{proof}
 Since $\mf{S}$ is a Ces\`aro convergent sequence structure, we have $\ell^1(\Z;X) \hookrightarrow \mf{S} \hookrightarrow c_0(\Z;X)$ densely and contractively. Therefore, by \cite[Proposition 1.3.3]{HNVW16} we have $\ell^1(\Z;X^*) \hookrightarrow \mf{S}^* \hookrightarrow \ell^\infty(\Z;X^*)$ contractively. Moreover for $\vec{x}^* \in \mf{S}^*$ we have, by the  translation invariance of $\mf{S}$, that
 \begin{equation*}
   \nrm{(x_{k+n}^*)_{k \in \Z}}_{\mf{S}^*} = \sup_{\nrm{\vec{x}}_{\mf{S}}\leq 1} \, \abss{\sum_{k \in \Z} \ip{x_{k},x_{ k+n}^*}}= \sup_{\nrm{\vec{x}}_{\mf{S}}\leq 1} \, \abss{\sum_{k \in \Z} \ip{x_{ k- n},x_{k}^*}} = \nrm{\vec{x}^*}_{\mf{S}^*}
 \end{equation*}
 for all $n \in \Z$, which proves that $\mf{S}^*$ is a sequence structure. Finally, for $n \in \N$ let $C_n$ denote the Ces\`aro operator as in \eqref{eq:Cesaro_operator} and note that for $\vec{x}^* \in \mf{S}^*$ we have, using that $\mf{S}$ is Ces\`aro bounded, that
 \begin{align*}
   \nrm{C_n\vec{x}^* }_{\mf{S}^*} &= \sup_{\nrm{\vec{x}}_{\mf{S}}\leq 1} \abss{\ip{C_n\vec{x},\vec{x}^*}}\leq \nrm{\vec{x}^*}_{\mf{S}^*}.
 \end{align*}
 Therefore, $\mf{S}^*$ is Ces\`aro bounded.
\end{proof}

When $(X_0,X_1)$ is a compatible couple of Banach spaces such that $X_0\cap X_1$ is dense in $X_0$ and $X_1$, the dual pair  $(X_0^*,X^*_1)$ is also a compatible couple of Banach spaces and we have
\begin{equation}\label{eq:proposition:duality;dual_sum}
\begin{aligned}
  (X_0\cap X_1)^* = X_0^*+X_1^*,\\
  (X_0+ X_1)^* = X_0^*\cap X_1^*,
\end{aligned}
\end{equation}
see e.g. \cite[Theorem 2.7.1]{BL76}.

Furthermore, note that for a Ces\`aro convergent sequentially structured Banach space $\mc{X}$ we have
\begin{equation}\label{eq:weightedseqdual}
\mc{X}(a)^* = \mc{X}^*(a^{-1}),\qquad a \in (0,\infty)
\end{equation}
which can be seen easily after observing that $\ell^1(\Z;X_j)(a) \hookrightarrow \mf{S}_j(a)$ densely for $j=0,1$ and thus $\mf{S}_j(a)^* \hookrightarrow \ell^\infty(\Z;X_j)(a^{-1})$.

Using Lemma \ref{lemma:dualstructure} and the alternative formulation of $(\mc{X}_0,\mc{X}_1)_\theta$ in Theorem~\ref{theorem:meanmethod}, we can now prove a duality result for the sequentially structured interpolation method. Let us note that in the framework of \cite{CKMR02} this was posed as an open problem, see \cite[p.662]{Ka16b}. In the following proposition we call a sequentially structured Banach space $\mc{X}$ reflection invariant if
\begin{equation*}
  \nrm{\vec{x}}_{\mf{S}} = \nrm{(x_{-k})_{k \in \Z}}_{\mf{S}}, \qquad \vec{x} \in \mf{S}.
\end{equation*}
In this case the reflection operator defines a isometry from $\mc{X}(a)$ to $\mc{X}(a^{-1})$.

\begin{proposition}\label{proposition:duality}
  Let $(\mc{X}_0,\mc{X}_1)$ be a compatible couple of reflection invariant Ces\`aro convergent sequentially structured Banach spaces and $\theta \in (0,1)$. Suppose that $X_0\cap X_1$ is dense in $X_0$ and $X_1$.
Then $(\mc{X}_0,\mc{X}_1)_\theta^* = (\mc{X}_0^*,\mc{X}_1^*)_\theta.$
\end{proposition}

\begin{proof}
Note that, using the reflection invariance of the sequence structures, the norm equivalence from Theorem~\ref{theorem:meanmethod} means that
$$
S:(\mc{X}_0,\mc{X}_1)_\theta \to \mf{S}_0(\ee^{\theta}) + \mf{S}_0(\ee^{\theta-1}),
$$
given by $x \mapsto (\ldots,x,x,x,\ldots)$ is a topological linear embedding.
Therefore, using \eqref{eq:proposition:duality;dual_sum} and \eqref{eq:weightedseqdual}, we have (see e.g.\ \cite[Theorem~III.10.1 and Proposition~VI.1.8]{Co90})
$$
\mathrm{ran}(S)^* = \faktor{\mf{S}_0^*(\ee^{-\theta}) \cap \mf{S}_1^*(\ee^{1-\theta})}{\mathrm{ran}(S)^{\perp}} = \faktor{\mf{S}_0^*(\ee^{-\theta}) \cap \mf{S}_1^*(\ee^{1-\theta})}{\ker(S^*)}.
$$
So the adjoint
\begin{equation*}\label{eq:proposition:duality;S*}
S^*:\mf{S}_0^*(\ee^{-\theta}) \cap \mf{S}_1^*(\ee^{1-\theta}) \to (\mc{X}_0,\mc{X}_1)_\theta^*
\end{equation*}
induces a topological linear isomorphism
\begin{equation}\label{eq:proposition:duality;tilde_S*}
\widetilde{S}^*: \faktor{\mf{S}_0^*(\ee^{-\theta}) \cap \mf{S}_1^*(\ee^{1-\theta})}{\ker(S^*)} \to (\mc{X}_0,\mc{X}_1)_\theta^*,\, \brac{\vec{x}^*} \mapsto S^*\vec{x}^*.
\end{equation}

In order to compute $S^*$, let $\vec{x}^* \in \mf{S}_0^*(\ee^{-\theta}) \cap \mf{S}_1^*(\ee^{1-\theta})$ and $x \in X_0 \cap X_1$.
Then, by Remark~\ref{remark:automaticconvergenceX0X1} and \eqref{eq:proposition:duality;dual_sum}, we have $\vec{x}^* \in \ell^1(\Z;(X_0\cap X_1)^*)$, so
\begin{align*}
\ip{x,S^*\vec{x}^*} = \ip{Sx,\vec{x}^*} = \ip{(\ldots,x,x,x,\ldots),\vec{x}^*} =
\sum_{k \in \Z}\ip{x,x^*_k} = \ipb{x,\sum_{k \in \Z}x^*_k}.
\end{align*}
Since we know that $(\mc{X}_0,\mc{X}_1)_\theta^* \hookrightarrow (X_0 \cap X_1)^*$ by Corollary~\ref{corollary:denseX0X1}, we deduce that
$S^*\vec{x}^* = \sum_{k \in \Z}x^*_k$.
In view of Remark~\ref{rmk:proposition:Banachembeddings}, we may thus conclude that $(\mc{X}_0,\mc{X}_1)_\theta^* = (\mc{X}_0^*,\mc{X}_1^*)_\theta$.
\end{proof}

Using Proposition \ref{proposition:duality} we immediately obtain duality results for the methods introduced in Example \ref{example:interpolationmethods}. In the following example we will use some Banach space geometry.
For an introduction to the \emph{Radon--N\'ykodim property}, or $\RNP$, we refer to \cite[Chapter 2]{Pi16} and for nontrivial type we refer to \cite[Chapter 7]{HNVW17}. Finally, for an introduction to \emph{Kantorovich--Banach spaces}, or $\KB$-spaces, we refer to \cite[Section 2.4]{MN91}

\begin{example}\label{example:dual}
Let $(X_0,X_1)$ be a compatible couple of Banach spaces and let $\theta \in (0,1)$. Suppose that $X_0\cap X_1$ is dense in $X_0$ and $X_1$.
\begin{enumerate}[(i)]
  \item \label{it:dualreal} For $p_0,p_1 \in [1,\infty)$ we have
  \begin{equation*}
    (X_0,X_1)_{\theta,p_0,p_1}^* = (X_0^*,X_1^*)_{\theta,p_0',p_1'}.
  \end{equation*}
\item \label{it:dualcomplex} We have
  \begin{equation*}
      \brac{X_0,X_1}_{\theta}^* = \brac{X_0^*,X_1^*}^\theta
  \end{equation*}
and if either $X_0^*$ or $X_1^*$ has $\RNP$, then
\begin{equation*}
      \brac{X_0,X_1}_{\theta}^* = \brac{X_0^*,X_1^*}_\theta.
  \end{equation*}
\item \label{it:dualrademachergauss}If $X_0$ and $X_1$ have nontrivial type, we have
\begin{align*}
  \hab{(X_0,X_1)_{\theta,\varepsilon}}^* &= (X_0^*,X_1^*)_{\theta,\varepsilon},\\
  \hab{(X_0,X_1)_{\theta,\gamma}}^* &= (X_0^*,X_1^*)_{\theta,\gamma}.
\end{align*}
\item \label{it:duallattice}If $X_0$ and $X_1$ are Banach lattices, we have for $p_0,p_1 \in [1,\infty)$
 \begin{equation*}
\hab{(X_0,X_1)_{\theta,\ell^{p_0},\ell^{p_1}}}^* = (X_0^*,X_1^*)_{\theta,\ell^{p_0'}_+,\ell^{p_1'}_+}.
\end{equation*}
and if either $X_0$ or $X_1$ is a $\KB$-space, then
 \begin{equation*}
\hab{(X_0,X_1)_{\theta,\ell^{p_0},\ell^{p_1}}}^* = (X_0^*,X_1^*)_{\theta,\ell^{p_0'},\ell^{p_1'}}.
\end{equation*}
\end{enumerate}
\end{example}

Note that the statements for the real and complex interpolation methods are well-known, the statements for the Rademacher and $\gamma$-interpolation methods could be extracted from \cite[Section 3.3]{KLW19} and the statement for the $\ell^q$-interpolation method is new.  All statements follow in a unified way from Example \ref{example:interpolationmethods}, Proposition \ref{proposition:duality} and suitable identifications of dual sequence space structures.

\begin{proof}[Proof of Example \ref{example:dual}]
     \ref{it:dualreal} is a consequence of $\ell^{p_j}(\Z;X_j)^* = \ell^{p_j'}(\Z;X_j^*)$ for $j=0,1$ (see \cite[Proposition 1.3.3]{HNVW16}). The first statement in \ref{it:dualcomplex} follows from the duality
    \begin{align*}
        L^2(\T;X_j)^* = \Lambda^2(\T;X_j^*), \qquad j=0,1,
    \end{align*}
    see \cite[Chapter 2]{Pi16}. For the second statement in \ref{it:dualcomplex}, assume without loss of generality that $X_0$ has $\RNP$. Then ${L}^2(\T;X_0) = {\Lambda}^2(\T;X_0)$, so ${\Lambda}^2(\T;X_0)$ is Ces\`aro convergent. Thus, since $[X_1,\widehat{\Lambda}^2(\T;X_1)]^\circ = [X_1,\widehat{L}^2(\T;X_1)]$, the statement follows from Proposition \ref{proposition:Xcirc}. For \ref{it:dualrademachergauss} note that $X_0$ and $X_1$ are $K$-convex by \cite[Theorem 7.4.23]{HNVW17}. Therefore we have
    \begin{align*}
      \varepsilon^2(\Z;X_j^*) &=  \varepsilon^2(\Z;X_j)^* ,  && j=0,1,\\
       \gamma^2(\Z;X_j^*) &=  \gamma^2(\Z;X_j)^* ,  && j=0,1,
    \end{align*}
     by \cite[Theorem 7.4.14 and 9.1.24]{HNVW17}, from which \ref{it:dualrademachergauss} follows. For the first identity in \ref{it:duallattice} note that $X(\ell^{p_j}(\Z))^* = X^*(\ell^{p_j'}_+(\Z))$ by \cite[Section 1.d.]{LT79}. For the final identity note that  $\ell^{p_j}(\Z;X_j) = \ell^{p_j}_+(\Z;X_j)$ for either $j=0$ or $j=1$, from which the embedding follows by Proposition \ref{proposition:Xcirc} as before.
\end{proof}

When either $\mf{S}_0^*$ or $\mf{S}_1^*$  is Ces\`aro convergent, it follows from Proposition \ref{proposition:duality} and Corollary \ref{corollary:denseX0X1} that $X_0^*\cap X_1^*$ is norming for $(\mc{X}_0,\mc{X}_1)_\theta$.
Using the description \eqref{eq:proposition:duality;tilde_S*} from the proof Proposition \ref{proposition:duality}, we can actually deduce that $X_0^*\cap X_1^*$ is norming for $(\mc{X}_0,\mc{X}_1)_\theta$ without any assumptions on $\mf{S}_0^*$ or $\mf{S}_1^*$ .

The key step in the proof of this norming result will be the following lemma.

\begin{lemma}\label{lem:proposition:dualitynorming}
 Let $(\mc{X}_0,\mc{X}_1)$ be a compatible couple of Ces\`aro convergent sequentially structured Banach spaces and let $\theta \in (0,1)$. Suppose that $X_0\cap X_1$ is dense in $X_0$ and $X_1$.
 Then $$c_{00}(\Z;X_0^* \cap X_1^*) \subseteq \mf{S}_0^*(\ee^{-\theta}) \cap \mf{S}_1^*(\ee^{1-\theta})$$ is norming for $\mf{S}_0(\ee^{\theta}) + \mf{S}_1(\ee^{\theta-1})$.
\end{lemma}
\begin{proof}
Let $\vec{x} \in \mf{S}_0(\ee^{\theta}) + \mf{S}_1(\ee^{\theta-1})$ and let $\varepsilon > 0$. As a consequence of \eqref{eq:proposition:duality;dual_sum}, there exists an $\vec{x}^* \in \mf{S}_0^*(\ee^{-\theta}) \cap \mf{S}_1^*(\ee^{1-\theta})$ of norm $1$ such that $$\nrm{\vec{x}}_{\mf{S}_0(\ee^{\theta}) + \mf{S}_1(\ee^{\theta-1})} \leq \abs{\ip{\vec{x},\vec{x}^*}}+\frac{\varepsilon}{2}.$$
For $n \in \N$ we define $\vec{x}^n := \CO_n\vec{x}$ and $\vec{y}^n := \CO_n\vec{x}^*$, where $\CO_n$ is the Ces\`aro operator given in \eqref{eq:Cesaro_operator}.
Note that $\ip{\vec{x}^n,\vec{x}^*} = \ip{\vec{x},\vec{y}^n}$.
As $\mc{X}_j$ is Ces\`aro convergent, we have $\vec{x} = \lim_{n \to \infty }\vec{x}^n$ in $\mf{S}_0(\ee^{\theta}) + \mf{S}_1(\ee^{\theta-1})$, so that
\begin{align*}
\nrm{\vec{x}}_{\mf{S}_0(\ee^{\theta}) + \mf{S}_1(\ee^{\theta-1})} &\leq \abs{\ip{\vec{x},\vec{x}^*}}+\frac{\varepsilon}{2}
= \lim_{n \to \infty} \abs{\ip{\vec{x}^n,\vec{x}^*}}+\frac{\varepsilon}{2} \\
&= \lim_{n \to \infty} \abs{\ip{\vec{x},\vec{y}^n}}+\frac{\varepsilon}{2}.
\end{align*}
There thus exists $n \in \N$ such that $\nrm{\vec{x}}_{\mf{S}_0(\ee^{\theta}) + \mf{S}_1(\ee^{\theta-1})} \leq \abs{\ip{\vec{x},\vec{y}^n}}+\varepsilon$.
As $\mc{X}_j^*$ is Ces\`aro bounded by Lemma~\ref{lemma:dualstructure}, we have
\begin{equation*}
\nrm{\vec{y}^n}_{\mf{S}_0^*(\ee^{-\theta})\cap \mf{S}_1^*(\ee^{1-\theta})} \leq \nrm{\vec{x}^*}_{\mf{S}_0^*(\ee^{-\theta})\cap \mf{S}_1^*(\ee^{1-\theta})} = 1.
\end{equation*}
The observation that $\vec{y}^n \in c_{00}(\Z;X_0^* \cap X_1^*)$ finishes the proof.
\end{proof}

We can now prove the announced proposition.

\begin{proposition}\label{proposition:dualitynorming}
  Let $(\mc{X}_0,\mc{X}_1)$ be a compatible couple of reflection invariant Ces\`aro convergent sequentially structured Banach spaces and let $\theta \in (0,1)$. Suppose that $X_0\cap X_1$ is dense in $X_0$ and $X_1$. Then for $x \in (\mc{X}_0,\mc{X}_1)_\theta$  we have
  \begin{equation}\label{eq:proposition:dualitynorming;1}
  \nrm{x}_{(\mc{X}_0,\mc{X}_1)_\theta} \eqsim_\theta \sup\, \cbraceb{ \abs{\ip{x,x^*}} : x^* \in X_0^*\cap X_1^*,\, \nrm{x^*}_{(\mc{X}_0^*,\mc{X}_1^*)_\theta} \leq 1}.
  \end{equation}
In particular, setting $\mc{Y}_j:= (\mc{X}_j^*)^\circ$ for $j=0,1$, we have
\begin{equation}\label{eq:proposition:dualitynorming;2}
\nrm{x}_{(\mc{X}_0,\mc{X}_1)_\theta}
\eqsim_\theta \sup\, \cbraceb{ \abs{\ip{x,x^*}} : x^* \in X_0^*\cap X_1^*,\, \nrm{x^*}_{(\mc{Y}_0,\mc{Y}_1)_\theta} \leq 1}.
\end{equation}
\end{proposition}

\begin{proof}
Using Proposition~\ref{proposition:Xcirc} and Lemma~\ref{lemma:dualstructure}, we know that the right-hand sides of \eqref{eq:proposition:dualitynorming;1} and \eqref{eq:proposition:dualitynorming;2} are equivalent.
Note that the inequality ``$\gtrsim$'' in \eqref{eq:proposition:dualitynorming;1} follows directly from Proposition \ref{proposition:duality}. For
the inequality ``$\lesssim$'' in \eqref{eq:proposition:dualitynorming;1} fix $x \in (\mc{X}_0,\mc{X}_1)_\theta$.
By Theorem~\ref{theorem:meanmethod} and the reflection invariance of the sequence structures, we have $(\ldots,x,x,x,\ldots) \in \mf{S}_0(\ee^{\theta}) + \mf{S}_1(\ee^{\theta-1}) $ with $$\nrm{x}_{(\mc{X}_0,\mc{X}_1)_\theta} \lesssim_\theta \nrm{(\ldots,x,x,x,\ldots) }_{\mf{S}_0(\ee^{\theta}) + \mf{S}_1(\ee^{\theta-1})}.$$
Combing this with Lemma~\ref{lem:proposition:dualitynorming}, we see that there exists $\vec{x}^* \in c_{00}(\Z;X_0^* \cap X_1^*)$ of norm $1$ such that $\nrm{x}_{(\mc{X}_0,\mc{X}_1)_\theta} \lesssim_\theta \abs{\ip{(\ldots,x,x,x,\ldots),\vec{x}^*}}$.
Now note that $$\ip{(\ldots,x,x,x,\ldots),\vec{x}^*} = \ip{x,S^*\vec{x}^*}$$ in the notation of the proof of Proposition~\ref{proposition:duality}.
Therefore, $$x^*:= S\vec{x}^* = \sum_{k \in \Z}x^*_k \in X_0^* \cap X_1^*$$ satisfies $\nrm{x}_{(\mc{X}_0,\mc{X}_1)_\theta} \lesssim_\theta \abs{\ip{x,x^*}}$ and $$\nrm{x^*}_{(\mc{X}_0^*,\mc{X}_1^*)_\theta} \leq \nrm{\vec{x}^*}_{(\mf{S}_0^*)^{\circ}(\ee^{-\theta}) \cap (\mf{S}_1^*)^{\circ}(\ee^{1-\theta})}=1.$$
This proves the inequality ``$\lesssim$'' in \eqref{eq:proposition:dualitynorming;1}, as desired.
\end{proof}

Concerning Example~\ref{example:dual}\ref{it:dualcomplex} on the duality for the complex interpolation method,  Proposition~\ref{proposition:dualitynorming} allows one to avoid the upper complex method in certain duality arguments without assuming RNP. We record this observation for future reference.

\begin{example}\label{ex:proposition:dualitynorming;complex}
Let $(X_0,X_1)$ be a compatible couple of Banach spaces and let $\theta \in (0,1)$. Suppose that $X_0\cap X_1$ is dense in $X_0$ and $X_1$. Then for $x \in \brac{X_0,X_1}_\theta$ we have
\begin{equation*}
  \nrm{x}_{\brac{X_0,X_1}_\theta} \eqsim_\theta \sup\, \cbraceb{ \abs{\ip{x,x^*}} : x^* \in X_0^*\cap X_1^*,\, \nrm{x^*}_{\brac{X_0^*,X_1^*}_\theta} \leq 1}.
  \end{equation*}
\end{example}

\subsection{Embeddings}
We can obtain embeddings between various interpolation methods from embeddings between the corresponding sequence structures, using the following proposition.

\begin{proposition}\label{proposition:embeddings}
  Let $(\mc{X}_0,\mc{X}_1)$ and $(\mc{Y}_0,\mc{Y}_1)$ be compatible couples of sequentially structured Banach spaces. Suppose that $X_0+X_1 \subseteq Y_0+Y_1$ and $\mf{S}_j \hookrightarrow \mf{T}_j$ for $j=0,1$. Then $ (\mc{X}_0,\mc{X}_1)_\theta \hookrightarrow (\mc{Y}_0,\mc{Y}_1)_\theta$ for $\theta \in (0,1)$.
\end{proposition}

Proposition \ref{proposition:embeddings} follows directly from the definition of our sequentially structured interpolation method. In view of \eqref{eq:ell_1_infty_sandwich} and Example~\ref{example:interpolationmethods}\ref{it:ssireal}, it has the following direct consequence for embeddings for the real interpolation method:
Let $(\mc{X}_0,\mc{X}_1)$ be a compatible couple of sequentially structured Banach spaces and let $\theta \in (0,1)$. Then we have
\begin{equation}\label{eq:proposition:embeddings;embd_real_int}
(X_0,X_1)_{\theta,1} \hookrightarrow  (\mc{X}_0,\mc{X}_1)_\theta \hookrightarrow (X_0,X_1)_{\theta,\infty}.
\end{equation}
Moreover, it directly implies embeddings like $(X_0,X_1)_{\theta,p} \hookrightarrow (X_0,X_1)_{\theta,q}$
for $1 \leq p \leq q \leq \infty$ by the corresponding embeddings for the sequence spaces $\ell^p(\Z;X_j)$ for $j=0,1$.

Let us next apply Proposition \ref{proposition:embeddings} to combinations of the interpolation methods introduced in Example \ref{example:interpolationmethods}.
In the following example we will again use some Banach space geometry. For an introduction to type and cotype we refer to \cite[Chapter 7]{HNVW17}, for an introduction to Fourier type to \cite{GKKKT98} and for $p$-convexity and $p$-concavity to \cite[Section 1.d]{LT79}.

\begin{example}\label{example:embeddings}
  Let $(X_0,X_1)$ be a compatible couple of Banach spaces and let $\theta \in (0,1)$.
  \begin{enumerate}[(i)]
      \item\label{it:embedding6}  Suppose $X_j$ has Fourier type $p_j \in [1,2]$ for $j=0,1$. Then we have
    \begin{equation*}
      (X_0,X_1)_{\theta, p_0,p_1} \hookrightarrow [X_0,X_1]_\theta \hookrightarrow (X_0,X_1)_{\theta, p_0',p_1'}.
    \end{equation*}
        \item\label{it:embedding4}  Suppose $X_j$ has type $p_j \in [1,2]$ and cotype $q_j \in [2,\infty]$ for $j=0,1$. Then we have
    \begin{equation*}
      (X_0,X_1)_{\theta, p_0,p_1} \hookrightarrow (X_0,X_1)_{\theta,\gamma} \hookrightarrow (X_0,X_1)_{\theta, q_0,q_1}.
    \end{equation*}
    \item\label{it:embedding5} If $X_0$ and $X_1$ have type $2$, then
    \begin{equation*}
      [X_0,X_1]_{\theta} \hookrightarrow (X_0,X_1)_{\theta,\gamma}
    \end{equation*}
    and if $X_0$ and $X_1$ have cotype $2$, then
    \begin{equation*}
      (X_0,X_1)_{\theta,\gamma} \hookrightarrow [X_0,X_1]_{\theta}.
    \end{equation*}
     \item \label{it:latticeembeddingconvexconcave} Suppose $X_0$ and $X_1$ are Banach lattices and let $p_0,p_1 \in [1,\infty]$. If $X_j$ is  $p_j$-convex for $j=0,1$, then
        \begin{equation*}
      (X_0,X_1)_{\theta,p_0,p_1} \hookrightarrow (X_0,X_1)_{\theta,\ell^{p_0},\ell^{p_1}}
    \end{equation*}
    and if $X_j$ is  $p_j$-concave for $j=0,1$, then
    \begin{equation*}
     (X_0,X_1)_{\theta,\ell^{p_0},\ell^{p_1}} \hookrightarrow (X_0,X_1)_{\theta,p_0,p_1}.
    \end{equation*}
  \end{enumerate}
\end{example}

 We note that the embeddings for the $\ell^q$-interpolation method in \ref{it:latticeembeddingconvexconcave} are new, whereas the embeddings for the $\gamma$-interpolation method in \ref{it:embedding4}  and \ref{it:embedding5} can be found in \cite[Section 3.4]{KLW19} and the embedding between real and complex interpolation in \ref{it:embedding6} was first proven in \cite{Pe69}.
 All embeddings will follow in a unified way from Example \ref{example:interpolationmethods}, Proposition \ref{proposition:embeddings} and suitable embeddings of sequence structures.

  \begin{proof}[Proof of Example \ref{example:embeddings}]
     The embeddings in \ref{it:embedding6} follow directly from the equivalence of Fourier type and periodic Fourier type (see \cite[Theorem 6.6]{GKKKT98}) and
   \ref{it:embedding4} is a consequence of the equivalence of (co)type and Gaussian (co)type (see \cite[Proposition 7.1.18 and Corollary 7.2.11]{HNVW17}).

  For \ref{it:embedding5} take $\vec{x} \in \widehat{L}^2(\T;X)$ and let $f \in L^2(\T;X)$ be such that $\widehat{f} = \vec{x}$. Since
  \begin{equation*}
    h \mapsto \has{\frac{1}{2 \pi} \int_{\T} h(t) \ee^{-ikt}\dd t}_{k \in \Z}
  \end{equation*}
  is an isometry from $L^2(\T)$ to $\ell^2(\Z)$, we have by \cite[Theorem 9.1.10 and 9.2.10]{HNVW17}
  \begin{equation*}
    \nrm{\vec{x}}_{\gamma(\Z;X)} \leq \frac{1}{(2\pi)^{1/2}} \nrm{f}_{\gamma(\T;X)} \lesssim  \frac{1}{(2\pi)^{1/2}} \nrm{f}_{L^2(\T;X)} =   \nrm{\vec{x}}_{\widehat{L}^2(\T;X)}.
  \end{equation*}
  The second embedding  in \ref{it:embedding5} can be proven similarly. Finally the embeddings in \ref{it:latticeembeddingconvexconcave} follow directly from the definition of $p_j$-convexity.
  \end{proof}

 \subsection{Interpolation of Banach function spaces}
 In this subsection we will study the interpolation spaces $(\mc{X}_0,\mc{X}_1)_\theta$ in the case that $X_0$ and $X_1$ are Banach \emph{function} spaces.

 Let $(S,\Sigma,\mu)$ be a measure space. A Banach space $X$ consisting of measurable functions $x \colon S \to \C$ is called a \emph{Banach function space} if it satisfies
\begin{enumerate}[(i)]
\item If $x \in X$ and $y\colon S \to \C$ is measurable with $\abs{y}\leq \abs{x}$, then $y \in X$ with $\nrm{y}_X\leq \nrm{x}_X$.
  \item  There is an $x \in X$ with $x > 0$ a.e.
  \item If $0\leq x_n \uparrow x$ for $(x_n)_{n=1}^\infty$ in $X$ and $\sup_{n\in \N}\nrm{x_n}_X<\infty$, then $x \in X$ and $\nrm{x}_X=\sup_{n\in\N}\nrm{x_n}_X$.
\end{enumerate}
We say that $X$ is \emph{order-continuous} if for any $0 \leq x_n \uparrow x$ with $(x_n)_{n=1}^\infty$ in $X$ and $x \in X$  we have $\nrm{x_n-x}_X \to 0$. For an introduction to Banach function spaces we refer to \cite{LN23,Za67}.

For two Banach function spaces $X_0$ and $X_1$ over the same measure space $(S,\mu)$ and $\theta \in (0,1)$ we can define the \emph{Calder\'on--Lozanovskii} product (see \cite{Ca64,Lo69}) as
\begin{equation*}
  X_0^{1-\theta}X_1^\theta:= \cbraceb{x \in L^0(S): \abs{x} \leq \abs{x_0}^{1-\theta} \abs{x_1}^\theta, \, x_0 \in X_0,\, x_1 \in X_1},
\end{equation*}
 which we equip with the norm
\begin{equation*}
  \nrm{x}_{X_0^{1-\theta}X_1^\theta}:= \inf_{\abs{x} \leq \abs{x_0}^{1-\theta}\abs{x_1}^\theta} \nrm{x_0}_{X_0}^{1-\theta} \nrm{x_1}_{X_1}^\theta.
\end{equation*}
 Then $ X_0^{1-\theta}X_1^\theta$ is a Banach function space with
 \begin{equation*}
   X_0\cap X_1 \hookrightarrow X_0^{1-\theta}X_1^\theta \hookrightarrow X_0+X_1,
 \end{equation*}
 i.e.\ $X_0^{1-\theta}X_1^\theta$ is an intermediate space with respect to $(X_0,X_1)$.

 Since a Banach function space $X$ is in particular a Banach lattice, it admits the sequence structure $X(\ell^q(\Z))$ introduced in Example~\ref{example:sequencestructures}\ref{it:sslattice}. Thus  we can apply the $\ell^q$-interpolation method from Example \ref{example:interpolationmethods}\ref{it:ssilattice} to a couple of Banach function spaces $(X_0,X_1)$ over the same measure space $(S,\Sigma, \mu)$. It turns out that the resulting intermediate spaces are just complex interpolation spaces.

 \begin{example}\label{example:BFSembeddings} Let $(X_0,X_1)$ be a compatible couple of Banach function spaces over a measure space $(S,\Sigma,\mu)$ and let $\theta \in (0,1)$.
For $q_0,q_1 \in [1,\infty)$ we have
   \begin{align*}
     [X_0,X_1]_\theta &= (X_0,X_1)_{\theta,\ell^{q_0},\ell^{q_1}} =  (X_0,X_1)_{\theta,c_0,c_0}.
      \end{align*}
 \end{example}

\begin{proof}
We start by considering the case $q_0=q_1=1$. By a Poisson integral argument (see \cite[Section 13.6]{Ca64}), we have $[X_0,X_1]_\theta \hookrightarrow (X_0^{1-\theta}X_1^\theta)^\circ$. We will prove that $(X_0^{1-\theta}X_1^\theta)^\circ \hookrightarrow (X_0,X_1)_{\theta,\ell^{1},\ell^{1}} $. Take $x \in X_0\cap X_1$ and assume without loss of generality $x \geq 0$. Let $x_0 \in X_0$ and $x_1 \in X_1$ such that $x = \abs{x_0}^{1-\theta}\abs{x_1}^\theta$ and $\nrm{x_0}_{X_0}^{1-\theta} \nrm{x_1}_{X_1}^{\theta}\leq 2 \, \nrm{x}_{(X_0^{1-\theta}X_1^\theta)^\circ}$. Define
  \begin{align*}
  E&:=\cbrace{x\in S: x(s) \neq 0}\\
    E_k &:= \cbraceb{s \in E: \ee^{k} \leq \frac{\abs{x_1}}{\abs{x_0}} \leq \ee^{k+1}}, \qquad k \in \Z.
  \end{align*}
  Take $n \in \N$ such that $\nrm{x}_{X_0} \leq \ee^{n\theta}\nrm{x}_{X_0}$ and $\nrm{x}_{X_1} \leq \ee^{n(1-\theta)}\nrm{x}_{X_1}$. Define
  \begin{equation*}
    y_k:= \begin{cases}
      x\ind_{E_k}, &\text{if $\abs{k} \leq n$}\\
      x\ind_{\bigcup_{m=n+1}^{\infty} E_m}, &\text{if $k=n+1$}\\
      x\ind_{\bigcup_{m=-n-1}^{-\infty} E_m}, &\text{if $k =-n-1$}\\
      0, &\text{otherwise}.
    \end{cases}
  \end{equation*}
  Then $(\ee^{k(j-\theta)}y_k)_{k\in \Z}\in X_j(\ell^{1}(\Z))$ for $j=0,1$ and $\sum_{k \in \Z} y_k = x$. Moreover, we have
  \begin{align*}
    \nrm{(\ee^{-k\theta} y_k)_{k \in \Z}}_{X_0(\ell^{1}(\Z))} &\leq \nrms{\sum_{k =-n}^n \ee^{-k\theta} \abs{x_0}^{1-\theta}\abs{x_1}^\theta \ind_{E_k}}_{X_0} \\&\hspace{1cm}+ \nrm{\ee^{-(n+1)\theta}x\ind_{\bigcup_{m=n+1}^\infty E_m}}_{X_0} \\&\hspace{1cm}+ \nrms{\sum_{m=-n-1}^{-\infty} \ee^{(n+1)\theta} \abs{x_0}^{1-\theta}\abs{x_1}^\theta \ind_{E_m}}_{X_0}
    \\ &\leq \ee^\theta\,
    \nrms{\sum_{k =-n}^n   \abs{x_0}  \ind_{E_k}}_{X_0} +\ee^{-(n+1)\theta} \nrm{x}_{X_0} \\&\hspace{1cm}+  \ee^\theta\,\nrms{\sum_{m=-n-1}^{-\infty} \ee^{(m+n)\theta}  \abs{x_0}  \ind_{E_m}}_{X_0}\\
    &\leq \hab{\ee^\theta + \ee^{-\theta} + \frac{\ee^\theta}{\ee^\theta-1}} \,\nrm{x}_{X_0}.
  \end{align*}
Combined with a similar estimate for $(\ee^{k(1-\theta)}y_k)_{k\in \Z}$ in $X_1(\ell^{1}(\Z))$  and Lemma \ref{lemma:logconvex}, we obtain
\begin{equation*}
  \nrm{x}_{(X_0,X_1)_{\theta,\ell^{1},\ell^{1}}} \lesssim_\theta \nrm{x}_{(X_0^{1-\theta}X_1^\theta)^\circ}.
\end{equation*}
This proves the embedding by the density of $X_0\cap X_1$ in $(X_0^{1-\theta}X_1^\theta)^\circ$, finishing the proof of the embedding $[X_0,X_1]_\theta \hookrightarrow (X_0,X_1)_{\theta,\ell^{1},\ell^{1}}$

For the converse embedding
take $\vec{x} \in \ell^{1}(\Z;X_j)$ and note that
  \begin{equation*}
    f(t) := \sum_{k\in \Z}  \ee^{ik t}x_k, \qquad t \in \T,
  \end{equation*}
  converges by assumption. Therefore
  \begin{equation*}
    \nrm{\vec{x}}_{\widehat{L}^1(\T;X_j)} = \frac{1}{2 \pi} \nrm{f}_{L^1(\T;X_j)} \leq\sup_{t \in \T}\, \nrmb{\sum_{k\in \Z} \ee^{ik t}x_k}_{X_j} \leq \nrm{\vec{x}}_{\ell^1(\Z;X_j)},
  \end{equation*}
  i.e. $\ell^{1}(\Z;X_j)\hookrightarrow \widehat{L}^1(\T;X_j)$ for $j=0,1$. This proves the embedding $
  (X_0,X_1)_{\theta,\ell^{1},\ell^{1}}  \hookrightarrow [X_0,X_1]_\theta
$ by Proposition \ref{proposition:Banachembeddings}, concluding the proof in the case $q_0=q_1=1$.

For the remaining cases, we note that, by Proposition \ref{proposition:embeddings} and the embeddings
$$\ell^{1}(\Z;X_j)\hookrightarrow \ell^{q_j}(\Z;X_j) \hookrightarrow c_0(\Z;X_j),$$
it suffices to show $(X_0,X_1)_{\theta,c_0,c_0} \hookrightarrow [X_0,X_1]_\theta$.
By Proposition \ref{proposition:Xcirc} we may without loss of generality assume that $X_0\cap X_1$ is dense in $X_0$ and $X_1$.
Fix $x \in X_0 \cap X_1$, so in particular $x \in [X_0,X_1]_\theta$. By Example~\ref{ex:proposition:dualitynorming;complex} there  exists $x^* \in [X_0^*,X_1^*]_\theta$ with norm $1$ such that
\begin{equation*}
  \nrm{x}_{[X_0,X_1]_\theta} \lesssim_\theta \abs{\ip{x,x^*}}.
\end{equation*}
Then, using embedding $[X_0,X_1]_\theta \hookrightarrow (X_0,X_1)_{\theta, \ell^1,\ell^1}\hookrightarrow (X_0,X_1)_{\theta, \ell^1_+,\ell^1_+}$  from the first part of this proof,
we obtain by Proposition \ref{proposition:duality}
\begin{align*}
  \nrm{x}_{[X_0,X_1]_\theta} \lesssim \abs{\ip{x,x^*}} &\leq \nrm{x}_{(X_0,X_1)_{{\theta,c_0,c_0}} } \nrm{x^*}_{(X_0^*,X^*_1)_{\theta, \ell^1_+,\ell^1_+}}\\
  &\lesssim_\theta \nrm{x}_{(X_0,X_1)_{{\theta,c_0,c_0}}}.
\end{align*}
Since $X_0 \cap X_1$ is dense in $(X_0,X_1)_{{\theta,c_0,c_0}}$ by Corollary~\ref{corollary:denseX0X1},
this finishes the proof.
\end{proof}

\begin{remark}
Example \ref{example:BFSembeddings} also yields that, if $X_0$ and $X_1$ have finite cotype, we have
\begin{align*}
      [X_0,X_1]_\theta &= (X_0,X_1)_{\theta,\varepsilon} = (X_0,X_1)_{\theta,\gamma},
   \end{align*}
which was already observed in \cite[Section 3.4]{KLW19}.
Indeed, this follows directly from the comparability of Rademacher, Gaussian and $\ell^2$-sums (see \cite[Corollary 7.2.10 and Theorem 7.2.13]{HNVW17}) and Proposition \ref{proposition:Banachembeddings}.

Moreover, combined with Example \ref{example:embeddings}\ref{it:latticeembeddingconvexconcave}, we see that for $p_0,p_1 \in [1,\infty]$ we have
        \begin{equation*}
      (X_0,X_1)_{\theta,p_0,p_1} \hookrightarrow [X_0,X_1]_{\theta}
    \end{equation*}
    if $X_j$ is $p_j$-convex for $j=0,1$ and
    \begin{equation*}
     [X_0,X_1]_{\theta} \hookrightarrow (X_0,X_1)_{\theta,p_0,p_1}.
    \end{equation*}
if $X_j$ is $p_j$-concave for $j=0,1$.
\end{remark}

When $E$ is a Banach function space and $[X,\mf{S}]$ is a sequentially structured Banach space, the pair
$[E(X),E(\mf{S})]$ is a sequentially structured Banach space as well. In the following proposition we will relate the associated sequentially structured interpolation spaces.

\begin{proposition}\label{proposition:BFSsequenceinterpolation}
Let $E$ be an order-continuous Banach function space and let $(\mc{X}_0,\mc{X}_1)$ be a compatible couple of Ces\`aro bounded sequentially structured Banach spaces. Assume that either $\mf{S}_0$ or $\mf{S}_1$ is a Ces\`aro convergent. Define
\begin{equation*}
  \mc{E}(\mc{X}_j) := [E(X_j), E(\mf{S}_j)], \qquad j = 0,1.
\end{equation*}
 Then we have for $\theta \in (0,1)$
\begin{equation*}
  \hab{\mc{E}(\mc{X}_0),\mc{E}(\mc{X}_1)}_\theta = E((\mc{X}_0,\mc{X}_1)_\theta).
\end{equation*}
\end{proposition}

\begin{proof}
Let $(S,\Sigma,\mu)$ be the measure space on which $E$ is defined. We start by showing
\begin{equation*}
\hab{\mc{E}(\mc{X}_0),\mc{E}(\mc{X}_1)}_\theta \hookrightarrow E((\mc{X}_0,\mc{X}_1)_\theta).
\end{equation*}
Take $x \in \hab{\mc{E}(\mc{X}_0),\mc{E}(\mc{X}_1)}_\theta$ and let $\vec{x}^j \in E(\mf{S}_j)(\ee^{j-\theta})$ for $j=0,1$ be such that $x_k^0+x_k^1 = x$ for all $k \in \Z$. Then we obtain by Theorem \ref{theorem:meanmethod} that
\begin{align*}
  \nrm{x}_{E((\mc{X}_0,\mc{X}_1)_\theta)}
  &\lesssim_\theta \nrmb{s \mapsto \sum_{j = 0,1}\, \nrm{\vec{x}^j(s)}_{\mf{S}_j(\ee^{j-\theta})}}_E\\
  &\leq \sum_{j = 0,1} \, \nrm{\vec{x}^j}_{E(\mf{S}_j)(\ee^{j-\theta})}.
\end{align*}
So, taking the infimum over all decompositions of $x$, the embedding follows by another application of Theorem \ref{theorem:meanmethod}.

For the converse embedding, note that, by the order-continuity of $E$, either $E(\mf{S}_0)$ or $E(\mf{S}_1)$ is Ces\`aro convergent, so
$E(X_0) \cap E(X_1)$ is dense in $\ha{\mc{E}(\mc{X}_0),\mc{E}(\mc{X}_1)}_\theta$ by Corollary \ref{corollary:denseX0X1}.
Since any $f \in E(X_0) \cap E(X_1)$ is strongly measurable as an $X_0\cap X_1$-valued function by the Pettis measurability theorem, we have, using the order-continuity of $E$ and \cite[Corollary 1.1.21]{HNVW16}, that the simple functions $f\colon S \to X_0\cap X_1$ are dense in $\ha{\mc{E}(\mc{X}_0),\mc{E}(\mc{X}_1)}_\theta$. Fix such an $f$ and write
$$f(s) = \sum_{n=1}^N \ind_{A_n}(s) x_n, \qquad s \in S,$$
 with $A_1,\cdots,A_N \in \Sigma$ pairwise disjoint and $x_1,\cdots,x_N \in X_0\cap X_1$. For $1\leq n \leq N$ let $\vec{x}^n \in \mf{S}_0(\ee^{-\theta})\cap \mf{S}_1(\ee^{1-\theta})$ be such that $\sum_{k \in \Z}x^n_{k}= x_n$ and
\begin{equation*}
  \nrmb{\vec{x}^n}_{\mf{S}_j(\ee^{j-\theta})}\leq 2 \, \nrm{x_n}_{\ha{\mc{X}_0,\mc{X}_1}_{\theta}}, \qquad j = 0,1.
\end{equation*}
Define $f_k(s) := \sum_{n=1}^N \ind_{A_n}x_{n}^k$ for $k \in \Z$. Then we have
\begin{align*}
 \nrm{f}_{\ha{\mc{E}(\mc{Y}_0),\mc{E}(\mc{Y}_1)}_{\theta}} &\leq \max_{j = 0,1} \, \nrmb{(f_k)_{k \in \Z}}_{E(\mf{S}_j)(\ee^{j-\theta})}\\
 &=\max_{j = 0,1} \,
 \nrms{\sum_{n=1}^N \ind_{A_n} \cdot  \nrm{\vec{x}^n}_{\mf{S}_j(\ee^{j-\theta})}}_{E}\\
 &\leq 2\, \nrms{\sum_{n=1}^N \ind_{A_n} \cdot \nrm{x_n}_{\ha{\mc{X}_0,\mc{X}_1}_{\theta}} }_{E} = 2\, \nrm{f}_{{E}\ha{\ha{\mc{X}_0,\mc{X}_1}_{\theta}}}.
\end{align*}
This implies $E((\mc{X}_0,\mc{X}_1)_\theta) \hookrightarrow \ha{\mc{E}(\mc{X}_0),\mc{E}(\mc{X}_1)}_\theta$ by density.
\end{proof}

\subsection{Changing the base number}\label{subsec:change_base_number}
The choice of the base number $\ee$ in the sequentially structured interpolation method is rather arbitrary. All of the theory up to this point could equally well have been established for $b \in (1,\infty)$.
However,  the version of the sequentially structured interpolation method with base $b=\ee$ is sufficient for most of our purposes. For simplicity of exposition we have therefore chosen to only treat that case.

An important exception is  our discussion of reiteration in Section \ref{section:reiteration}, for which we cannot avoid using the sequentially structured interpolation for a general base number $b$. This motivates us to discuss the more general base $b\in (1,\infty)$ in this subsection.

\begin{definition}\label{definition:sequenceinterpolation;basis}
Let $(\mc{X}_0,\mc{X}_1)$ be a compatible couple of sequentially structured Banach spaces, let $b \in (1,\infty)$ and let $\theta \in (0,1)$.
For $x \in X_0+X_1$ we set
\begin{align*}
  \nrm{x}_{(\mc{X}_0,\mc{X}_1)_{\theta;b}}
  & := \inf \,\nrm{\vec{x}}_{\mf{S}_0(b^{-\theta})\cap \mf{S}_1(b^{1-\theta})}
\end{align*}
where the infimum is taken over all  sequences $\vec{x} \in \mf{S}_0(b^{-\theta}) \cap \mf{S}_1(b^{1-\theta})$ such that $\sum_{k \in \Z} x_k = x$ with convergence in $X_0+X_1$.
We define
\begin{equation*}
  (\mc{X}_0,\mc{X}_1)_{\theta;b} := \cbraceb{x \in X_0+X_1:\nrm{x}_{(\mc{X}_0,\mc{X}_1)_{\theta;b}}< \infty},
\end{equation*}
with norm $\nrm{\,\cdot\,}_{(\mc{X}_0,\mc{X}_1)_{\theta;b}}$.
\end{definition}

In  the next proposition we will show that $(\mc{X}_0,\mc{X}_1)_{\theta;b}$ is independent of $b$ under suitable assumptions on $\mc{X}_0$ and $\mc{X}_1$. Loosely speaking these assumptions will be stability under multipliers $\vec{a} \in \ell^\infty(\Z)$ and stability under inserting or removing zeros in a sequence $\vec{x}$. In order to formulate this second assumption, it will be convenient to introduce the following notation.
We denote by $\pi(\Z)$ the set of all mappings $\sigma:\Z \to \Z \cup \{*\}$ with the property that $\#\sigma^{-1}(\cbrace{n}) \leq 1$ for all $n \in \Z$.
Given a Banach space $X$ and $\sigma \in \pi(\Z)$,
we define the linear operator $T_\sigma:\ell^0(\Z;X) \to \ell^0(\Z;X)$ by
\begin{equation*}
T_\sigma\vec{x}(k) :=  \begin{cases}
0, & \sigma(k)=*,\\
x_{\sigma(k)}, & \sigma(k) \in \Z.
\end{cases}
\end{equation*}

\begin{proposition}\label{prop:change_basis}
Let $(\mc{X}_0,\mc{X}_1)$ be a compatible couple of sequentially structured Banach spaces such that for $j=0,1$
\begin{align}
\label{eq:prop:change_basis;1}   \nrm{(\alpha_k x_k)_{k \in \Z}}_{\mf{S}_j} &
\leq \nrm{\vec{\alpha}}_{\ell^\infty(\Z)}\nrm{\vec{x}}_{\mf{S}_j}, && \vec{\alpha} \in \ell^\infty(\Z),\, \vec{x} \in \mf{S}_j,\\
\label{eq:prop:change_basis;2}  \nrm{T_\sigma\vec{x}}_{\mf{S}_j} &\leq \nrm{\vec{x}}_{\mf{S}_j}, && \vec{x} \in \mf{S}_j, \, \sigma \in \pi(\Z).
\end{align}
Then, for all $a,b \in (1,\infty)$ and $\theta \in (0,1)$, we have
\begin{equation*}
(\mc{X}_0,\mc{X}_1)_{\theta;a} = (\mc{X}_0,\mc{X}_1)_{\theta;b}.
\end{equation*}
\end{proposition}
\begin{proof}
Assume without loss of generality $a \neq b$ and write $b=a^\delta$ with $\delta \in (0,\infty) \setminus \{1\}$.
First assume that $\delta < 1$.
Then $h(k):=\lfloor\frac{k}{\delta}\rfloor$ defines a strictly increasing function $h:\Z \to \Z$, where $\lfloor t \rfloor$ denotes the least integer part of $t \in \R$.
We can thus define $\sigma \in \pi(\Z)$ by
$$
\sigma(n) := \begin{cases}
*, & n \notin h(\Z),\\
k, & n=h(k), k \in \Z.
\end{cases}
$$
Now let $x \in (\mc{X}_0,\mc{X}_1)_{\theta;a}$. Take $\vec{x} \in \mf{S}_0(a^{-\theta}) \cap \mf{S}_1(a^{1-\theta})$ such that $x=\sum_{k \in \Z}x_k$ in $X_0+X_1$ and set $\vec{y} :=T_\sigma\vec{x}$.
Then $x = \sum_{k \in \Z}y_k$ in $X_0+X_1$.
Note that, for $n=h(k)$ with $k \in \Z$ and $j=0,1$, we have
$$
b^{(j-\theta)n} = a^{(j-\theta)\delta n} =
a^{(j-\theta)(k-\delta\cdot \mrm{frac}(\frac{k}{\delta}))} = b^{-(j-\theta)\cdot \mrm{frac}(\frac{k}{\delta})} a^{(j-\theta)k},
$$
where $\mrm{frac}(t)=t-\lfloor t \rfloor \in [0,1)$ denotes the fractional part of $t \in \R$. We thus have
$$
(b^{(j-\theta)n}y_n)_{n \in \Z} = T_{\sigma}\big(b^{-(j-\theta)\cdot \mrm{frac}(\frac{k}{\delta})} a^{(j-\theta)k}x_k\big)_{k \in \Z},
$$
so that
\begin{align*}
\nrm{\vec{y}}_{\mf{S}_j(b^{j-\theta})}
&= \nrm{T_{\sigma}(b^{-(j-\theta)\cdot \mrm{frac}(\frac{k}{\delta})} a^{(j-\theta)k}x_k)_{k \in \Z}}_{\mf{S}_j} \\
&\stackrel{\eqref{eq:prop:change_basis;2}}{\leq} \nrm{(b^{-(j-\theta)\cdot\mrm{frac}(\frac{k}{\delta})}a^{(j-\theta)k}x_k)_{k \in \Z}}_{\mf{S}_j} \\
&\stackrel{\eqref{eq:prop:change_basis;1}}{\lesssim}_{\theta,b} \nrm{(a^{(j-\theta)k}x_k)_{k \in \Z}}_{\mf{S}_j} = \nrm{\vec{x}}_{\mf{S}_j(a^{j-\theta})}.
\end{align*}
Therefore,
$$
\nrm{x}_{(\mc{X}_0,\mc{X}_1)_{\theta;b}} \leq \nrm{\vec{y}}_{\mf{S}_0(b^{-\theta}) \cap \mf{S}_1(b^{1-\theta})} \lesssim_{\theta,a,b} \nrm{\vec{x}}_{\mf{S}_0(a^{-\theta}) \cap \mf{S}_1(a^{1-\theta})}.
$$
Taking the infimum over all $\vec{x}$ as above we obtain that $\nrm{x}_{(\mc{X}_0,\mc{X}_1)_{\theta;b}} \lesssim_{\theta,a,b} \nrm{x}_{(\mc{X}_0,\mc{X}_1)_{\theta;a}}$.

Next we assume that $\delta > 1$. Then $g(k) := \lfloor\frac{k}{\delta}\rfloor$ defines a surjection $g:\Z \to \Z$ with
$$
k_{n} := \#g^{-1}(\cbrace{n}) \in \{\lfloor \delta \rfloor,\lfloor \delta \rfloor+1\}, \qquad n \in \Z.
$$
Let us write $\#g^{-1}(\cbrace{n}) = \{i_{n,1},\ldots,i_{n,k_n}\}$ with $i_{n,1} < \ldots < i_{n,k_n}$ and define $\sigma_m \in \pi(\Z)$ for $m=1,\ldots,\lfloor \delta \rfloor+1$ by
$$
\sigma_{m}(n) := \begin{cases}
i_{n,m}, & m \leq k_n, \\
*, &m = k_{n}+1.
\end{cases}
$$
Now let $x \in (\mc{X}_0,\mc{X}_1)_{\theta;a}$. Take $\vec{x} \in \mf{S}_0(a^{-\theta}) \cap \mf{S}_1(a^{1-\theta})$ such that $x=\sum_{k \in \Z}x_k$ in $X_0+X_1$ and set $\vec{y} :=\sum_{m=1}^{\lfloor \delta \rfloor+1}T_{\sigma_m}\vec{x}$.
Then $x = \sum_{k \in \Z}y_k$ in $X_0+X_1$. Note that for $m=1,\ldots,\lfloor \delta \rfloor+1$ and $j=0,1$
$$
b^{(j-\theta)n} = a^{(j-\theta)\delta n} = a^{(j-\theta)(i_{n,m}-\delta\cdot \mrm{frac}(\frac{i_{n,m}}{\delta}))}  = b^{-(j-\theta)\cdot \mrm{frac}(\frac{i_{n,m}}{\delta})}a^{(j-\theta)i_{n,m}}.
$$
We thus have
$$
(b^{(j-\theta)n}y_n)_{n \in \Z}=
 \sum_{m=1}^{\lfloor \delta \rfloor+1}T_{\sigma_m}\big(b^{-(j-\theta)\cdot\mrm{frac}(\frac{k}{\delta})}a^{(j-\theta)k}x_k\big)_{k \in \Z},
$$
so that
\begin{align*}
\nrm{\vec{y}}_{\mf{S}_j(b^{j-\theta})}
&\leq \sum_{m=1}^{\lfloor \delta \rfloor+1}\nrm{T_{\sigma_m}(b^{-(j-\theta)\cdot \mrm{frac} (\frac{k}{\delta})} a^{(j-\theta)k}x_k)_{k \in \Z}}_{\mf{S}_j} \\
&\stackrel{\eqref{eq:prop:change_basis;2}}{\leq} (\lfloor \delta \rfloor+1) \nrm{(b^{-(j-\theta)\cdot \mrm{frac}(\frac{k}{\delta})}a^{(j-\theta)k}x_k)_{k \in \Z}}_{\mf{S}_j} \\
&\stackrel{\eqref{eq:prop:change_basis;1}}{\lesssim}_{\theta,a,b} \nrm{(a^{(j-\theta)k}x_k)_{k \in \Z}}_{\mf{S}_j} = \nrm{\vec{x}}_{\mf{S}_j(a^{j-\theta})}.
\end{align*}
Therefore,
$$
\nrm{x}_{(\mc{X}_0,\mc{X}_1)_{\theta;b}} \leq \nrm{\vec{y}}_{\mf{S}_0(b^{-\theta}) \cap \mf{S}_1(b^{1-\theta})} \lesssim_{\theta,a,b} \nrm{\vec{x}}_{\mf{S}_0(a^{-\theta}) \cap \mf{S}_1(a^{1-\theta})}.
$$
Taking the infimum over all $\vec{x}$ as above we obtain that $\nrm{x}_{(\mc{X}_0,\mc{X}_1)_{\theta;b}} \lesssim_{\theta,a,b} \nrm{x}_{(\mc{X}_0,\mc{X}_1)_{\theta;a}}$.
\end{proof}

Proposition \ref{prop:change_basis} shows that all the concrete interpolation methods from Example~\ref{example:interpolationmethods}, with the exception of the complex methods in \ref{it:ssicomplex}, can be realized with any base number $b \in (1,\infty)$. In Example~\ref{ex:changing_base_torus_complex_int} we will see that the base number can also be changed for the complex methods.

\section{Complex formulations}\label{sec:complex_formulations}
As we noted in Example \ref{example:interpolationmethods}\ref{it:ssicomplex}, the complex interpolation methods are contained in our framework by using spaces of Fourier coefficients. However, this formulation is not the original one introduced by Calder\'on \cite{Ca64}, which is in terms of analytic functions on the strip
\begin{equation*}
  \mathbb{S}:= \cbrace{z \in \C: 0<\re(z)<1}.
\end{equation*}
One may wonder if there is a relation between our framework and the classical formulation of the complex interpolation methods. It turns out that our framework admits a formulation in terms of analytic functions on $\mbb{S}$, which yields a complex formulation of all interpolation methods in our framework. This in particular means that, from our viewpoint, there is nothing inherently real or complex about the real and complex interpolation methods.
 These interpolation methods are rather living on opposite sides
of the Fourier transform.

The idea to use analytic functions for interpolation methods beyond the complex interpolation methods is not new. Implicitly it goes back to the work of Lions and Peetre \cite[Section 1.4]{LP64} (see also \cite{Pe69}). In a more explicit form, analytic functions were used in the context of the real interpolation method by Zafran \cite{Za80} (see also \cite{CJMR89}) to study the spectrum of bounded linear operators on interpolation spaces. Moreover, analytic functions in the context of the $\gamma$-interpolation method were used by Su\'arez and Weis \cite{SW06} (see also \cite{KLW19}) to interpolate analytic operator families.
Furthermore, Cwikel, Kalton, Milman and Rochberg \cite{CKMR02} made extensive use of analytic functions to prove commutator estimates in their interpolation framework.

\bigskip

Let $X$ be a Banach space and let $\ms{H}(\mathbb{S};X)$ be the space of all analytic functions $f \colon {\mathbb{S}} \to X$. The space $\ms{H}_\pi(\mathbb{S};X)$ is the $2\pi$-periodic subspace of $\ms{H}(\mathbb{S};X)$, consisting of all $f \in \ms{H}(\mathbb{S};X)$ such that $$f(z)=f(z+2\pi i), \qquad z \in {\mathbb{S}}.$$
To $f \in \ms{H}(\mathbb{S};X)$ we associate the functions
$$
f_s(t):= f(s+it),\qquad t \in \R, \, s \in (0,1).
$$

It is sometimes more convenient to work with analytic functions on the annulus
$$
\mbb{A}:= \cbraceb{z \in \C: 1<\abs{z}<\ee},
$$
which is obtained from $\mbb{S}$ by the conformal mapping $z \mapsto \ee^z$. We let $\ms{H}(\mbb{A};X)$ be the space of all analytic functions $f \colon \mbb{A} \to X$. Given $f \in \ms{H}(\mbb{A};X)$, we note that for $g(z):= f(\ee^z)$ we have $g \in \ms{H}_{\pi}(\mbb{S};X)$.
Furthermore, for $s \in (0,1)$ and $k \in \Z$, we can compute the $k$-th Fourier coefficient of $g_s$ as follows
\begin{align*}
  \ee^{-ks} \widehat{g}_s(k) &= \ee^{-ks}  \frac{1}{2\pi}\int_{\T} g(s + it)e^{-ikt}\dd t\\
   &=\frac{1}{2\pi} \int_{\T} f(\ee^{s + it})e^{-k(s-it)}\dd t =\frac{1}{2\pi i}
   \int_{\gamma_s} \frac{f(z)}{z^{k+1}} \dd z
\end{align*}
where $\gamma_s$ is the curve parameterized by $\gamma_s(t) = \ee^{s+it}$ for $t \in [0,1]$. So $\cbrace{\ee^{-ks} \widehat{g}_s(k)}_{k \in \Z}$ coincides with the Laurent series of $f$ around $0$. Since the Laurent series of $f$ around $0$ does not depend on $s$, this yields the following observation:

\begin{lemma}\label{lemma:indeps}
  Let $X$ be a Banach space. For $f \in \ms{H}_\pi(\mathbb{S};X)$ the sequence $(\ee^{-ks} \widehat{f}_s(k))_{k\in \Z} $
  is independent of $s \in (0,1)$.
\end{lemma}

Let $(\mc{X}_0,\mc{X}_1)$ be a compatible couple of sequentially structured Banach spaces and let $s \in (0,1)$. We define ${\ms{H}}_\pi({\mathbb{S}};\mc{X}_0,\mc{X}_1)$ as the space of all $f \in
{\ms{H}}_\pi({\mathbb{S}};X_0+X_1)$ such that
$$
\nrm{f}_{{\ms{H}}_\pi({\mathbb{S}};\mc{X}_0,\mc{X}_1)}:= \max_{j = 0,1}\,\nrm{\widehat{f}_s}_{\mf{S}_j(\ee^{j-s})}<\infty,
$$
which is independent of $s$ by Lemma \ref{lemma:indeps}. Similarly, we define ${\ms{H}}({\mathbb{A}};\mc{X}_0,\mc{X}_1)$ as the space of all $f \in
{\ms{H}}({\mathbb{A}};X_0+X_1)$ such that
$$
\nrm{f}_{{\ms{H}}({\mathbb{A}};\mc{X}_0,\mc{X}_1)}:= \max_{j = 0,1}\,\nrmb{\cbrace{f_k}_{k \in \Z}}_{\mf{S}_j(\ee^{j})}<\infty,
$$
where $\cbrace{f_k}_{k \in \Z}$ denotes the Laurent series of $f$ around $0$. From the discussion preceding Lemma \ref{lemma:indeps} it is clear that ${\ms{H}}_{\pi}({\mathbb{S}};\mc{X}_0,\mc{X}_1)$ and ${\ms{H}}({\mathbb{A}};\mc{X}_0,\mc{X}_1)$ are isometrically isomorphic.
The connection to the sequentially structured interpolation method will become clear in the following lemma.

\begin{lemma}\label{lemma:complexisometric}
Let $(\mc{X}_0,\mc{X}_1)$ be a compatible couple of sequentially structured Banach spaces and let $\theta \in (0,1)$. The map $f \mapsto \widehat{f}_\theta$ from ${\ms{H}}_\pi(\mathbb{S};\mc{X}_0,\mc{X}_1)$ to $\mf{S}_0(e^{-\theta}) \cap \mf{S}_1(e^{1-\theta})$ is an isometric isomorphism with inverse  $$\vec{x} \mapsto \bracb{ z \mapsto \sum_{k \in \Z}e^{k(z-\theta)}x_{k}}.$$
\end{lemma}
\begin{proof}
  The map $f \mapsto \widehat{f}_\theta$ is contractive by definition. Conversely, for $\vec{x} \in \mf{S}_0(\ee^{-\theta})\cap \mf{S}_1(\ee^{j-\theta})$  we have by \eqref{eq:sequence3}
$$\nrm{x_k}_{X_0+X_1} \leq  \min\cbrace{\ee^{k\theta}, \ee^{k(\theta-1)}}, \qquad k \in \Z,$$
so we know that
\begin{equation*}
  f(z):= \sum_{k \in \Z}  \ee^{k(z-\theta)}x_k , \qquad z \in {\mathbb{S}},
\end{equation*}
converges uniformly on any compact subset of $\mathbb{S}$ and therefore we have $f \in \ms{H}_\pi(\mathbb{S};X_0+X_1)$. Moreover, we have $ \widehat{f}_\theta= \vec{x} $ and thus  $f \in \ms{H}_\pi(\mathbb{S};\mc{X}_0,\mc{X}_1)$, finishing the proof.
\end{proof}

Using Lemma \ref{lemma:complexisometric}, our first complex formulation of the sequentially structured interpolation method now follows immediately.

\begin{proposition}\label{proposition:complexsimple}
  Let $(\mc{X}_0,\mc{X}_1)$ be a compatible couple of sequentially structured Banach spaces and let $\theta \in (0,1)$.
For $x \in X_0+X_1$ we have
\begin{align*}
  \nrm{x}_{(\mc{X}_0,\mc{X}_1)_\theta} &= \inf\,\cbraces{\nrm{f}_{{\ms{H}}_\pi({\mathbb{S}};\mc{X}_0,\mc{X}_1)}: f \in {\ms{H}}_\pi(\mathbb{S};\mc{X}_0,\mc{X}_1),\, f(\theta) =x}\\
  &= \inf\,\cbraces{\nrm{f}_{{\ms{H}}({\mathbb{A}};\mc{X}_0,\mc{X}_1)}:f \in {\ms{H}}(\mathbb{A};\mc{X}_0,\mc{X}_1),\, f(\ee^\theta) =x}.
\end{align*}
\end{proposition}

\begin{proof}
  The first equality follows directly from Lemma \ref{lemma:complexisometric} and the observation that $f(\theta)=\sum_{k \in \Z} \widehat{f}_\theta$. The second equality follows from the isometric isomorphism between ${\ms{H}}_{\pi}({\mathbb{S}};\mc{X}_0,\mc{X}_1)$ and ${\ms{H}}({\mathbb{A}};\mc{X}_0,\mc{X}_1)$.
\end{proof}

We remark that the formulation in Proposition \ref{proposition:complexsimple} in terms of functions on the annulus $\mbb{A}$ is very close in spirit to the complex formulation of the interpolation framework in \cite{CKMR02}.

\bigskip

Although Proposition \ref{proposition:complexsimple} expresses the norm of the sequentially structured interpolation method in terms of analytic functions, taking the sequence structures as in Example \ref{example:interpolationmethods}\ref{it:ssicomplex} does not yield the classical formulation of the complex interpolation methods of Calder\'on \cite{Ca64}. Indeed, the main culprit is that the norm of e.g. ${\ms{H}}_\pi({\mathbb{S}};\mc{X}_0,\mc{X}_1)$ is expressed in terms of the function $f_\theta$ for some $\theta \in (0,1)$, whereas the norm of classical complex interpolation methods is expressed in terms of boundary functions $f_j$ for $j=0,1$.

For $f \in {\ms{H}}_\pi({\mathbb{S}};\mc{X}_0,\mc{X}_1)$ it is not clear whether an extension to $\overline{\mbb{S}}$ exists and thus whether a reformulation of Proposition \ref{proposition:complexsimple} closer in spirit to the complex interpolation methods is possible. In the upcoming two subsections we will handle this extension problem in two different ways:
\begin{itemize}
  \item In Subsection \ref{subs:lowercomplex} we will restrict ${\ms{H}}_\pi({\mathbb{S}};\mc{X}_0,\mc{X}_1)$ to those functions that have a continuous extension to $\overline{\mbb{S}}$, yielding a formulation close in spirit to Calder\'on's lower complex interpolation method.
        \item In Subsection \ref{subs:uppercomplex} we will show that any function in ${\ms{H}}_\pi({\mathbb{S}};\mc{X}_0,\mc{X}_1)$ extends to $\overline{\mbb{S}}$ in a distributional sense, yielding a formulation close in spirit to Calder\'on's upper complex interpolation method.
\end{itemize}
Both approaches have advantages and disadvantages. The first approach is the least technical, but in general only yields a norm equivalence for $x \in X_0\cap X_1$. The second approach does yield a norm equivalence for $x \in X_0+X_1$, but requires us to develop some distribution theory on the torus.

\subsection{The lower complex formulation}\label{subs:lowercomplex}
We start with the formulation of the sequentially structured interpolation method in the spirit of the lower complex interpolation method. Let $\ms{H}_\pi(\overline{\mathbb{S}};X)$ be the subspace of $\ms{H}_\pi(\mathbb{S};X)$ consisting of all functions $f\in \ms{H}_\pi(\mathbb{S};X)$ which extend to a continuous function $f:\overline{\mathbb{S}} \to X$ and for $j=0,1$ define
$$
f_j(t):= f(j+it),\qquad t \in \R.
$$
Note that in the following theorem one only obtains a norm equivalence for $x \in X_0\cap X_1$.

\begin{theorem}\label{theorem:complexmethodformulation}
Let $(\mc{X}_0,\mc{X}_1)$ be a compatible couple of sequentially structured Banach spaces and let $\theta \in (0,1)$.
For $x \in X_0+X_1$ we define
\begin{equation*}
\nrm{x}_{(\mc{X}_0,\mc{X}_1)_\theta}^{(c)} := \inf\,\cbraces{\max_{j = 0,1}\,\nrm{\widehat{f}_j}_{\mf{S}_j}: f \in {\ms{H}}_\pi(\overline{\mathbb{S}};X_0+X_1),\, f(\theta) =x}.
\end{equation*}
Then
\begin{equation}\label{eq:theorem:complexmethodformulation;sum}
\nrm{x}_{(\mc{X}_0,\mc{X}_1)_\theta} \leq \nrm{x}_{(\mc{X}_0,\mc{X}_1)_\theta}^{(c)}, \qquad x \in X_0+X_1,
\end{equation}
and, if $\mf{S}_0$ and $\mf{S}_1$ are Ces\`aro bounded, we have
\begin{equation}\label{eq:theorem:complexmethodformulation;intersection}
\nrm{x}_{(\mc{X}_0,\mc{X}_1)_\theta} \gtrsim_\theta  \nrm{x}_{(\mc{X}_0,\mc{X}_1)_\theta}^{(c)}, \qquad x \in X_0 \cap X_1.
\end{equation}
\end{theorem}

\begin{proof}
Suppose that $x \in X_0+X_1$ with $\nrm{x}_{(\mc{X}_0,\mc{X}_1)_\theta}^{(c)}< \infty$ and fix an $f \in \ms{H}_\pi(\overline{\mathbb{S}};X_0+X_1)$ with $f(\theta)=x$. Taking the limit $s \to j$ in Lemma \ref{lemma:indeps}, we have
\begin{equation*}
  \widehat{f}_j(k) = \ee^{k(j-\theta)}\widehat{f}_\theta(k), \qquad k \in \Z.
\end{equation*}
Since  $\ms{H}_\pi(\overline{\mathbb{S}};X_0+X_1) \subseteq \ms{H}_\pi({\mathbb{S}};X_0+X_1)$, we deduce \eqref{eq:theorem:complexmethodformulation;sum} directly from Proposition~\ref{proposition:complexsimple}.

Next we assume that $\mf{S}_0$ and $\mf{S}_1$ are Ces\`aro bounded.
Take $x \in X_0 \cap X_1$ and choose a finitely non-zero $\vec{x} \in \mf{S}_0(\ee^{-\theta}) \cap \mf{S}_1(\ee^{1-\theta})$ such that $\sum_{k\in \Z}x_k =x$.
For
\begin{equation*}
  f(z):= \sum_{k \in \Z}  \ee^{k(z-\theta)}x_k , \qquad z \in \overline{\mathbb{S}}
\end{equation*}
we have $f \in \ms{H}_\pi(\overline{\mathbb{S}};X_0+X_1)$, $f(\theta)=x$ and $\widehat{f}_j(k) =  \ee^{k(j-\theta)}x_k$ for all $k \in \Z$.
Therefore,
\begin{equation*}
  \max_{j = 0,1}\,\nrm{\widehat{f}_j}_{\mf{S}_j} = \max_{j = 0,1}\,\nrm{\vec{x}}_{\mf{S}_j(\ee^{j-\theta})},
\end{equation*}
so taking the infimum over all such $\vec{x}$, we obtain $\nrm{x}_{(\mc{X}_0,\mc{X}_1)_\theta}^{(c)} \lesssim \nrm{x}_{(\mc{X}_0,\mc{X}_1)_\theta}$ by Lemma \ref{lemma:finiterep}.
\end{proof}

Note that the proof of Theorem~\ref{theorem:complexmethodformulation} shows that, in the inequality \eqref{eq:theorem:complexmethodformulation;intersection}, it is already enough to restrict the infimum in the definition of $\nrm{\,\cdot\,}_{(\mc{X}_0,\mc{X}_1)_\theta}^{(c)} $ to $ f \in {\ms{H}}_\pi(\overline{\mathbb{S}};X_0+X_1)$ of the form $$f(z)=\sum_{k=-N}^{N} \ee^{k(z-\theta)} x_k, \qquad z \in \overline{\mathbb{S}},$$ with $(x_k)_{k=-N}^{N} \subseteq X_0 \cap X_1$ and $N \in \N$.
Furthermore, using the full power of Lemma \ref{lemma:finiterep}, we can restrict to a dense subspace of $X_0\cap X_1$. For future reference we record this observation  in
the following corollary.

\begin{corollary}\label{cor:theorem:complexmethodformulation;tensor}
Let $(\mc{X}_0,\mc{X}_1)$ be a compatible couple of Ces\`aro bounded sequentially structured Banach spaces, let $\breve{X}$ be a dense subspace of $X_0\cap X_1$ and let $\theta \in (0,1)$.
Then we have for $x \in \breve{X}$
\begin{equation*}
 \nrm{x}_{(\mc{X}_0,\mc{X}_1)_\theta} \eqsim_\theta \inf\,\cbraces{\max_{j = 0,1}\,\nrm{\widehat{f}_j}_{\mf{S}_j}: f \in \ms{H}_{\mathrm{Trig}}(\overline{\mathbb{S}}) \otimes \breve{X},\, f(\theta) =x},
\end{equation*}
where we set
\begin{equation*}
  \ms{H}_{\mathrm{Trig}}(\overline{\mathbb{S}}):= \cbrace{z\mapsto \ee^{kz}:k \in \Z}\subseteq  \ms{H}_\pi(\overline{\mathbb{S}}).
\end{equation*}
\end{corollary}

Let $(\mc{X}_0,\mc{X}_1)$ be a compatible couple of sequentially structured Banach spaces and let $\theta \in (0,1)$.
Then
\begin{equation*}
(\mc{X}_0,\mc{X}_1)^{(c)}_\theta := \left\{ x \in X_0+X_1 : \nrm{x}_{(\mc{X}_0,\mc{X}_1)_\theta}^{(c)} < \infty \right\}
\end{equation*}
is a normed space.
Indeed, it follows readily from the definition that $\nrm{\,\cdot\,}_{(\mc{X}_0,\mc{X}_1)_\theta}^{(c)}$ is a semi-norm and it follows from  \eqref{eq:theorem:complexmethodformulation;sum} that it is also positive definite.

Theorem \ref{theorem:complexmethodformulation} raises the question whether
\begin{equation}\label{eq:complex_formulation}
(\mc{X}_0,\mc{X}_1)_\theta =(\mc{X}_0,\mc{X}_1)^{(c)}_\theta.
\end{equation}
In the case that $\mf{S}_0$ and $\mf{S}_1$ are Ces\`aro bounded and additionally either $\mf{S}_0$ or $\mf{S}_1$ is Ces\`aro convergent, we have that $X_0 \cap X_1$ is a dense subspace of $(\mc{X}_0,\mc{X}_1)_\theta$ by Corollary~\ref{corollary:denseX0X1}, so we would be able to extend \eqref{eq:theorem:complexmethodformulation;intersection} by density if $(\mc{X}_0,\mc{X}_1)^{(c)}_\theta$ defines a Banach space.

Although the above question is interesting from a theoretical perspective, Theorem~\ref{theorem:complexmethodformulation} already suffices for most applications involving Ces\`aro convergent sequence structures.
In particular, this is the case for proving the first part of
Example~\ref{example:interpolationmethods}\ref{it:ssicomplex}.
Therefore, we will first engage ourselves in the latter in the next example and come back to \eqref{eq:complex_formulation} afterwards.

\begin{example}\label{example:complexmethod}
Let $(\mc{X}_0,\mc{X}_1)$ be a compatible couple of sequentially structured Banach spaces, let $\theta \in (0,1)$ and take $p_0,p_1 \in [1,\infty]$.
If
\begin{equation*}
\mf{S}_j= \begin{cases}
\widehat{L}^{p_j}(\T;X_j)&\text{ if }p_j<\infty,\\
\widehat{C}(\T;X_j)&\text{ if }p_j=\infty,
\end{cases}
\end{equation*} for $j=0,1$, then we have $(\mc{X}_0,\mc{X}_1)_\theta  = \brac{X_0,X_1}_{\theta}$.
\end{example}

Let us mention that the above example for $p= \infty$ is implicitly contained in \cite{Ca64,Cw78}.

\begin{proof}
Let $\mc{Y}_j=[X_j,\mf{T}_j]$ be given by $\mf{T}_j := \widehat{C}(\T;X_j)$ for $j=0,1$.
By Proposition~\ref{proposition:embeddings}, we have $(\mc{Y}_0,\mc{Y}_1)_\theta \hookrightarrow (\mc{X}_0,\mc{X}_1)_\theta$, so it suffices to show that
$(\mc{X}_0,\mc{X}_1)_\theta \hookrightarrow \brac{X_0,X_1}_{\theta}$ and $\brac{X_0,X_1}_{\theta} \hookrightarrow (\mc{Y}_0,\mc{Y}_1)_\theta$.

In order to show that $(\mc{X}_0,\mc{X}_1)_\theta \hookrightarrow \brac{X_0,X_1}_{\theta}$, let $x \in X_0 \cap X_1$ and let $f \in \ms{H}_\pi(\overline{\mathbb{S}})\otimes(X_0\cap X_1)$ be such that $f(\theta) = x$. Then we have for $g(z) := \ee^{(z-\theta)^2}f(z)$ that $g \in \ms{H}(\overline{\mathbb{S}};X_0+X_1)$,  $g(\theta)=x$ and
\begin{align*}
\nrm{g_j}_{L^{p_j}(\R;X_j)}
&\leq e^{(j-\theta)^2}\sum_{k \in \Z} \ee^{-k^2\pi^2 } \nrm{f_j}_{L^{p_j}(\T;X_j)} \lesssim \nrm{\widehat{f}_j}_{\widehat{L}^{p_j}(\T;X_j)} = \nrm{\widehat{f}_j}_{\mf{S}_j}.
\end{align*}
Therefore, by \cite[Corollary C.2.11]{HNVW16}, we have
\begin{align*}
\nrm{x}_{[X_0,X_1]_\theta} &\lesssim_{\theta,p_0,p_1} \max_{j=0,1} \, \nrm{\widehat{f}_j}_{\mf{S}_j}.
\end{align*}
Taking the infimum over all such $f$, we obtain the estimate
$\nrm{x}_{[X_0,X_1]_\theta} \lesssim_{\theta,p_0,p_1} \nrm{x}_{(\mc{X}_0,\mc{X}_1)_\theta}$ by Corollary \ref{cor:theorem:complexmethodformulation;tensor}.
Since $X_0 \cap X_1$ is dense in $(\mc{X}_0,\mc{X}_1)_\theta$ by Corollary~\ref{corollary:denseX0X1},  the embedding $(\mc{X}_0,\mc{X}_1)_\theta  \hookrightarrow \brac{X_0,X_1}_{\theta}$ follows.

Next, in order to prove the embedding $\brac{X_0,X_1}_{\theta} \hookrightarrow (\mc{Y}_0,\mc{Y}_1)_\theta$, let $x \in \brac{X_0,X_1}_{\theta}$.
Take $f \in \ms{H}(\overline{\mathbb{S}};X_0+X_1)$ such that $f(\theta) = x$. Define \begin{equation}\label{eq:psiperiodic}
\psi(z) := \frac{1}{z-\theta}\hab{\ee^{{z-\theta}}-1}\cdot  \ee^{(z-\theta)^2},\qquad z \in \overline{\mathbb{S}},
\end{equation}
which is analytic on $\mathbb{S}$ and satisfies $\psi(\theta)=1$ and $\psi(\theta+2\pi i k) = 0$ for $k \in \Z\setminus \cbrace{0}$. Thus if we define
\begin{equation*}
g(z) := \sum_{n \in \Z} f(z+2\pi i n) \cdot \psi(z+2\pi i n), \qquad z \in \overline{\mathbb{S}},
\end{equation*}
we have $g \in \ms{H}_\pi(\overline{\mathbb{S}};X_0+X_1)$ and $g(\theta) = f(\theta)=x$. Moreover, for $j=0,1$ we have
\begin{align*}
\nrm{g_j}_{L^{\infty}(\T;X_j)} &\leq \nrm{f_j}_{L^{\infty}(\R;X_j)} \cdot \sup_{t \in \T} \sum_{n \in \Z} \abs{\psi(j+it+2\pi i n)}\\
&\lesssim_\theta \nrm{f_j}_{L^{\infty}(\R;X_j)}
\end{align*}
Thus, using Theorem~\ref{theorem:complexmethodformulation}, we have
\begin{align*}
\nrm{x}_{(\mc{X}_0,\mc{X}_1)_\theta} &\leq \max_{j=0,1} \, \nrm{\widehat{g}_j}_{\widehat{C}(\T;X_j)} \leq  \max_{j=0,1} \, \nrm{g_j}_{{L}^{\infty}(\T;X_j)}\lesssim_\theta \max_{j=0,1} \, \nrm{f_j}_{{L}^{\infty}(\R;X_j)}.
\end{align*}
Taking the infimum over all such $f$, the embedding $\brac{X_0,X_1}_{\theta} \hookrightarrow (\mc{Y}_0,\mc{Y}_1)_\theta$ follows.
 \end{proof}

\begin{remark}\label{remark:nonperiodic}
The classical definition of the lower complex interpolation method uses non-periodic holomorphic functions on the strip $\mathbb{S}$, which in turn give rise to a function, rather than a sequence, on the Fourier side. In order to fit this into our abstract framework, one would need a continuous version of a sequence structure, which is not always possible. For the specific case of real interpolation, this was done in \cite{LL21c}.
\end{remark}

Now let us return to the question in \eqref{eq:complex_formulation}, which is mainly of theoretical interest. It boils down to the question whether $(\mc{X}_0,\mc{X}_1)^{(c)}_\theta$ is a Banach space. To answer this question, let us define some auxiliary normed spaces.
Let $(\mc{X}_0,\mc{X}_1)$ be a compatible couple of sequentially structured Banach spaces.
We define
\begin{align*}
  {\ms{H}}_\pi(\overline{\mathbb{S}};\mc{X}_0,\mc{X}_1):={\ms{H}}_\pi({\mathbb{S}};\mc{X}_0,\mc{X}_1) \cap {\ms{H}}_\pi(\overline{\mathbb{S}};X_0+X_1),
\end{align*}
equipped with the norm
\begin{equation*}
\nrm{f}_{{\ms{H}}_\pi(\overline{\mathbb{S}};\mc{X}_0,\mc{X}_1)} := \nrm{f}_{{\ms{H}}_\pi({\mathbb{S}};\mc{X}_0,\mc{X}_1)} =
\max_{j = 0,1}\,\nrm{\widehat{f}_j}_{\mf{S}_j},
\end{equation*}
where the second equality follows by taking limits in Lemma \ref{lemma:indeps}.
Furthermore, we define
\begin{equation*}
{\ms{H}}^\infty_\pi(\overline{\mathbb{S}};\mc{X}_0,\mc{X}_1)
:= {\ms{H}}_\pi(\overline{\mathbb{S}};\mc{X}_0,\mc{X}_1) \cap C_{\rm b}(\overline{\mathbb{S}};X_0+X_1)
\end{equation*}
equipped with the intersection norm.
Note that we have $${\ms{H}}^\infty_\pi(\overline{\mathbb{S}};\mc{X}_0,\mc{X}_1) = {\ms{H}}_\pi(\overline{\mathbb{S}};\mc{X}_0,\mc{X}_1)$$ as sets, but their respective norms differ. We start with two preparatory lemmata.

\begin{lemma}\label{lemma:H_infty_BS}
Let $(\mc{X}_0,\mc{X}_1)$ be a compatible couple of sequentially structured Banach spaces. Then ${\ms{H}}^\infty_\pi(\overline{\mathbb{S}};\mc{X}_0,\mc{X}_1)$ is a Banach space.
\end{lemma}
\begin{proof}
Let $(f^{n})_{n \in \N}$ be a Cauchy sequence in ${\ms{H}}^\infty_\pi(\overline{\mathbb{S}};\mc{X}_0,\mc{X}_1)$.
On the one hand, as ${\ms{H}}_\pi(\overline{\mathbb{S}};X_0+X_1)$ is a closed linear subspace of the Banach space $ C_{\rm b}(\overline{\mathbb{S}};X_0+X_1)$,  we obtain that $f=\lim_{n \to \infty}f^{n}$ in $ C_{\rm b}(\overline{\mathbb{S}};X_0+X_1)$ for some $f \in {\ms{H}}_\pi(\overline{\mathbb{S}};X_0+X_1)$.
On the other hand, as $(\widehat{f}_{j}^n)_{n \in \N}$ is a Cauchy sequence in $\mf{S}_j$, we have $\vec{x}_j = \lim_{n \to \infty}\widehat{f}_{j}^n$ in $\mf{S}_j$ for some $\vec{x}_j \in \mf{S}_j$ for $j=0,1$.
The convergence in $ C_{\rm b}(\overline{\mathbb{S}};X_0+X_1)$ implies that $f_j=\lim_{n \to \infty}f_{j}^n$ in $$ C_{\rm b}(\T;X_0+X_1) \hookrightarrow L^1(\T;X_0+X_1)$$ for $j=0,1$, so that
$$
\widehat{f}_{j}(k) = \lim_{n \to \infty}\widehat{f}_{j}^n(k), \qquad k \in \Z, j=0,1.
$$
Since $\vec{x}_j = \lim_{n \to \infty}\widehat{f}_{j}^n$ in $\mf{S}_j \hookrightarrow \ell^0(\Z;X_0+X_1)$, it follows that $\widehat{f}_{j} = \vec{x}_j$ and thus $\widehat{f}_{j} = \lim_{n \to \infty}\widehat{f}_{j}^n$ in $\mf{S}_j$ for $j=0,1$.
Therefore, $f = \lim_{n \to \infty}f^n$ in ${\ms{H}}^\infty_\pi(\overline{\mathbb{S}};\mc{X}_0,\mc{X}_1)$.
\end{proof}

\begin{lemma}\label{lem:H_infty_X_vs_Y}
Let $(\mc{X}_0,\mc{X}_1)$ be a compatible couple of sequentially structured Banach spaces. Let $\mf{T}_j$ be the sequence structure on $X_j$ given by $\mf{T}_j := \mf{S}_j \cap \widehat{C}(\T;X_0+X_1)$ and set $\mc{Y}_j:=[X_j,\mf{T}_j]$.  Then
\begin{equation*}
{\ms{H}}^\infty_\pi(\overline{\mathbb{S}};\mc{X}_0,\mc{X}_1) = {\ms{H}}_\pi(\overline{\mathbb{S}};\mc{Y}_0,\mc{Y}_1) = {\ms{H}}^\infty_\pi(\overline{\mathbb{S}};\mc{Y}_0,\mc{Y}_1)
\end{equation*}
with an equality of norms.
\end{lemma}
\begin{proof}
By the  Phragm\'en-Lindel\"of theorem, we have for every function $f \in {\ms{H}}_\pi(\overline{\mathbb{S}};X_0+X_1)$ that
$$
\nrm{f}_{L^\infty(\mathbb{S};X_0+X_1)} = \max_{j=0,1}\nrm{f_j}_{ C_{\rm b}(\R;X_0+X_1)},
$$
from which the desired result follows.
\end{proof}

We are now ready to characterize when \eqref{eq:complex_formulation} holds.

\begin{proposition}\label{prop:theorem:complexmethodformulation;beyond_intersection}
Let $(\mc{X}_0,\mc{X}_1)$ be a compatible couple of Ces\`aro bounded sequentially structured Banach spaces such that $\mf{S}_0$ or $\mf{S}_1$ is Ces\`aro convergent and let $\theta \in (0,1)$.
Let $\mf{T}_j$ be the sequence structure on $X_j$ given by
\begin{equation*}
\mf{T}_j := \mf{S}_j \cap \widehat{C}(\T;X_0+X_1)
\end{equation*}
and set $\mc{Y}_j:=[X_j,\mf{T}_j]$.
Then the following statements are equivalent.
\begin{enumerate}[(i)]
    \item\label{it:cor:prop:theorem:complexmethodformulation;beyond_intersection;c} $(\mc{X}_0,\mc{X}_1)_\theta = (\mc{X}_0,\mc{X}_1)^{(c)}_\theta$ with an equivalence of norms.

    \item\label{it:cor:prop:theorem:complexmethodformulation;beyond_intersection;BS} $(\mc{X}_0,\mc{X}_1)^{(c)}_\theta$ is a Banach space.

    \item\label{it:cor:prop:theorem:complexmethodformulation;beyond_intersection;PL} There exists a finite constant $C \geq 1$ such that, for all $f \in {\ms{H}}_\pi(\overline{\mathbb{S}};\mc{X}_0,\mc{X}_1)$ there exists $g \in {\ms{H}}^\infty_\pi(\overline{\mathbb{S}};\mc{X}_0,\mc{X}_1)$ with $g(\theta)=f(\theta)$ and
\begin{equation*}
\nrm{g}_{{\ms{H}}^\infty_\pi(\overline{\mathbb{S}};\mc{X}_0,\mc{X}_1)} \leq C\nrm{f}_{{\ms{H}}_\pi(\overline{\mathbb{S}};\mc{X}_0,\mc{X}_1)}.
\end{equation*}

    \item\label{it:cor:prop:theorem:complexmethodformulation;beyond_intersection;X_vs_Y} $(\mc{X}_0,\mc{X}_1)_\theta = (\mc{Y}_0,\mc{Y}_1)_\theta$ with an equivalence of norms.
\end{enumerate}
\end{proposition}

Before we go to the proof of Proposition~\ref{prop:theorem:complexmethodformulation;beyond_intersection}, let us note that the sequence structure $\mf{T}_j$ on $X_j$ in Proposition \ref{prop:theorem:complexmethodformulation;beyond_intersection} is not a ``natural'' sequence structure.
Loosely speaking, unless $\mf{S}_j \hookrightarrow \widehat{C}(\T;X_j)$, $\mf{T}_j$ has information of the other Banach space $X_k$, $j \neq k \in \{0,1\}$, built into its definition.
In this sense Proposition~\ref{prop:theorem:complexmethodformulation;beyond_intersection} provides many ``artificial'' examples for which \eqref{eq:complex_formulation} holds true. However, there are only a few ``natural'' examples, such as complex interpolation and real interpolation with parameter $1$.

\begin{proof}
The equivalence ``\ref{it:cor:prop:theorem:complexmethodformulation;beyond_intersection;c}$\Leftrightarrow$\ref{it:cor:prop:theorem:complexmethodformulation;beyond_intersection;BS}''
 follows from a combination of Theorem~\ref{theorem:complexmethodformulation} and Corollary~\ref{corollary:denseX0X1}.
For ``\ref{it:cor:prop:theorem:complexmethodformulation;beyond_intersection;BS}$\Leftrightarrow$\ref{it:cor:prop:theorem:complexmethodformulation;beyond_intersection;PL}'' note that
$$
(\mc{X}_0,\mc{X}_1)^{(c)}_\theta = \faktor{\ms{H}_\pi(\overline{\mathbb{S}};\mc{X}_0,\mc{X}_1)}{\big\{f \in \ms{H}_\pi(\overline{\mathbb{S}};\mc{X}_0,\mc{X}_1) : f(\theta)=0 \big\}}.
$$
Defining
$$
({\mc{X}_0,\mc{X}_1})^{(c);\infty}_\theta := \faktor{{\ms{H}}^\infty_\pi(\overline{\mathbb{S}};\mc{X}_0,\mc{X}_1)}{\big\{f \in {\ms{H}}^\infty_\pi(\overline{\mathbb{S}};\mc{X}_0,\mc{X}_1) : f(\theta)=0 \big\}},
$$
we have $(\mc{X}_0,\mc{X}_1)^{(c)}_\theta = ({\mc{X}_0,\mc{X}_1})^{(c);\infty}_\theta$ as linear spaces with $({\mc{X}_0,\mc{X}_1})^{(c);\infty}_\theta \hookrightarrow (\mc{X}_0,\mc{X}_1)^{(c)}_\theta$ thanks to the fact that ${\ms{H}}_\pi(\overline{\mathbb{S}};\mc{X}_0,\mc{X}_1) = {\ms{H}}^\infty_\pi(\overline{\mathbb{S}};\mc{X}_0,\mc{X}_1)$ as linear spaces with ${\ms{H}}^\infty_\pi(\overline{\mathbb{S}};\mc{X}_0,\mc{X}_1) \hookrightarrow {\ms{H}}_\pi(\overline{\mathbb{S}};\mc{X}_0,\mc{X}_1)$.
Recall that $(\mc{X}_0,\mc{X}_1)^{(c)}_\theta$ is a normed space.
Furthermore, as $$\big\{f \in {\ms{H}}^\infty_\pi(\overline{\mathbb{S}};\mc{X}_0,\mc{X}_1) : f(\theta)=0 \big\}$$ is a closed linear subspace of ${\ms{H}}^\infty_\pi(\overline{\mathbb{S}};\mc{X}_0,\mc{X}_1)$, we see that $({\mc{X}_0,\mc{X}_1})^{(c);\infty}_\theta$ is a Banach space in view of Lemma~\ref{lemma:H_infty_BS}.
Therefore, by the open mapping theorem, \ref{it:cor:prop:theorem:complexmethodformulation;beyond_intersection;BS} holds true if and only if $(\mc{X}_0,\mc{X}_1)^{(c)}_\theta \hookrightarrow ({\mc{X}_0,\mc{X}_1})^{(c);\infty}_\theta$, which is readily be seen to be equivalent to \ref{it:cor:prop:theorem:complexmethodformulation;beyond_intersection;PL}.

For ``\ref{it:cor:prop:theorem:complexmethodformulation;beyond_intersection;c}$\Rightarrow$\ref{it:cor:prop:theorem:complexmethodformulation;beyond_intersection;X_vs_Y}''
assume that \ref{it:cor:prop:theorem:complexmethodformulation;beyond_intersection;c} holds true.
Then, by the above, \ref{it:cor:prop:theorem:complexmethodformulation;beyond_intersection;PL} holds true as well.
So $(\mc{X}_0,\mc{X}_1)_\theta = (\mc{X}_0,\mc{X}_1)^{(c)}_\theta \hookrightarrow ({\mc{X}_0,\mc{X}_1})^{(c);\infty}_\theta$. By the  Phragm\'en-Lindel\"of theorem we note that
\begin{equation}\label{eq:phragmenlindelof}
  \nrm{f}_{L^\infty(\overline{\mathbb{S}};X_0+X_1)} = \max_{j=0,1}\nrm{f_j}_{ C_{\rm b}(\R;X_0+X_1)}, \qquad f \in {\ms{H}}_\pi(\overline{\mathbb{S}};X_0+X_1),
\end{equation}
so we have $({\mc{X}_0,\mc{X}_1})^{(c);\infty}_\theta = (\mc{Y}_0,\mc{Y}_1)_\theta^{(c)}$.
Since $(\mc{Y}_0,\mc{Y}_1)_\theta^{(c)} \hookrightarrow (\mc{Y}_0,\mc{Y}_1)_\theta$ by Theorem~\ref{theorem:complexmethodformulation} and since $(\mc{Y}_0,\mc{Y}_1)_\theta \hookrightarrow (\mc{X}_0,\mc{X}_1)_\theta$ by Proposition~\ref{proposition:embeddings}, it follows that
\begin{equation*}
(\mc{X}_0,\mc{X}_1)_\theta \hookrightarrow  (\mc{Y}_0,\mc{Y}_1)_\theta \hookrightarrow (\mc{X}_0,\mc{X}_1)_\theta.
\end{equation*}
This shows that \ref{it:cor:prop:theorem:complexmethodformulation;beyond_intersection;X_vs_Y} holds true.

Finally, for ``\ref{it:cor:prop:theorem:complexmethodformulation;beyond_intersection;X_vs_Y}$\Rightarrow$\ref{it:cor:prop:theorem:complexmethodformulation;beyond_intersection;c}''
assume that \ref{it:cor:prop:theorem:complexmethodformulation;beyond_intersection;X_vs_Y} holds true. Note that by \eqref{eq:phragmenlindelof} we know that
\ref{it:cor:prop:theorem:complexmethodformulation;beyond_intersection;PL} holds true with
$(\mc{Y}_0,\mc{Y}_1)$ in place of $(\mc{X}_0,\mc{X}_1)$.
Furthermore, note that $\mc{Y}_0$ and $\mc{Y}_1$ are Ces\`aro bounded sequentially structured Banach spaces such that either $\mf{T}_0$ or $\mf{T}_1$ is Ces\`aro convergent.
Therefore, by the equivalence of \ref{it:cor:prop:theorem:complexmethodformulation;beyond_intersection;c} and \ref{it:cor:prop:theorem:complexmethodformulation;beyond_intersection;PL} applied to $(\mc{Y}_0,\mc{Y}_1)$ instead of $(\mc{X}_0,\mc{X}_1)$, we have $(\mc{Y}_0,\mc{Y}_1)_\theta = (\mc{Y}_0,\mc{Y}_1)^{(c)}_\theta$. Since we clearly have $(\mc{Y}_0,\mc{Y}_1)^{(c)}_\theta \hookrightarrow (\mc{X}_0,\mc{X}_1)^{(c)}_\theta$ and since $(\mc{X}_0,\mc{X}_1)^{(c)}_\theta \hookrightarrow (\mc{X}_0,\mc{X}_1)_\theta$ by Theorem~\ref{theorem:complexmethodformulation}, it follows that
$$
(\mc{X}_0,\mc{X}_1)_\theta = (\mc{Y}_0,\mc{Y}_1)_\theta \hookrightarrow (\mc{X}_0,\mc{X}_1)^{(c)}_\theta \hookrightarrow (\mc{X}_0,\mc{X}_1)_\theta.
$$
Therefore, \ref{it:cor:prop:theorem:complexmethodformulation;beyond_intersection;c} holds true.
\end{proof}

\subsection{The upper complex formulation}\label{subs:uppercomplex}
We now turn to a formulation of the sequence structured interpolation method in the spirit of the upper complex method. This formulation is more general than the formulation in the spirit of the lower complex method in the previous subsection, as it allows for sequence structures which are not Ces\`aro bounded and it gives a norm equality for all $x \in (\mc{X}_0,\mc{X}_1)_\theta$, rather than just $x \in X_0\cap X_1$. The price we pay is that we need to work with distributions, rather than functions, on the boundary of $\mathbb{S}$.

We will need some distribution theory on the torus in this section. Let $X$ be a Banach space and define the following spaces of Schwartz functions:
\begin{align*}
 \ms{S}(\R;X)&:= \cbrace{f\colon \R \to X : f \text{ rapidly decreasing and smooth}},\\
 \ms{S}(\T;X)&:= \cbrace{f\colon \T \to X : f \text{ smooth}},\\
 \ms{S}(\Z;X)&:= \cbrace{f\colon \Z \to X : f \text{ rapidly decreasing}},
\end{align*}
with their natural topologies. We omit $X$ when $X=\C$. For $G \in \cbrace{\R,\T,\Z}$ we define the spaces of distributions
\begin{align*}
  \ms{S}'(G;X) &:= \mc{L}(\ms{S}(G),X).
\end{align*}
We equip these spaces of distributions the topology of pointwise convergence, which is the analogue of the usual weak$^*$ topology from the scalar-valued case $X=\C$.
However, we will also use the topology of bounded convergence.
To this end, it will be convenient to write $\ms{S}'_{\mrm{bc}}(G;X)$ for $\ms{S}'(G;X)$ equipped with the topology of bounded convergence.

Let us remark that, although the topology of bounded convergence is stronger than the topology of pointwise convergence, $\ms{S}'(G;X)$ and $\ms{S}'_{\mrm{bc}}(G;X)$ actually have the same convergent sequences since the Schwartz space $\ms{S}(G)$ is a so-called \emph{Montel space} (see e.g.\ \cite[Section~34.4]{Tr06})

We note that (using the topology of bounded convergence) we have the topological linear isomorphism (which can be seen by inspection of \cite[Remark 3.22]{BGH19})
\begin{equation}\label{eq:OMisomorphism}
  \ms{S}'_{\mrm{bc}}(\Z;X) = \mc{O}_{\mc{M}}(\Z;X),
\end{equation}
where $\mc{O}_{\mc{M}}(\Z;X)$ is the space of all $f \colon \Z \to X$ of at most polynomial growth, i.e. the space of all $f \colon \Z \to X$ such that there exist an  $N>0$ such that
\begin{equation*}
  \nrm{f(k)}_{X} \lesssim (1+\abs{k})^N, \qquad k \in \Z,
\end{equation*}
with its natural inductive limit topology. Furthermore we note that the Fourier transform and its inverse
\begin{align*}
  \ms{F}&\colon \ms{S}(\T;X) \to \ms{S}(\Z;X), &(\ms{F}f)(k) &:= \frac{1}{2\pi}\int_{\T} f(t)\ee^{ikt}\dd t,\\
  \ms{F}^{-1}&\colon \ms{S}(\Z;X) \to \ms{S}(\T;X), &(\ms{F}^{-1}g)(t) &:= \sum_{k\in \Z} \ee^{ikt} g(k),
\end{align*}
are continuous and $\ms{F}^{-1} \circ \ms{F}$ and $\ms{F} \circ \ms{F}^{-1}$ are the identity mappings on $\ms{S}(\T;X)$ and $\ms{S}(\Z;X)$ respectively (see \cite[Section 3.3]{BGH19}). On the spaces of distributions the mappings
\begin{align*}
  \ms{F}&\colon \ms{S}'(\T;X) \to \ms{S}'(\Z;X), &\ms{F}u\,(f) &:= u\left((\ms{F}^{-1}f)(-\cdot)\right),\\
  \ms{F}&\colon \ms{S}'(\Z;X) \to \ms{S}'(\T;X), &\ms{F}^{-1}u\,(g) &:= u\left((\ms{F}g)(-\cdot)\right),
\end{align*}
have similar properties.
The latter remains true when $\ms{S}'$ replaced by $\ms{S}'_{\mrm{bc}}$, that is, when we use the topology of bounded convergence instead of the topology of pointwise convergence.

There is a continuous embedding (cf.\ \cite[Theorem~16.3]{DK10})
\begin{equation}\label{eq:periodic_extension}
\iota_{\T,\R}:\ms{S}'_{\mrm{bc}}(\T;X) \hookrightarrow \ms{S}'_{\mrm{bc}}(\R;X),
\end{equation}
with the property that, for all $\mu \in \Lambda^1(\T;X) \subseteq \ms{S}'(\T;X)$, $\iota_{\T,\R}\mu$ equals the $2\pi$-periodic extension of $\mu$.
This can be seen as follows.
Note that
$$
T:\ms{S}(\R) \to \ms{S}(\T),\,
\phi \mapsto \sum_{k \in \Z}\phi(\,\cdot\,+2\pi k),
$$
is a well-defined continuous linear mapping with the property that, for every $f \in \Lambda^1(\T;X) {\subseteq} \ms{S}'(\T;X)$, $\mu \circ T$ equals the $2\pi$-periodic extension of $\mu$.
Therefore, we can define $\iota_{\T,\R}\,u:=u \circ T$.

Lemma~\ref{lemma:indeps} allows us to define ``distributional boundary values'' $f_j$ of certain $f \in \ms{H}_\pi(\mathbb{S};X)$ under a suitable condition that connects well with the setting of Proposition~\ref{proposition:complexsimple}.

\begin{lemma}\label{lem:boundary_distributions;periodic}
Let $X$ be a Banach space and let $\theta \in (0,1)$.
Then, for each $f \in \ms{H}_\pi(\mathbb{S};X)$ with
\begin{equation}\label{eq:lem:boundary_distributions;periodic;distr_assump}
(\ee^{k(j-\theta)}\widehat{f}_\theta(k))_{k \in \Z} \in \ms{S}'(\Z;X), \qquad j=0,1,
\end{equation}
there exist unique distributions $f_j \in \ms{S}'(\T;X)$ with the property that
\begin{equation}\label{eq:lem:boundary_distributions;periodic}
f_j = \lim_{s \to j}f_s \quad \text{in} \quad \ms{S}'_{\mrm{bc}}(\T;X), \qquad j=0,1.
\end{equation}
Moreover, the Fourier transform of $f_j$ is given by
\begin{equation}\label{eq:measure_distr_FT}
\widehat{f}_j(k) = \ee^{k(j-\theta)}\widehat{f}_\theta(k),\quad k \in \Z, \, j=0,1.
\end{equation}

\end{lemma}
\begin{proof}
Uniqueness follows from the uniqueness of limits in $\ms{S}'_{\mrm{bc}}(\T;X)$.
By assumption \eqref{eq:lem:boundary_distributions;periodic;distr_assump}  we can define $f_j \in \ms{S}'(\T;X)$ through \eqref{eq:measure_distr_FT}.
We only need to show the convergence in \eqref{eq:lem:boundary_distributions;periodic}.

As the Fourier transform is a topological isomorphism from $\ms{S}'_{\mrm{bc}}(\T;X)$ to $\ms{S}'_{\mrm{bc}}(\Z;X)$, \eqref{eq:lem:boundary_distributions;periodic} is equivalent to
\begin{equation*}
\widehat{f}_j = \lim_{s \to j}\widehat{f}_s \quad \text{in} \quad \ms{S}'_{\mrm{bc}}(\Z;X),
\end{equation*}
which, using \eqref{eq:OMisomorphism}, can be reformulated as the existence of $N \in \N$ such that
\begin{equation}\label{eq:lem:boundary_distributions;periodic;reformulated}
\lim_{s \to j}\sup_{k \in \Z}(1+\abs{k})^{-N}\,\nrm{\widehat{f}_j(k)-\widehat{f}_s(k)}_{X} =  0.
\end{equation}
We claim that \eqref{eq:lem:boundary_distributions;periodic;reformulated} holds true for $N=M+1$, where $M \in \N$ is such that
\begin{equation*}
C:= \sup_{k \in \Z}\, (1+\abs{k})^{-M}\max\{\ee^{-k\theta}, \ee^{k(1-\theta)}\}\nrm{\widehat{f}_\theta(k)}_X < \infty,
\end{equation*}
which exists thanks to the assumption in \eqref{eq:lem:boundary_distributions;periodic;distr_assump}.
Note that, by \eqref{eq:measure_distr_FT} and Lemma~\ref{lemma:indeps},
\begin{align*}
\widehat{f}_j(k)-\widehat{f}_s(k)
=(\ee^{k(j-\theta)}-\ee^{k(s-\theta)})\widehat{f}_\theta(k).
\end{align*}
Moreover,  we have
\begin{align*}
\abs{\ee^{(j-\theta)k}-\ee^{(s-\theta)k}}
&\leq \abs{j-s}\,\abs{k} \max\{\ee^{(1-\theta)k},\ee^{-\theta k}\}.
\end{align*}
Therefore,
\begin{align*}
\nrm{\widehat{f}_j(k)-\widehat{f}_s(k)}_{X}
&\leq \abs{j-s}\,\abs{k}\max\{\ee^{k(1-\theta)},\ee^{-k\theta }\}\nrm{\widehat{f}_\theta(k)}_X \\
&\leq C\abs{j-s}(1+\abs{k})^{M+1},
\end{align*}
which shows that \eqref{eq:lem:boundary_distributions;periodic;reformulated} indeed holds true for $N=M+1$.
\end{proof}

The next lemma allows us to bootstrap pointwise convergence to bounded convergence for $f \in \ms{H}_\pi(\mathbb{S};X)$.

\begin{lemma}\label{lem:bootstrap_pointwise_bounded_convergence}
Let $X$ be a Banach space, let $\theta \in (0,1)$, let $f \in \ms{H}_\pi(\mathbb{S};X)$ and let $f_j \in \ms{S}'(\T;X)$ for $j=0,1$.
Then the following assertions are equivalent:
\begin{enumerate}[(i)]
    \item\label{it:lem:bootstrap_pointwise_bounded_convergence;FT_theta} $\widehat{f}_j(k) = \ee^{k(j-\theta)}\widehat{f}_\theta(k)$ for $k \in \Z$ and $j=0,1$.
    \item\label{it:lem:bootstrap_pointwise_bounded_convergence;FT_s} $\widehat{f}_j(k) = \ee^{k(j-s)}\widehat{f}_s(k)$ for every $k \in \Z$, $s \in (0,1)$ and $j=0,1$.
    \item\label{lem:bootstrap_pointwise_bounded_convergence;bc} $f_j = \lim_{s \to j}f_{s}$ in $\ms{S}'_{\mrm{bc}}(\T;X)$ for $j=0,1$.
    \item\label{lem:bootstrap_pointwise_bounded_convergence;pc} $f_j = \lim_{s \to j}f_{s}$ in $\ms{S}'(\T;X)$ for $j=0,1$.
\end{enumerate}
\end{lemma}
\begin{proof}
Note that the implication ``\ref{it:lem:bootstrap_pointwise_bounded_convergence;FT_theta}$\Rightarrow$\ref{lem:bootstrap_pointwise_bounded_convergence;bc}'' follows from Lemma~\ref{lem:boundary_distributions;periodic} and the implication  ``\ref{lem:bootstrap_pointwise_bounded_convergence;bc}$\Rightarrow$\ref{lem:bootstrap_pointwise_bounded_convergence;pc}'' is a direct consequence of the fact that the topology of bounded convergence is stronger than the topology of pointwise convergence.
As the implication ``\ref{it:lem:bootstrap_pointwise_bounded_convergence;FT_s}$\Rightarrow$\ref{it:lem:bootstrap_pointwise_bounded_convergence;FT_theta}'' holds true trivially, it thus suffices to prove that ``\ref{lem:bootstrap_pointwise_bounded_convergence;pc}$\Rightarrow$``\ref{it:lem:bootstrap_pointwise_bounded_convergence;FT_s}''

For ``\ref{lem:bootstrap_pointwise_bounded_convergence;pc}$\Rightarrow$``\ref{it:lem:bootstrap_pointwise_bounded_convergence;FT_s}''
 we deduce from $f_j = \lim_{s \to j}f_s$ in $\ms{S}'(\T;X)$ and the continuity of
the Fourier transform from $\ms{S}'(\T;X)$ to $\ms{S}'(\Z;X)$ that $\widehat{f}_j = \lim_{s \to j}\widehat{f}_s$ in $\ms{S}'(\Z;X) \hookrightarrow \ell^0(\Z;X).$
After pointwise multiplication with $(\ee^{(j-s)k})_{k \in \Z}$ we thus find that for $j=0,1$
$$
\widehat{f}_j(k) = \lim_{s \to j}\ee^{(j-s)k}\widehat{f}_s(k), \qquad k \in \Z.
$$
In view of Lemma~\ref{lemma:indeps}, the latter implies that $\widehat{f}_j(k) = \ee^{(j-s)k}\widehat{f}_s(k)$ for all $s \in (0,1)$ and $k \in \Z$.
\end{proof}

Combining Lemma \ref{lem:boundary_distributions;periodic} and Lemma \ref{lem:bootstrap_pointwise_bounded_convergence}, we now obtain the following statement on the existence of boundary values for functions in  $\ms{H}_{\pi}(\mbb{S};\mc{X}_0,\mc{X}_1)$.

\begin{proposition}\label{proposition:boundaryvalues}
  Let $(\mc{X}_0,\mc{X}_1)$ be a compatible couple of sequentially structured Banach spaces and let $\theta \in (0,1)$. Let $f \in \ms{H}_{\pi}(\mathbb{S};X_0+X_1)$. Then we have
 $f \in \ms{H}_\pi(\mathbb{S};\mc{X}_0,\mc{X}_1)$ if and only if there exist unique distributions $f_j \in \ms{S}'(\T;X_j)$ with $\widehat{f}_j \in \mf{S}_j$ and with the property that
\begin{equation*}
f_j = \lim_{s \to j}f_s \quad \text{in} \quad \ms{S}'(\T;X_0+X_1), \qquad j=0,1.
\end{equation*}
Moreover, we have
\begin{equation*}
  \nrm{f}_{\ms{H}_\pi(\mathbb{S};\mc{X}_0,\mc{X}_1)} = \max_{j = 0,1}\,\nrm{\widehat{f}_j}_{\mf{S}_j}.
\end{equation*}
\end{proposition}

\begin{proof}
  First assume that $f \in \ms{H}_\pi(\mathbb{S};\mc{X}_0,\mc{X}_1)$. Since $\mf{S}_j \hookrightarrow \ell^{\infty}(\Z;X_j) \hookrightarrow \ms{S}'(\Z;X_0+X_1)$ for $j=0,1$, we see that $f$ satisfies \eqref{eq:lem:boundary_distributions;periodic;distr_assump}, so by Lemma \ref{lem:boundary_distributions;periodic} we obtain the existence of $f_j \in \ms{S}'(\T;X_0+X_1)$ with the convergence $f_j = \lim_{s \to j}f_s$ in $\ms{S}'_{\mrm{bc}}(\T;X_0+X_1)$, which implies convergence in $\ms{S}'(\T;X_0+X_1)$.
  Furthermore we know that $$\widehat{f}_j = \ha{\ee^{k(j-\theta)}\widehat{f}_\theta(k)}_{k \in \Z} \in \mf{S}_j \hookrightarrow \ms{S}'(\Z;X_j),$$ so $f_j \in \ms{S}'(\T;X_j)$ and we have the norm equality. The converse implication is a direct consequence of Lemma \ref{lem:bootstrap_pointwise_bounded_convergence}.
\end{proof}

Proposition \ref{proposition:boundaryvalues} shows that, for an analytic $f \colon \mathbb{S} \to X_0+X_1,$ we can fully characterize when $f$ belongs to $\ms{H}_\pi(\mathbb{S};\mc{X}_0,\mc{X}_1)$ in terms of the existence of boundary data with certain properties. Moreover, it allows us to rewrite the norm of $\ms{H}_\pi(\mathbb{S};\mc{X}_0,\mc{X}_1)$ in terms of this boundary data. Combined with
 Proposition~\ref{proposition:complexsimple}, this now yields the upper complex method reformulation of our sequentially structured interpolation method.

\begin{theorem}\label{thm:complexsimple;upper-c}
Let $(\mc{X}_0,\mc{X}_1)$ be a compatible couple of sequentially structured Banach spaces and let $\theta \in (0,1)$.
For $x \in X_0+X_1$ we have
\begin{align*}
\nrm{x}_{(\mc{X}_0,\mc{X}_1)_\theta} &= \inf\,\cbraces{
\max_{j = 0,1}\,\nrm{\widehat{f}_j}_{\mf{S}_j}: f \in {\ms{H}}_\pi(\mathbb{S};\mc{X}_0,\mc{X}_1),\, f(\theta) =x}.
\end{align*}
\end{theorem}

Theorem~\ref{thm:complexsimple;upper-c} will enable us to use the formulation of Calder\'on's upper complex interpolation method given in \cite[Section~8.3]{Pi16} to prove the second part of Example \ref{example:interpolationmethods}\ref{it:ssicomplex}. As a preparation, let us discuss when one can write a holomorphic function on the strip $\mathbb{S}$ as a Poisson integral.

The Poisson kernels for the strip $\mathbb{S}$ are the functions $P_j$ for $j=0,1$ on $\mathbb{S} \times \R$ given by
\begin{equation}
P_{j}(u+iv;t)=\frac{\sin(\pi u)\exp(\pi(v-t))}{\sin^2(\pi u)+(\cos(\pi u)-(-1)^j\exp(\pi(v-t)))^2},
\end{equation}
for $u \in (0,1)$ and $v, t \in \R.$
Note that, for each $z\in \mathbb{S}$, $P_{j}(z;\,\cdot\,) \in \ms{S}(\R)$.

\begin{lemma}\label{lem:Poisson_transform}
Let $X$ be a Banach space. Let $f \in \ms{H}(\mathbb{S};X)$ be such that
\begin{equation*}\label{eq:lem:Poisson_transform;bdd_small_strips}
\sup_{s_0<\re(z)<s_1}\nrm{f(z)}_X < \infty, \qquad 0 < s_0 < s_1 < 1,
\end{equation*}
and the limits $f_j = \lim_{s \to j}f_{s}$ exist in $\ms{S}'(\R;X)$ for $j=0,1$.
Then
\begin{equation}\label{eq:lem:Poisson_transform}
f(z) = f_0[P_0(z;\,\cdot\,)] + f_1[P_1(z;\,\cdot\,)], \qquad z \in \mathbb{S}.
\end{equation}
\end{lemma}
\begin{proof}
Fix $z_0 \in \mathbb{S}$. Choose $0<s_0<s_1<1$ such that $s_0 < \re(z_0) < s_1$. Then $g(z):= f((1-z)s_0+zs_1)$ for $z \in \overline{\mathbb{S}}$ defines a bounded function $g \in \ms{H}(\overline{\mathbb{S}};X)$. Therefore, by \cite[Lemma~C.2.9]{HNVW16},
\begin{equation*}
g(z) = \int_\R g_0(t)P_0(z;t)\dd t+ \int_\R g_1(t)P_1(z;t)\dd t, \qquad z \in \mathbb{S}.
\end{equation*}
In particular, taking $z=z_0$, interpreting $f_{s_0}$ and $f_{s_1}$ as tempered distributions  and using the definition of $g$, we have
\begin{equation*}
f((1-z_0)s_0+z_0s_1) =  f_{s_0}[P_0(z_0;\cdot)]+ f_{s_1}[P_1(z_0;\cdot)].
\end{equation*}
Since $P_{j}(z;\,\cdot\,) \in \ms{S}(\R)$ for $j=0,1$, we obtain \eqref{eq:lem:Poisson_transform} in the limit by letting $s_0 \searrow 0$ and $s_1 \nearrow 1$.
\end{proof}

If  $(\mc{X}_0,\mc{X}_1)$ is a compatible couple of sequentially structured Banach spaces, any $f \in {\ms{H}}_\pi(\mathbb{S};\mc{X}_0,\mc{X}_1)$ satisfies the conditions of Lemma~\ref{lem:Poisson_transform} by Proposition \ref{proposition:boundaryvalues}. Moreover, one can check the convergence assumption $f_j = \lim_{s \to j} f_s$ for $j=0,1$ in Lemma~\ref{lem:Poisson_transform} for bounded functions in ${\ms{H}}(\mathbb{S};X)$, in which case we have a stronger form of convergence.

\begin{lemma}\label{lem:bm}
Let $X$ be a Banach space.
For each bounded function $f \in \ms{H}(\mathbb{S};X)$ there exists  unique measures $\mu_j \in \Lambda^\infty(\R;X)$ for $j=0,1$ with $\nrm{\mu_j}_{\Lambda^\infty(\R;X)} \leq \nrm{f}_{L^\infty(\mathbb{S};X)}$ and with the property that for all $\phi \in L^1(\R)$.
\begin{equation}\label{lem:boundary_measures}
\int_{\R}\phi(t)\dd\mu_j(t) = \lim_{s \to j}\int_{\R}\phi(t)f_s(t)\dd t
\end{equation}
with convergence in $X$.
\end{lemma}

\begin{proof}
Note that $\mu_j$ is uniquely determined by the condition \eqref{lem:boundary_measures}.
In order to establish the existence, for $s,t \in (0,1)$ and $a,b \in \R$ with $s \neq t$ and $a<b$ we denote by $\Gamma_{s,t,a,b}$ the rectangle with corners $$s+ai,t+ai,t+bi,s+bi,$$
oriented clockwise if $s>t$ and counterclockwise if $s<t$.
Then, by the Cauchy-Goursat theorem,
\begin{align*}
0 &= \int_{\Gamma_{s,t,a,b}}f(z)\dd z= i \int_a^b f_s(\tau)-  f_{t}(\tau)\dd \tau
+ \int_{s}^{t}f(\tau+ib)-f(\tau+ia)\dd\tau.
\end{align*}
As a consequence,
\begin{equation*}
\lim_{s \to j}\int_a^b f_s(\tau)\dd \tau    = \int_a^b f_{t}(\tau)\dd \tau +i\int_{j}^{t}[f(\tau+ib)-f(\tau+ia)]\dd\tau
\end{equation*}
and therefore
\begin{align*}
\nrms{\lim_{s \to j}\int_a^b f_s(\tau)\dd \tau}_{X}
&\leq \limsup_{t \to j}\big((b-a)+2|t-j| \big)\nrm{f}_{L^\infty(\mathbb{S};X)} \\
&=(b-a)\nrm{f}_{L^\infty(\mathbb{S};X)}.
\end{align*}
As in \cite[Remark~2.17]{Pi16}, it can then be shown that
\begin{equation}\label{eq:lem:boundary_measures;def_interval}
\mu_j([a,b)) :=  \lim_{s \to j}\int_a^b f_s(\tau)\dd \tau, \qquad a,b \in \R, a < b,
\end{equation}
extends to a unique $X$-valued measure $\mu_j \in \Lambda^\infty(\R;X)$ with $\nrm{\mu_j}_{\Lambda^\infty(\R;X)} \leq \nrm{f}_{L^\infty(\mathbb{S};X)}$.
From \eqref{eq:lem:boundary_measures;def_interval} it follows that \eqref{lem:boundary_measures} holds true for all step functions $\phi \in L^1(\R)$.
As $f$ is bounded on $\mathbb{S}$, an approximation argument subsequently yields \eqref{lem:boundary_measures} for all $\phi \in L^1(\R)$.
\end{proof}

Let $(X_0,X_1)$ be a compatible couple of Banach spaces.
As in \cite[Section~8.3]{Pi16}, we denote by $\widetilde{\mc{F}}(X_0,X_1)$ the space of all bounded analytic functions $f:\mathbb{S} \to X_0+X_1$ that are given as the Poisson transform
\begin{equation}\label{eq:Poisson_transform_measure}
f(z) = \int_\R P_0(z;t)\dd \mu_0(t)+\int_\R P_1(z;t)\dd \mu_1(t), \qquad  z \in \mathbb{S},
\end{equation}
for (necessarily unique) measures $\mu_j \in \Lambda^\infty(\R;X_j)$, $j=0,1$.
 $\widetilde{\mc{F}}(X_0,X_1)$ is equipped with the norm
$$
\nrm{f}_{\widetilde{\mc{F}}(X_0,X_1)} := \max_{j=0,1}\nrm{\mu_j}_{\Lambda^\infty(\R;X_j)}.
$$
In this notation, \cite[Theorem~8.31]{Pi16} states that we can formulate the norm of Calder\'on's upper complex method \cite{Ca64} as
\begin{equation}\label{eq:upper_complex_Pisier}
\nrm{x}_{\brac{X_0,X_1}^{\theta}} = \inf\,\cbraces{
\nrm{f}_{\widetilde{\mc{F}}(X_0,X_1)}: f \in \widetilde{\mc{F}}(X_0,X_1),\, f(\theta) =x}.
\end{equation}

\begin{example}\label{example:complexmethod;upper}
Let $(\mc{X}_0,\mc{X}_1)$ be a compatible couple of sequentially structured Banach spaces and let $\theta \in (0,1)$.
Take $p_0,p_1 \in [1,\infty]$.
If $\mf{S}_j  = \widehat{\Lambda}^{p_j}(\T;X_j)$ for $j=0,1$, then we have $(\mc{X}_0,\mc{X}_1)_\theta  = \brac{X_0,X_1}^{\theta}$.
\end{example}
\begin{proof}
Let $\mc{Y}_j=[X_j,\mf{T}_j]$ be given by $\mf{T}_j := \widehat{\Lambda}^{\infty}(\T;X_j)$ for $j=0,1$.
Then, by Proposition~\ref{proposition:embeddings}, $(\mc{Y}_0,\mc{Y}_1)_\theta \hookrightarrow (\mc{X}_0,\mc{X}_1)_\theta$.
Therefore, it suffices to show that
$(\mc{X}_0,\mc{X}_1)_\theta \hookrightarrow \brac{X_0,X_1}^{\theta}$ and $\brac{X_0,X_1}^{\theta} \hookrightarrow (\mc{Y}_0,\mc{Y}_1)_\theta$.

In order to show that $(\mc{X}_0,\mc{X}_1)_\theta \hookrightarrow \brac{X_0,X_1}^{\theta}$, let $x \in (\mc{X}_0,\mc{X}_1)_\theta$.
Take $f \in {\ms{H}}_\pi(\mathbb{S};\mc{X}_0,\mc{X}_1)$ with $f(\theta)=x$.
Then we have for $g(z) := \ee^{(z-\theta)^2}f(z)$ that $g \in \ms{H}(\mathbb{S};X_0+X_1)$ and $g(\theta)=x$. Define $h_s(t):= \ee^{(s+it-\theta)^2}$, which means that $g_{s}=h_{s}f_{s}$ for each $s \in (0,1)$. Note that $\lim_{s\to j} h_s= h_j$ in $\ms{S}(\R)$.
Furthermore, note that $\cbrace{h_s}_{s \in (0,1)}$ is bounded in $\ms{S}(\R)$ and, as a consequence, pointwise multiplication with $\cbrace{h_s}_{s \in (0,1)}$ is a uniformly bounded collection of operators on $\ms{S}(\R)$.
By Proposition \ref{proposition:boundaryvalues} and \eqref{eq:periodic_extension}, we have $f_j=\lim_{s \to j}f_{s}$ in $\ms{S}'_{\mrm{bc}}(\T;X) \hookrightarrow \ms{S}'_{\mrm{bc}}(\R;X)$.
Therefore, setting $g_j := h_jf_j$, we obtain for every $\phi \in \ms{S}(\R)$ that
\begin{equation*}
  g_j[\phi] - g_s[\phi] = \hab{f_j[h_j\phi] -  f_j[h_s\phi]}  + \hab{f_j[h_s\phi] - f_s[h_s\phi]} \to 0 \quad \text{for }s \to j
\end{equation*}
for $j=0,1$, i.e $\lim_{s \to j} g_s =  g_j$  in $\ms{S}'(\R;X)$.

Now, as in the proof of Example~\ref{example:complexmethod},
\begin{align*}
\nrm{g_j}_{\Lambda^{p_j}(\R;X_j)}
&\leq e^{(j-\theta)^2}\sum_{k \in \Z} \ee^{-k^2\pi^2 } \nrm{f_j}_{\Lambda^{p_j}(\T;X_j)} \\&\lesssim \nrm{\widehat{f}_j}_{\widehat{\Lambda}^{p_j}(\T;X_j)}\leq \max_{j = 0,1}\,\nrm{\widehat{f}_j}_{\mf{S}_j}.
\end{align*}
Since $g$ is given as
$$
g(z) = g_0[P_0(z;\,\cdot\,)] + g_1[P_1(z;\,\cdot\,)], \qquad z \in \mathbb{S},
$$
by Lemma~\ref{lem:Poisson_transform}, it follows from \cite[Remark~8.36]{Pi16} that
$$
\nrm{x}_{\brac{X_0,X_1}^{\theta}}
\lesssim_{p_0,p_1}\max_{j=0,1}\,\nrm{g_j}_{\Lambda^{p_j}(\R;X_j)} \lesssim \max_{j = 0,1}\,\nrm{\widehat{f}_j}_{\mf{S}_j}.
$$
Taking the infimum over all $f$ as above, we conclude $\nrm{x}_{\brac{X_0,X_1}^{\theta}} \lesssim_{p_0,p_1} \nrm{x}_{(\mc{X}_0,\mc{X}_1)_\theta}$ by Theorem~\ref{thm:complexsimple;upper-c}.

Next, in order to prove the embedding $\brac{X_0,X_1}^{\theta} \hookrightarrow (\mc{Y}_0,\mc{Y}_1)_\theta$, let $x \in \brac{X_0,X_1}^{\theta}$.
Take $f \in \widetilde{\mc{F}}(X_0,X_1)$ such that $f(\theta) =x$.
Then, by a combination of Lemma~\ref{lem:Poisson_transform} and Lemma~\ref{lem:bm}, we have that $f$ is given as \eqref{eq:Poisson_transform_measure} for measures $\mu_j \in \Lambda^\infty(\R;X_0+X_1)$, $j=0,1$, with for $\phi \in L^1(\R)$
\begin{equation}\label{eq:theorem:complexmethod;upper;conv_measure}
\int_{\R}\phi\,\dd\mu_j = \lim_{s \to j}\int_{\R}\phi(t)f_s(t)\dd t
\end{equation}
with convergence in $X_0+X_1$.
By the discussion on \cite[p.~322]{Pi16}, $\mu_0$ and $\mu_1$ are uniquely determined by~\eqref{eq:Poisson_transform_measure}.
Therefore, $\mu_j \in \Lambda^\infty(\R;X_j)$ with
\begin{equation*}
\max_{j=0,1}\,\nrm{\mu_j}_{\Lambda^\infty(\R;X_j)} = \nrm{f}_{\widetilde{\mc{F}}(X_0,X_1)}
\end{equation*}
by definition of $\widetilde{\mc{F}}(X_0,X_1)$.
Let $\psi$ be as in \eqref{eq:psiperiodic} and define
\begin{align*}
g(z) &:= \sum_{n \in \Z} f(z+2\pi i n) \cdot \psi(z+2\pi i n), &&z \in {\mathbb{S}}\\
\nu_j(A) &:= \sum_{n \in \Z} \int_{A+2\pi n}  \psi(j+it)\dd \mu_j(t), &&A \in \mc{B}(\T).
\end{align*}
Then, as in the proof of Example~\ref{example:complexmethod}, we have $g \in \ms{H}_\pi(\mathbb{S};X_0+X_1)$ with $g(\theta)=f(\theta)=x$ and   $\nrm{\nu_j}_{\Lambda^{\infty}(\T;X_j)}\lesssim \nrm{\mu_j}_{\Lambda^\infty(\R;X_j)}$.
Moreover, using \eqref{eq:theorem:complexmethod;upper;conv_measure}, we find that
\begin{equation*}
\int_{\R}\phi\dd\nu_j = \lim_{s \to j}\int_{\R}\phi(t)g_s(t)\dd t \qquad \text{in} \quad X_0+X_1, \qquad \phi \in L^1(\T),
\end{equation*}
so that $\nu_j = \lim_{s \to j}g_{s}$ in $\ms{S}'(\T;X_0+X_1)$. Therefore, by Proposition \ref{proposition:boundaryvalues}, we have
$g \in {\ms{H}}_\pi(\mathbb{S};\mc{Y}_0,\mc{Y}_1)$ with
\begin{equation*}
\nrm{g}_{{\ms{H}}_\pi(\mathbb{S};\mc{Y}_0,\mc{Y}_1)} = \max_{j=0,1}\nrm{\nu_j}_{\Lambda^{\infty}(\T;X_j)}\lesssim \max_{j=0,1} \,\nrm{\mu_j}_{\Lambda^\infty(\R;X_j)}= \nrm{f}_{\widetilde{\mc{F}}(X_0,X_1)}.
\end{equation*}
By Proposition~\ref{proposition:complexsimple}, we thus obtain that $\nrm{x}_{(\mc{Y}_0,\mc{Y}_1)_\theta} \lesssim \nrm{f}_{\widetilde{\mc{F}}(X_0,X_1)}$.
Taking the infimum over all $f$ as above we deduce that $\nrm{x}_{(\mc{Y}_0,\mc{Y}_1)_\theta} \lesssim \nrm{x}_{\brac{X_0,X_1}^{\theta}}$ by~\eqref{eq:upper_complex_Pisier}.
\end{proof}

\subsection{Changing the base number and the torus}\label{subsec:change_torus}

Recall from Subsection~\ref{subsec:change_base_number} that all the concrete interpolation methods from Example~\ref{example:interpolationmethods}, with the exception of the complex methods in \ref{it:ssicomplex}, fulfill the conditions of Proposition \ref{prop:change_basis} and can thus be realized with any base number $b \in (1,\infty)$.
In this subsection we will see that the complex interpolation methods from Example~\ref{example:interpolationmethods}\ref{it:ssicomplex} can also be realized with any base number $b \in (1,\infty)$ as well as with any underlying torus $\T_\lambda$ with $\lambda \in (0,\infty)$, where
$$
\T_\lambda:=  \faktor{\R}{2\lambda\Z} \simeq \cbraceb{ e^{i\frac{\pi}{\lambda}t} : t \in [-\lambda,\lambda) }.
$$
This will play an important role in our study of reiteration in Section \ref{section:reiteration}.

Let $\lambda \in (0,\infty)$.
For $p\in [1,\infty]$ and a Banach space $X$, we set
    \begin{equation*}
      \widehat{L}^p(\T_\lambda;X):= \cbraceb{\ms{F}_\lambda f: f \in L^p(\T_\lambda;X)},
    \end{equation*}
        with norm
\begin{equation*}
  \nrm{\ms{F}_\lambda f}_{\widehat{L}^p(\T;X)} := \frac{1}{(2\lambda)^{1/p}}\nrm{f}_{L^p(\T_\lambda;X)},
\end{equation*}
    where we use the $2\lambda$-periodic Fourier coefficients
    \begin{equation*}
  \ms{F}_\lambda f\,(k):= \frac{1}{2 \lambda} \int_{\T_\lambda} f(t)\ee^{-\frac{\pi}{\lambda}ikt}\dd t, \qquad  k \in \Z.
\end{equation*}
We define $\widehat{C}(\T_\lambda;X)$ and $\widehat{\Lambda}^p(\T_\lambda;X)$ analogously.

Let $b \in (1,\infty)$ and $\lambda \in (0,\infty)$ satisfy $\lambda = \frac{\pi}{\log b}$. As mentioned in Subsection \ref{subsec:change_base_number}, all the general theory of sequentially structured interpolation carries over verbatim from $(\mc{X}_0,\mc{X}_1)_{\theta} = (\mc{X}_0,\mc{X}_1)_{\theta;\ee}$ to  $(\mc{X}_0,\mc{X}_1)_{\theta;b}$.
For the theory on complex formulations in this section we need to add that the torus $\T=\T_{\pi}$ needs to be replaced by the torus $\T_{\lambda}$ and that the $2\pi$-periodicity in the spaces of analytic functions needs to be replaced by $2\lambda$-periodicity.
Regarding notation, we will correspondingly write $\ms{H}_{\lambda}$ instead of $\ms{H}_{\pi}$ in each of the spaces of analytic functions that appear in this section.

In the specific case of the complex interpolation methods from Example~\ref{example:interpolationmethods}\ref{it:ssicomplex}, we can change the base number and the torus independently from each other thanks to the scaling properties of the Fourier transform and the corresponding function spaces on the Fourier side, which is the content of the next example.

\begin{example}\label{ex:changing_base_torus_complex_int}
Let $(\mc{X}_0,\mc{X}_1)$ be a compatible couple of sequentially structured Banach spaces, let $q_0,q_1 \in [1,\infty]$, let $b \in (1,\infty)$, let  $\lambda \in (0,\infty)$ and let $\theta \in (0,1)$.
\begin{enumerate}[(i)]
    \item\label{it:ex:changing_base_torus_complex_int;lower} If
    \begin{equation*}
    \mf{S}_j= \begin{cases}
    \widehat{L}^{p_j}(\T_\lambda;X_j) &\text{ if $p_j\in [1,\infty)$}\\
    \widehat{C}(\T_\lambda;X_j)&\text{ if $p_j = \infty$},
    \end{cases}
    \end{equation*}
    for  $j=0,1$, then
    \begin{equation*}
    \brac{X_0,X_1}_{\theta} =
    \ha{\mc{X}_0,\mc{X}_1}_{\theta;b}.
    \end{equation*}
    \item\label{it:ex:changing_base_torus_complex_int;upper} If $\mf{S}_j = \widehat{\Lambda}^{p_j}(\T_\lambda;X_j)$  for $j=0,1$, then
    \begin{equation*}
    \brac{X_0,X_1}^{\theta} =
    \ha{\mc{X}_0,\mc{X}_1}_{\theta;b}.
    \end{equation*}
\end{enumerate}
\end{example}
\begin{proof}
Set $\mu := \frac{\pi}{\log b} \in (0,\infty)$ and note that the proofs of
Example~\ref{example:complexmethod} and Example~\ref{example:complexmethod;upper} remain valid with the base number $b$ instead of $\ee$ and the torus $\T_\mu$ instead of $\T=\T_{\pi}$.
Defining $\mf{T}_j$ in the same way as $\mf{S}_j$  with $\T_\lambda$ replaced by $\T_\mu$, we accordingly find that
$$
\left([X_0,\mf{T}_0],[X_1,\mf{T}_1]\right)_{\theta;b} = \left\{\begin{array}{lll}
\brac{X_0,X_1}_{\theta}, & \text{in case  \ref{it:ex:changing_base_torus_complex_int;lower}}, \\
\brac{X_0,X_1}^{\theta}, & \text{in case \ref{it:ex:changing_base_torus_complex_int;upper}}.
\end{array}\right.
$$
On the other hand, the observation that
$$
\ms{F}_\mu^{-1}\vec{x}=[\ms{F}_\lambda^{-1}\vec{x}](\frac{\lambda}{\mu}\,\cdot\,),\qquad\qquad  \vec{x} \in \ms{S}'(\Z;X_j),
$$
implies that $\mf{S}_j=\mf{T}_j$ for $j=0,1$. Combining the above we arrive at the desired result.
\end{proof}

\section{Interpolation of operators}\label{section:interpolationofoperators}
We now turn to the interpolation of operators using the sequentially structured interpolation method.
Let $(\mc{X}_0,\mc{X}_1)$ and $(\mc{Y}_0,\mc{Y}_1)$ be compatible couples of sequentially structured Banach spaces. To interpolate the boundedness of an operator $T\colon X_0 +Y_0 \to X_1+Y_1$ with the sequentially structured interpolation method, it is in general not sufficient to assume boundedness of $T$ from $X_j$ to $Y_j$ for $j=0,1$. Instead, we need a so-called $(\mf{S}_j, \mf{T}_j)$-boundedness assumption on $T$ for $j=0,1$, which we will introduce now.

\subsection{$(\mf{S},\mf{T})$-boundedness}\label{subsection:STboundedness}
Let $\mc{X}$ and $\mc{Y}$ be sequentially structured Banach spaces. Let $\vec{T} \in \ell^0(\Z;\mc{L}(X,Y))$ and define
\begin{equation*}
  \vec{T}\vec{x} := (T_kx_k)_{k\in \Z}, \qquad \vec{x} \in \ell^0(\Z;X).
\end{equation*}
We say that $\vec{T}$ is $(\mf{S}, \mf{T})$-bounded if $\vec{T}$ defines a bounded operator from $\mf{S}$ to $\mf{T}$.
When $\mc{X}=\mc{Y}$ we say that $\vec{T}$ is $\mf{S}$-bounded.
We say that a single operator $T \in \mc{L}(X,Y)$ is $(\mf{S},\mf{T})$-bounded if the sequence $(\ldots,T,T,T,\ldots)$ is $(\mf{S},\mf{T})$-bounded and write $$\nrm{T}_{\mf{S} \to \mf{T}}:= \nrm{(\ldots,T,T,T,\ldots)}_{\mf{S} \to \mf{T}}.$$

For specific choices of the sequence structures $\mf{S}$ and $\mf{T}$ we have that any $T \in \mc{L}(X,Y)$ is $(\mf{S},\mf{T})$-bounded. Indeed, this is for example the case in the following examples:
\begin{enumerate}[(i)]
  \item $\mf{S}=\ell^p(\Z;X)$ and $\mf{T} = \ell^p(\Z;Y)$ for $p \in [1,\infty]$.
  \item $\mf{S}=\widehat{L}^p(\T;X)$ and $\mf{T} = \widehat{L}^p(\T;Y)$ for $p \in [1,\infty]$.
  \item $\mf{S}=\widehat{\Lambda}^p(\T;X)$ and $\mf{T} = \widehat{\Lambda}^p(\T;Y)$ for $p \in [1,\infty]$.
  \item $\mf{S}=\varepsilon^p(\Z;X)$ and $\mf{T} = \varepsilon^q(\Z;Y)$ for $p,q \in [1,\infty)$.
  \item $\mf{S}=\gamma^p(\Z;X)$ and $\mf{T} = \gamma^q(\Z;Y)$ for $p,q \in [1,\infty)$.
  \item If $X$ and $Y$ are Banach lattices and $\mf{S}=X(\ell^2(\Z))$ and $\mf{T} = Y(\ell^2(\Z))$.
\end{enumerate}
The proof of the first four claims is trivial, the fifth and sixth follow from the Kahane-Khintchine inequalities (see \cite[Section 6.2.b]{HNVW16}) and the final claim is a consequence of the Krivine-Grothendieck theorem (see \cite[Theorem 1.f.14]{LT79}). The final claim fails when one replaces the $2$ by any $q \in [1,\infty] \setminus \cbrace{2}$ (see \cite[Example 2.16]{KU14}).

For sequences $\vec{T} \in \ell^0(\Z;\mc{L}(X,Y))$, we have the following relations to pre-existing notions in the literature.

\begin{example}
Let $X$ and $Y$ be Banach spaces and $\vec{T} \in \ell^0(\Z;\mc{L}(X,Y))$.
\begin{enumerate}[(i)]
  \item For $p \in [1,\infty]$  we have
  \begin{equation*}
    \nrm{\vec{T}}_{\ell^p(\Z;X)\to \ell^p(\Z;Y)} = \sup_{k \in \Z}\, \nrm{T_k}_{X \to Y}.
  \end{equation*}
  \item For $p,q \in [1,\infty)$ we have
  \begin{align*}
     \nrm{\vec{T}}_{\varepsilon^p(\Z;X)\to \varepsilon^q(\Z;Y)} &\eqsim_{p,q} \nrm{\cbrace{T_k:k \in \Z}}_{\mc{R}} \\
     \nrm{\vec{T}}_{\gamma^p(\Z;X)\to \gamma^q(\Z;Y)} &\lesssim_{p,q} \nrm{\cbrace{T_k:k \in \Z}}_{\gamma},
  \end{align*}
  where $\nrm{\Gamma}_{\mc{R}}$ and $\nrm{\Gamma}_{\gamma}$ denote the $\mc{R}$- and $\gamma$-bound of a set $\Gamma\subseteq \mc{L}(X,Y)$ respectively (see \cite[Chapter 8]{HNVW17}).
  \item If $X$ and $Y$ are Banach lattices,   we have for $q \in [1,\infty]$
  \begin{equation*}
    \nrm{\vec{T}}_{X(\ell^q(\Z))\to Y(\ell^q(\Z))} \leq \nrm{\cbrace{T_k:k \in \Z}}_{\ell^q}
  \end{equation*}
  where $\nrm{\Gamma}_{\ell^q}$ denotes the $\ell^q$-bound of a set $\Gamma\subseteq \mc{L}(X,Y)$  (see \cite{KU14}).
\end{enumerate}
\end{example}

\subsection{Interpolation of operators} Let $(\mc{X}_0,\mc{X}_1)$ and $(\mc{Y}_0,\mc{Y}_1)$ be compatible couples of sequentially structured Banach spaces. To interpolate the boundedness of an operator $T\colon X_0 +Y_0 \to X_1+Y_1$ with the sequentially structured interpolation method, we will use $(\mf{S}_j,\mf{T}_j)$-boundedness assumptions on $T$. As we saw in the previous subsection, when $\mf{S}_j$ and $\mf{T}_j$ are e.g. the sequence structures associated with real or complex interpolation method, the boundedness of $T$ from $X_j$ to $Y_j$ implies the $(\mf{S}_j,\mf{T}_j)$-boundedness of $T$ for $j=0,1$. So, in these cases, our theory covers the well-known interpolation of operators with these methods. Our theory also covers e.g. $\ell^q$-interpolation, in which case we can only interpolate $\ell^q$-bounded operators.

\begin{theorem}\label{theorem:interpolation_operators}
Let $(\mc{X}_0,\mc{X}_1)$ and $(\mc{Y}_0,\mc{Y}_1)$ be compatible couples of sequentially structured Banach spaces.
Let $T\colon X_0+X_1 \to Y_0 + Y_1$ be a linear operator such that $T\colon X_j \to Y_j$ is $(\mf{S}_j,\mf{T}_j)$-bounded for $j=0,1$.
Then $T$ acts as a bounded linear operator from $(\mc{X}_0,\mc{X}_1)_\theta$ to $(\mc{Y}_0,\mc{Y}_1)_\theta$ for any $\theta \in (0,1)$ with
\begin{equation*}
  \nrm{T}_{(\mc{X}_0,\mc{X}_1)_\theta \to (\mc{Y}_0,\mc{Y}_1)_\theta} \leq \ee^\theta \nrm{T}^{1-\theta}_{\mf{S}_0 \to \mf{T}_0} \nrm{T}^\theta_{\mf{S}_1 \to \mf{T}_1}.
\end{equation*}
\end{theorem}
\begin{proof}
Let $x \in (\mc{X}_0,\mc{X}_1)_\theta$ and $\varepsilon>0$ and assume without loss of generality that $\nrm{T}_{\mf{S}_0 \to \mf{T}_0}, \nrm{T}_{\mf{S}_1 \to \mf{T}_1}>0$.
 Let $\vec{x} \in \mf{S}_0(\ee^{-\theta}) \cap \mf{S}_1(\ee^{1-\theta})$ such that
$\sum_{k \in \Z} x_k = x$ with convergence in $X_0+X_1$  and
\begin{equation*}
\nrm{\vec{x}}_{\mf{S}_0(\ee^{-\theta})\cap \mf{S}_1(\ee^{1-\theta})} \leq \nrm{x}_{(\mc{X}_0,\mc{X}_1)_\theta}+\varepsilon.
\end{equation*}
Let $n \in \Z$ be such that $\ee^n \leq \frac{\nrm{T}_{\mf{S}_0 \to \mf{T}_0}}{\nrm{T}_{\mf{S}_1 \to \mf{T}_1} }\leq \ee^{n+1}$ and define $\vec{y}:= (x_{k-n})_{k \in \Z}$. Then  $\sum_{k \in \Z} y_k = x$ with convergence in $X_0+X_1$. Moreover, by \eqref{eq:sequence2}, we have $\vec{y} \in \mf{S}_0(\ee^{-\theta}) \cap \mf{S}_1(\ee^{1-\theta})$  with
$$
\nrm{\vec{y}}_{\mf{S}_0(\ee^{-\theta})\cap \mf{S}_1(\ee^{1-\theta})} \leq  \ee^{n(j-\theta)}\,\nrm{\vec{x}}_{\mf{S}_0(\ee^{-\theta})\cap \mf{S}_1(\ee^{1-\theta})}.
$$
Therefore, we obtain
\begin{align*}
  \nrm{Tx}_{(\mc{Y}_0,\mc{Y}_1)_\theta} &\leq \max_{j=0,1}\,\nrmb{(\ee^{k(j-\theta)}Ty_k)_{k \in \Z}}_{\mf{T}_j}\\
  &\leq \max_{j=0,1}\,\ee^{n(j-\theta)}  \nrm{T}_{\mf{S}_j \to \mf{T}_j} \nrmb{(\ee^{k(j-\theta)}x_k)_{k \in \Z}}_{\mf{S}_j}\\
  &\leq \max_{j=0,1}\,\nrm{T}_{\mf{S}_j \to \mf{T}_j}\has{\frac{\nrm{T}_{\mf{S}_0 \to \mf{T}_0}}{\nrm{T}_{\mf{S}_1 \to \mf{T}_1} }}^{j-\theta} \ee^{\theta} \nrm{\vec{x}}_{\mf{S}_j(\ee^{j-\theta})}\\
  &\leq \ee^\theta \nrm{T}^{1-\theta}_{\mf{S}_0 \to \mf{T}_0} \nrm{T}^\theta_{\mf{S}_1 \to \mf{T}_1}\hab{\nrm{x}_{(\mc{X}_0,\mc{X}_1)_\theta}+\varepsilon}
\end{align*}
Noting that $\varepsilon>0$ was arbitrary, this proves the proposition.
\end{proof}

\begin{remark}\label{rmk:theorem:interpolation_operators}
  The constant $\ee^\theta$ in Theorem \ref{theorem:interpolation_operators} is an artifact of our discrete method. Upon inspection of the proof of Theorem \ref{theorem:interpolation_operators}, one sees that one can get rid of this constant when $\frac{\nrm{T}_{\mf{S}_0 \to \mf{T}_0}}{\nrm{T}_{\mf{S}_1 \to \mf{T}_1} }= \ee^{n}$ for some $n \in \Z$. This holds in particular  in the important special case that $\nrm{T}_{\mf{S}_0 \to \mf{T}_0} = \nrm{T}_{\mf{S}_1 \to \mf{T}_1} = 1$.
\end{remark}

Let us illustrate that the $(\mf{S}_j,\mf{T}_j)$-boundedness assumptions for $j=0,1$ in Theorem \ref{theorem:interpolation_operators} can not be omitted in general, but also that it is not a necessary condition in all cases.

\begin{example}\label{ex:interpolatrion_Fourier_transform}
Let $q \in (1,\infty)$ and $\theta \in (0,1)$. Set
\begin{equation*}
(\mc{X}_0,\mc{X}_1) := \left([L^2(\R^n),L^2(\R^n;\ell^q(\Z))],[W^{1,2}(\R^n),W^{1,2}(\R^n;\ell^q(\Z))]\right),
\end{equation*}
and
\begin{equation*}
(\mc{Y}_0,\mc{Y}_1) := \left([L^2(\R^n),L^2(\R^n;\ell^q(\Z))],[L^2_1(\R^n),L^{2}_1(\R^n;\ell^q(\Z))]\right),
\end{equation*}
where for a Banach space $X$ we set
\begin{equation*}
L^{2}_s(\R^n;X) = \big\{ f \in L^0(\R^n;X) : (1+\abs{\,\cdot\,}^2)^{s/2} f \in L^2(\R^n;X) \big\}  ,\qquad s \in \R.
\end{equation*}
Then the Fourier transform $\ms{F}$ is bounded from $X_j$ to $Y_j$ for $j=0,1$, while, on the one hand,
\begin{equation}\label{eq:ex:inter_Fourier_transform}
\ms{F} \in \mc{L}\left((\mc{X}_0,\mc{X}_1)_\theta,(\mc{Y}_0,\mc{Y}_1)_\theta\right)
\end{equation}
if and only if $q \in (1,2]$ and, on the other hand, $\ms{F}$ is $(\mf{S}_j,\mf{T}_j)$-bounded for $j=0,1$ if and only if $q=2$.
\end{example}
\begin{proof}
Let us first note that the statement on the $(\mf{S}_j,\mf{T}_j)$-boundedness of $\ms{F}$ follows from Kwapie\'n's theorem (see e.g.\ \cite[Theorem~2.1.18]{HNVW16}).
So it remains to prove the characterization of \eqref{eq:ex:inter_Fourier_transform}.

By \cite[Proposition~5.1]{Ku15} and \cite[Proposition~4.13]{KU14}, we have
\begin{equation*}
(\mc{X}_0,\mc{X}_1)_\theta = F^{\theta}_{2,q}(\R^{n}),
\end{equation*}
where $F^{\theta}_{2,q}(\R^{n})$ denotes a Triebel-Lizorkin space.
Furthermore, by Example~\ref{example:BFSembeddings} and \cite[1.18.5]{Tr78}, we have
\begin{equation*}
(\mc{Y}_0,\mc{Y}_1)_\theta = ({Y}_0,{Y}_1)_{\theta, \ell^q,\ell^q} =  [L^2(\R^n),L^2_1(\R^n)]_\theta = L^2_\theta(\R^n).
\end{equation*}
Therefore, \eqref{eq:ex:inter_Fourier_transform} is equivalent to $F^{\theta}_{2,q}(\R^{n}) \hookrightarrow H^{\theta,2}(\R^n)$, where $H^{\theta,2}(\R^n)$ denotes a Bessel potential space.
In view of the Littlewood-Paley decomposition $H^{\theta,2}(\R^n) = F^{\theta}_{2,2}(\R^{n})$, the embedding $F^{\theta}_{2,q}(\R^{n}) \hookrightarrow H^{\theta,2}(\R^n)$ holds true if and only if $q \in (1,2]$. This finishes the proof.
\end{proof}

The above example implies that $(\mc{X}_0,\mc{X}_1)_\theta$ and $(\mc{Y}_0,\mc{Y}_1)_\theta$ are not necessarily what in the literature is called an \emph{interpolation pair} for $(X_0,X_1)$ and $(Y_0,Y_1)$. However, this is not the right way to think about sequentially structured interpolation, as it does not take the given sequence structures on $(X_0,X_1)$ and $(Y_0,Y_1)$ into account.
Instead, we need to think about $(\mc{X}_0,\mc{X}_1)_\theta$ and $(\mc{Y}_0,\mc{Y}_1)_\theta$ as an interpolation pair for $(\mc{X}_0,\mc{X}_1)$ and $(\mc{Y}_0,\mc{Y}_1)$. See also \cite[Remark~2.12]{Ku15} for some category theoretical considerations for the $\ell^q$-interpolation method.

\subsection{Stein interpolation}
In \cite{St56} Stein proved a convexity principle for the interpolation of analytic operator families on $L^p$-spaces. An important special case of \cite[Theorem 1]{St56} states that an analytic family of linear operators $\cbrace{T_z}_{z \in \overline{\mathbb{S}}}$ for which
\begin{equation*}
  \nrm{T_{j+it}f}_{L^{q_j}(S)} \leq M_j \nrm{f}_{L^{p_j}(S)}, \qquad t \in \R,
\end{equation*}
for any simple function $f$ and $p_0,p_1,q_0,q_1 \in [1,\infty]$ also satisfies
\begin{equation*}
  \nrm{T_{\theta}f}_{L^{q}(S)} \leq M_0^{1-\theta}M_1^\theta \nrm{f}_{L^p{(S)}},
\end{equation*}
for $\theta \in (0,1)$, $\frac{1}{p} = \frac{1-\theta}{p_0}+\frac{\theta}{p_1}$ and $\frac{1}{q} = \frac{1-\theta}{q_0}+\frac{\theta}{q_1}$.
After the development of the complex interpolation method by Calder\'on \cite{Ca64}, this theorem was generalized to general compatible couples of (quasi)-Banach spaces, see e.g. \cite{CJ84,CS88,Vo92}.
 In \cite{SW06} Stein interpolation was proved for the $\gamma$-interpolation method, using the complex formulation of the $\gamma$-interpolation method in \cite{KLW19}. Using the complex formulation of the sequentially structured interpolation method, we will now prove Stein interpolation for any two compatible couples of sequentially structured Banach spaces $(\mc{X}_0,\mc{X}_1)$ and $(\mc{Y}_0,\mc{Y}_1)$. In particular we obtain Stein interpolation for the real interpolation method, which we already treated in a continuous setting in \cite{LL21c}. Note that Stein interpolation for the abstract framework in \cite{CKMR02} was posed as an open problem, see \cite[p.662]{Ka16b}.

\begin{theorem}\label{theorem:Steininterpolation1}
Let $(\mc{X}_0,\mc{X}_1)$ and $(\mc{Y}_0,\mc{Y}_1)$ be compatible couples of sequentially structured Banach spaces and let $\breve{X}$ be a dense subspace of $X_0\cap X_1$. Suppose that $\mf{S}_0$ and $\mf{S}_1$ are Ces\`aro bounded. Let $\cbrace{T(z)}_{z \in \overline{\mathbb{S}}}$ be a family of linear operators from $\breve{X}$ to $Y_0+Y_1$ such that:
\begin{enumerate}[(i)]
  \item $T(\cdot) x \in \ms{H}_{\pi}(\overline{\mathbb{S}};Y_0+Y_1)$ for all $x \in \breve{X}$
  \item \label{it:stein2} For $j=0,1$ there are $M_j>0$ such that for any $\vec{x} \in c_{00}(\Z;\breve{X})$.
  \begin{equation*}
    \nrms{ \ms{F}\hab{t \mapsto \sum_{k \in \Z} \ee^{ikt} T\ha{j+it}x_k} }_{\mf{S}_j} \leq M_j \, \nrm{\vec{x}}_{\mf{S}_j}.
  \end{equation*}
\end{enumerate}
Then we have $T(\theta)\breve{X} \subseteq (\mc{Y}_0,\mc{Y}_1)_\theta$ for any $\theta \in (0,1)$ with
\begin{equation*}
  \nrm{T(\theta) x}_{(\mc{Y}_0,\mc{Y}_1)_\theta} \lesssim_\theta  M_0^{1-\theta} M_1^\theta \, \nrm{x}_{(\mc{X}_0,\mc{X}_1)_\theta}, \qquad x \in \breve{X}.
\end{equation*}
In particular, if either $\mc{X}_0$ or $\mc{X}_1$ is Ces\`aro convergent, $T(\theta)$ extends to a bounded linear operator from $(\mc{X}_0,\mc{X}_1)_\theta$ to $(\mc{Y}_0,\mc{Y}_1)_\theta$
\end{theorem}

\begin{proof}
  Let $x \in \breve{X}$. Using Corollary \ref{cor:theorem:complexmethodformulation;tensor} we can find an
  $\vec{x} \in c_{00}(\Z;\breve{X})$ such that $f(z) := \sum_{k \in \Z}\ee^{k(z-
  \theta)} x_k$ satisfies $f(\theta)=x$ and
  \begin{equation*}
    \max_{j = 0,1}\,\nrm{\widehat{f}_j}_{\mf{S}_j} \lesssim_\theta \nrm{x}_{(\mc{X}_0,\mc{X}_1)_\theta}.
  \end{equation*}
  Note that $(\ee^{k(j-\theta)}x_k)_{k \in \Z} = \widehat{f}_j$.
  Let $n \in \Z$ such that $\ee^n\leq \frac{M_0}{M_1} \leq \ee^{n+1}$ and define $g(z) := \ee^{n(z-\theta)} T(z) f(z)$. Then $g(\theta) = T(\theta) x$ and by our assumptions $g \in \ms{H}_\pi(\overline{\mathbb{S}};Y_0+Y_1)$. Thus, using Theorem \ref{theorem:complexmethodformulation} and \eqref{eq:sequence2}, we have
  \begin{align*}
    \nrm{T(\theta) x}_{(\mc{Y}_0,\mc{Y}_1)} &\leq \max_{j=0,1} \, \nrm{\widehat{g}_j}_{\mf{S}_j}\\
    &= \max_{j=0,1} \, \ee^{n(j-\theta)} \nrms{ \ms{F}\hab{t \mapsto \sum_{k \in \Z}  \ee^{k(j-\theta)} \cdot\ee^{i(k+n)t} \cdot T\ha{j+it}x_k} }_{\mf{S}_j}\\
    &\leq \ee^{\theta}\max_{j=0,1} \has{\frac{M_0}{M_1}}^{j-\theta} M_j \, \nrm{(\ee^{(k-n)(j-\theta)}x_{k-n})_{k \in \Z}}_{\mf{S}_j}\\
    &\lesssim_\theta  M_0^{1-\theta} M_1^\theta\, \nrm{x}_{(\mc{X}_0,\mc{X}_1)_\theta},
  \end{align*}
 proving the first claim. The second claim follows by density (see Corollary \ref{corollary:denseX0X1}).
\end{proof}

\begin{remark}~\label{remark:stein}
  \begin{enumerate}[(i)]
    \item Stein interpolation for the complex method is often formulated for families of analytic operators from either $X_0+X_1$ or $X_0\cap X_1$ to $Y_0+Y_1$. However, checking the analyticity of $T(\cdot)x$ in applications  is often only feasible for specific $x$, e.g.\ taking $\breve{X}$ the space of simple functions when $X_j = L^{p_j}$ for $j=0,1$. See also \cite[Remark 2.2]{Vo92}.
    \item The periodicity assumption on $\cbrace{T(z)}_{z \in \overline{\mathbb{S}}}$ in Theorem \ref{theorem:Steininterpolation1}, which is an artifact of our discrete framework, can sometimes be inconvenient in applications. For specific examples, one can use a continuous formulation to circumvent the periodicity assumption, which for the real interpolation can be found \cite{LL21c} and for the $\gamma$-interpolation method in  \cite{SW06}.
    \item\label{it:steinlambda} In the spirit of Subsection \ref{subsec:change_torus}, Theorem \ref{theorem:Steininterpolation1} also works in a $2\lambda$-periodic setting for $\lambda \in (0,\infty)$.
  \end{enumerate}
\end{remark}

Taking $\mf{S}_j = \widehat{C}(\T;X_j)$ and $\mf{T}_j = \widehat{C}(\T;Y_j)$ in Theorem \ref{theorem:Steininterpolation1}, assumption  \ref{it:stein2} reduces to $T\ha{j+it}x \in C(\T;Y_j)$ for all $x \in \breve{X}$ and
\begin{equation*}
  \nrm{T\ha{j+it}x}_{Y_j} \leq M_j \, \nrm{x}_{X_j}, \qquad t \in \R,\, x \in \breve{X},
\end{equation*}
for $j=0,1$, which yields a periodic version of Stein interpolation for complex interpolation (see \cite[Theorem 2.1]{Vo92}). The power of Theorem~\ref{theorem:Steininterpolation1} is of course that it also works for other interpolation methods. For example, using Stein interpolation for the real interpolation method we interpolated weighted $L^p$-spaces and analytic semigroups in \cite{LL21c}. Stein interpolation for the sequentially structured interpolation method will play a key role in
Lemma~\ref{lemma:steinweightedss}, which in turn allows us to deduce a reiteration theorem.

\section{Reiteration}\label{section:reiteration}
In this section we will study reiteration for the sequentially structured interpolation method, i.e. we will investigate what happens when you interpolate sequentially structured interpolation spaces. For this, it will be crucial to work with sequentially structured interpolation spaces with general base number, which we introduced in Subsection \ref{subsec:change_base_number}. We will
freely use the developed general theory for sequentially structured interpolation
with base number $e$ for the general base number $b$, see also Subsections \ref{subsec:change_base_number} and \ref{subsec:change_torus}.

Our main  result reads as follows:

\begin{theorem}\label{thm:reiteration}
Let $(\mc{X}_0,\mc{X}_1)$ and $(\mc{Y}_0,\mc{Y}_1)$ be compatible couples of sequentially structured Banach spaces.
Let $0 \leq \theta_0 < \theta_1 \leq 1$ and $\theta \in (0,1)$ and set $\omega = (1-\theta)\theta_0 + \theta\theta_1$. Take $a,b \in (1,\infty)$ satisfying $b=a^{\theta_1-\theta_0}$ and suppose
\begin{equation}\label{eq:reitembedding}
    \mf{S}_0(a^{-\theta_j}) \cap \mf{S}_1(a^{1-\theta_j}) \hookrightarrow
    \mf{T}_j\hookrightarrow \mf{S}_0(a^{-\theta_j}) + \mf{S}_1(a^{1-\theta_j}),
     \end{equation}
for $j=0,1$. Then we have
    \begin{equation*}
   (\mc{X}_0,\mc{X}_1)_{\omega;a}=  (\mc{Y}_0,\mc{Y}_1)_{\theta;b}.
    \end{equation*}
\end{theorem}

In Theorem \ref{thm:reiteration}, one typically takes $Y_j=(\mc{X}_0,\mc{X}_1)_{\theta_j}$. In the upcoming subsections, we will check that the sequence structure $\mf{T}_j$ on $Y_j$ satisfies the assumed embeddings in \eqref{eq:reitembedding} in various concrete situations. Using the change of basis results in Proposition \ref{prop:change_basis} and Example~\ref{ex:changing_base_torus_complex_int}, this will, in particular, yield reiteration identities for the real, complex and $\gamma$-interpolation methods.

\begin{proof}[Proof of Theorem~\ref{thm:reiteration}]
First let $x \in (\mc{X}_0,\mc{X}_1)_{\omega;a}$.
Let $\vec{x} \in \ell^0(\Z;X)$ be such that $\vec{x} \in \mf{S}_0(a^{-\omega}) \cap \mf{S}_1(a^{1-\omega})$ and $x = \sum_{k \in \Z}x_k$ in $X_0+X_1$.
Then, as $(j-\theta)(\theta_1-\theta_0) = \theta_j-\omega$, we have by the first embedding in \eqref{eq:reitembedding}
\begin{align*}
\nrm{\vec{x}}_{\mf{T}_0(b^{-\theta}) \cap \mf{T}_1(b^{1-\theta})}
= \max_{j=0,1}\nrm{\vec{x}}_{\mf{T}_j(a^{\theta_j-\omega})} &\lesssim \max_{i=0,1}\,\nrm{\vec{x}}_{\mf{S}_i(a^{i-\omega})} \\&= \nrm{\vec{x}}_{\mf{S}_0(a^{-\omega}) \cap \mf{S}_1(a^{1-\omega})}.
\end{align*}
In particular, by Remark~\ref{remark:automaticconvergenceX0X1}, we know that $\sum_{k \in \Z}x_k$ also converges in $Y_0+Y_1$ and since the second embedding in \eqref{eq:reitembedding} implies that $Y_0+Y_1 \hookrightarrow X_0+X_1$, we must have $x = \sum_{k \in \Z}x_k$ in $Y_0+Y_1$.
Therefore, we obtain
\begin{equation*}
\nrm{x}_{(\mc{Y}_0,\mc{Y}_1)_{\theta;b}} \lesssim \nrm{\vec{x}}_{\mf{S}_0(a^{-\omega}) \cap \mf{S}_1(a^{1-\omega})}.
\end{equation*}
Taking the infimum over all representations $x = \sum_{k \in \Z}x_k$, we conclude
\begin{equation*}
\nrm{x}_{(\mc{Y}_0,\mc{Y}_1)_\theta;b} \lesssim \nrm{x}_{(\mc{X}_0,\mc{X}_1)_{\omega;a}}.
\end{equation*}
This proves the embedding  $(\mc{X}_0,\mc{X}_1)_{\omega;a} \hookrightarrow (\mc{Y}_0,\mc{Y}_1)_{\theta;b}$.

Now let $y \in (\mc{Y}_0,\mc{Y}_1)_{\theta;b}$ and put $\vec{y} := (\ldots,y,y,y,\ldots)$.
Let $\vec{y}_0 \in \mf{T}(b^{-\theta})$ and $\vec{y}_1 \in \mf{T}(b^{1-\theta})$ be such that $\vec{y}=\vec{y}_0+\vec{y}_1$.
Then, using the second embedding in \eqref{eq:reitembedding}, we have
\begin{align*}
\nrm{\vec{y}}_{\mf{S}_0(a^{-\omega})+\mf{S}_1(a^{1-\omega})}
&\leq \norm{\vec{y}_0}_{\mf{S}_0(a^{-\omega})+\mf{S}_1(a^{1-\omega})} + \nrm{\vec{y}_1}_{\mf{S}_0(a^{-\omega})+\mf{S}_1(a^{1-\omega})} \\
&\lesssim \norm{\vec{y}_0}_{\mf{T}_0(a^{\theta_0-\omega})} + \nrm{\vec{y}_1}_{\mf{T}_1(a^{\theta_1-\omega})} \\
&= \nrm{\vec{y}_0}_{\mf{T}_0(b^{-\theta})} + \nrm{\vec{y}_1}_{\mf{T}_1(b^{1-\theta})}.
\end{align*}
Taking the infimum over all decompositions $\vec{y}=\vec{y}_0+\vec{y}_1$, we find that
\begin{equation*}
\norm{\vec{y}}_{\mf{S}_0(a^{-\omega})+\mf{S}_1(a^{1-\omega})} \lesssim
\norm{\vec{y}}_{\mf{T}_0(b^{-\theta})+\mf{S}_0(b^{1-\theta})}.
\end{equation*}
By  Theorem~\ref{theorem:meanmethod}, this proves the embedding $(\mc{Y}_0,\mc{Y}_1)_{\theta;b} \hookrightarrow (\mc{X}_0,\mc{X}_1)_{\omega;a}$.
\end{proof}

\subsection{The real interpolation method}
In this subsection we will show that Theorem~\ref{thm:reiteration} implies the classical reiteration theorem for the real interpolation method. Let us start by recalling the classes $J_\theta$ and $K_\theta$ (see e.g. \cite[1.10.1]{Tr78}).
Let $(X_0,X_1,Y)$ be a compatible triple of Banach spaces   and let $\theta \in [0,1]$.
\begin{itemize}
    \item $Y$ is said to be of \emph{class $J_\theta$} between $X_0$ and $X_1$ if $X_0 \cap X_1 \hookrightarrow Y$ and for all $t \in (0,\infty)$ and $x \in X_0 \cap X_1$,
    $$
    \nrm{x}_{Y} \lesssim \max\big\{t^{-\theta}\nrm{x}_{{X}_0},t^{1-\theta}\nrm{x}_{{X}_1}\big\} =: t^{-\theta}J(t,x;X_0,X_1).
    $$
    \item $Y$ is said to be of \emph{class $K_\theta$} between $X_0$ and $X_1$ if $Y \hookrightarrow X_0+X_1$ and for all $t \in (0,\infty)$ and $x \in Y$,
    $$
    t^{-\theta}K(t,x;X_0,X_1) :=
    \inf\left(t^{-\theta}\nrm{x_0}_{X_0}+t^{1-\theta}\nrm{x_1}_{X_1}\right) \lesssim \nrm{x}_{Y},
    $$
    where the infimum is taken over all $x_0 \in X_0$ and $x_1 \in X_1$ such that $x=x_0+x_1$.
\end{itemize}
Given these definitions, the classical reiteration theorem for the real interpolation method (see e.g. {\cite[1.10.2]{Tr78}}) reads as follows:

\begin{example} \label{example:realreit}
Let $(X_0,X_1)$ and $(Y_0,Y_1)$ be compatible couples of Banach spaces.
Let $0 \leq \theta_0 < \theta_1 \leq 1$ and $\theta \in (0,1)$ and set $\omega = (1-\theta)\theta_0 + \theta\theta_1$. Suppose that $Y_j$ is of class $J_{\theta_j}$ and $K_{\theta_j}$ between $X_0$ and $X_1$
for $j=0,1$. Then we have for $p \in [1,\infty]$
    \begin{equation*}
   ({X}_0,{X}_1)_{\omega,p}=  ({Y}_0,{Y}_1)_{\theta,p}.
    \end{equation*}
\end{example}

Using Example \ref{example:interpolationmethods}\ref{it:ssireal} and Proposition \ref{prop:change_basis}, we know that the conclusions of Theorem~\ref{thm:reiteration} and Example~\ref{example:realreit} coincide when one takes the sequence structures in Theorem~\ref{thm:reiteration} to be $\ell^p(\Z;X)$-spaces. Therefore, to deduce Example~\ref{example:realreit}  from Theorem~\ref{thm:reiteration},
it suffices to relate the
the classes $J_\theta$ and $K_\theta$  to the assumed embeddings in \eqref{eq:reitembedding}.  This is the content of the following proposition, which thus, in particular, implies Example \ref{example:realreit}.
\begin{proposition}\label{proposition:def:class_J_K_seqstruct}
Let $(X_0,X_1,Y)$ be a compatible triple of Banach spaces, let $p,p_0,p_1 \in [1,\infty]$ and $\theta \in  (0,1)$ satisfy $\frac{1}{p}=\frac{1-\theta}{p_0}+\frac{\theta}{p_1}$, let $\mf{S}_j = \ell^{p_j}(\Z;X_j)$ for $j=0,1$ and let $\mf{T} = \ell^p(\Z;Y)$.
Then the following statements hold true for any $a \in (1,\infty)$:
\begin{enumerate}[(i)]
    \item\label{it:ex:def:class_J_K_seqstruct;J} $Y$ is of class $J_\theta$ between $X_0$ and $X_1$ if and only if $$\mf{S}_0(a^{-\theta}) \cap \mf{S}_1(a^{1-\theta}) \hookrightarrow
    \mf{T}.$$
    \item\label{it:ex:def:class_J_K_seqstruct;K} $Y$ is of class $K_\theta$ between $X_0$ and $X_1$ if and only if $$\mf{T}\hookrightarrow \mf{S}_0(a^{-\theta}) + \mf{S}_1(a^{1-\theta}).$$
\end{enumerate}
\end{proposition}

\begin{proof}
We start by proving \ref{it:ex:def:class_J_K_seqstruct;J}. First assume that
$\mf{S}_0(a^{-\theta}) \cap \mf{S}_1(a^{1-\theta}) \hookrightarrow
    \mf{T},$
let $t \in (1,\infty)$ and $x \in X_0 \cap X_1$.
Pick $k \in \Z$ such that $t \in [a^{k},a^{k+1})$ and consider $\vec{x} :=x \otimes \one_{\{k\}} = (\ldots,0,x,0,\ldots) \in \ell^0(\Z;X_0 \cap X_1)$ .
In light of \eqref{eq:sequence1} and \eqref{eq:sequence2} and by choice of $k$, we get
\begin{align*}
\nrm{x}_{Y} &= \nrm{\vec{x}}_{\mf{T}} \lesssim \max_{j=0,1}\,\nrm{\vec{x}}_{\mf{S}_j(a^{j-\theta})} = \max_{j=0,1}\,a^{(j-\theta) k}\nrm{x}_{X_j} \\
&\lesssim  t^{-\theta}J(t,x;X_0,X_1).
\end{align*}
Therefore, $Y$ is of class $J_\theta$ between $X_0$ and $X_1$. Conversely, assume that $Y$ is of class $J_\theta$ between $X_0$ and $X_1$. Taking $t = \nrm{x}_{X_0} /\nrm{x}_{X_1}$ in the definition of $J_\theta$, we see that $$\norm{x}_Y \lesssim \norm{x}^{1-\theta}_{X_0}\norm{x}^{\theta}_{X_1}, \qquad x \in X_0 \cap X_1.$$
Therefore, given $\vec{x} \in \mf{S}_0(a^{-\theta}) \cap \mf{S}_1(a^{1-\theta})$, H\"older's inequality yields that
\begin{align*}
\nrm{\vec{x}}_{\mf{T}}
= \norm{\vec{x}}_{\ell^{p}(\Z;Y)} &\lesssim \norm{(a^{-k\theta }x_k)_{k \in \Z}}_{\ell^{p_0}(\Z;X_0)}^{\theta}\norm{(a^{k(1-\theta)}x_k)_{k \in \Z}}_{\ell^{p_1}(\Z;X_1)}^{1-\theta} \\
&= \norm{\vec{x}}_{\mf{S}_0(a^{-\theta})}^{\theta}\norm{\vec{x}}_{\mf{S}_1(a^{1-\theta})}^{1-\theta}
\lesssim \nrm{\vec{x}}_{\mf{S}_0(a^{-\theta}) \cap \mf{S}_1(a^{1-\theta})},
\end{align*}
which proves that $\mf{S}_0(a^{-\theta}) \cap \mf{S}_1(a^{1-\theta}) \hookrightarrow
    \mf{T}.$

For \ref{it:ex:def:class_J_K_seqstruct;K}, we first assume $\mf{T}\hookrightarrow \mf{S}_0(a^{-\theta}) + \mf{S}_1(a^{1-\theta})$, let $t \in (1,\infty)$ and $x \in Y$. Pick $k \in \Z$ such that $t \in [a^{k},a^{k+1})$ and consider $$\vec{x} :=x \otimes \one_{\{k\}} = (\ldots,0,x,0,\ldots) \in \mf{T}.$$ Let $\vec{x}^j \in \mf{S}_j(a^{j-\theta})$ for $j=0,1$ such that $\vec{x} = \vec{x}^0 +\vec{x}^1$.
By choice of $k$, \eqref{eq:sequence1} and \eqref{eq:sequence2} we have
\begin{align*}
t^{-\theta}K(t,x;X_0,X_1) \leq t^{-\theta}\nrm{x_k^0}_{X_0}+t^{1-\theta}\nrm{x_k^1}_{X_1} \lesssim  \nrm{\vec{x}^0}_{\mf{S}_0(a^{-\theta})}+\nrm{\vec{x}^1}_{\mf{S}_1(a^{1-\theta})}
\end{align*}
Taking the infimum over all such $\vec{x}^j$ shows that $Y$ is of class $K_\theta$. For the converse we will only treat the case $p_0,p_1<\infty$, the cases where $p_0=\infty$ and/or $p_1=\infty$ being similar, but simpler. Take $\tau > 0$ and $y \in Y$. Using the assumption that $Y$ is of class $K_\theta$ between $X_0$ and $X_1$ with $t = \tau \nrm{y}^\beta$, where $\beta := \frac{p_1-p_0}{\theta p_0+(1-\theta)p_1}$, we see that there exists $x_0 \in X_0$ and $x_1 \in X_1$ such that $y=x_0+x_1$ and
\begin{equation*}
\tau^{j-\theta}\nrm{x_j}_{X_j} = t^{j-\theta}\nrm{x_j}_{X_j} \nrm{y}_Y^{-\beta(j-\theta)}\leq C\nrm{y}_{Y}^{1-\beta(j-\theta)} = C \norm{y}_{Y}^{\frac{p}{p_j}},
\end{equation*}
for $j=0,1$. Choosing $\tau=a^{k}$ we find that, given $\vec{y} \in \mf{T}$ with $\nrm{\vec{y}}_{\mf{T}}=1$, there exist $\vec{x}^0 \in \ell^0(\Z;X_0)$ and $\vec{x}^1 \in \ell^0(\Z;X_1)$ with $\vec{y}=\vec{x}^0+\vec{x}^1$ and $$a^{k(j-\theta)}\norm{x^j_{k}}_{X_j}\leq 2C\, \norm{y_k}_Y^{\frac{p}{p_j}}, \qquad j=0,1,$$ for all  $k \in \Z$.
So
\begin{align*}
\nrm{\vec{y}}_{\mf{S}_0(a^{-\theta})+\mf{S}_1(a^{1-\theta})}
&\leq
\nrm{\vec{x}^0}_{\mf{S}_0(a^{-\theta})} + \nrm{\vec{x}^1}_{\mf{S}_1(a^{1-\theta})}
\\&=\sum_{j=0,1}\norm{(a^{k(j-\theta)}x^j_k)}_{\ell^{p_j}(\Z;X_j)} \\
&\lesssim \sum_{j=0,1}\nrm{(y_k)_{k \in \Z}}_{\ell^{p}(\Z;Y)}^{\frac{p}{p_j}} = \sum_{j=0,1}\nrm{\vec{y}}_{\mf{T}}^{\frac{p}{p_j}} = 2,
\end{align*}
proving the embedding $\mf{T}\hookrightarrow \mf{S}_0(a^{-\theta}) + \mf{S}_1(a^{1-\theta}).$
\end{proof}

\begin{remark}
The ``if'' parts of the statements in Proposition \ref{proposition:def:class_J_K_seqstruct} do not use the specific choice of $\mf{S}_0$, $\mf{S}_1$ and $\mf{T}$ and hold for any sequence structures.
\end{remark}

The classes $J_\theta$ and $K_\theta$ can be characterized by embeddings using the real interpolation spaces $(X_0,X_1)_{\theta,1}$ and $(X_0,X_1)_{\theta,\infty}$. In particular, we have that
\begin{equation*}
(X_0,X_1)_{\theta,1} \hookrightarrow  Y \hookrightarrow (X_0,X_1)_{\theta,\infty}.
\end{equation*}
if and only if $Y$ is of class $J_\theta$ and $K_\theta$ between $X_0$ and $X_1$ (see e.g.\ \cite[1.10.1]{Tr78}). Therefore, in view of \eqref{eq:proposition:embeddings;embd_real_int}, we know that any  for any compatible couple of sequentially structured Banach spaces $(\mc{X}_0,\mc{X}_1)$ we have  that $(\mc{X}_0,\mc{X}_1)_\theta$ is of class $J_\theta$ and $K_\theta$ between $X_0$ and $X_1$ for  $\theta \in (0,1)$. Combining this observation with Example \ref{example:realreit}, we obtain:

\begin{example}\label{ex:thm:reiteration;abstract_example}
Let $(X_0,X_1)$ be a compatible couple of Banach spaces. Let $\mf{S}_j$ and $\mf{T}_j$ be sequence structures on $X_j$ and set $\mc{X}_j := [X_j,\mf{S}_j]$ and $\mc{Y}_j:= [X_j,\mf{T}_j]$ for $j=0,1$.
Let $p \in [1,\infty]$, let $0 < \theta_0 < \theta_1 < 1$ and $\theta \in (0,1)$ and set $\omega = (1-\theta)\theta_0 + \theta\theta_1$.
Then we have the reiteration identity
\begin{equation*}
\hab{(\mc{X}_0,\mc{X}_1)_{\theta_0},(\mc{Y}_0,\mc{Y}_1)_{\theta_1}}_{\theta,p}
= (X_0,X_1)_{\omega,p}.
\end{equation*}
\end{example}

\subsection{The complex interpolation method}
We now turn to reiteration for the complex interpolation method. In contrast to the previous subsection, in which we only recovered the already known reiteration theorem for the real interpolation method, we will obtain various new reiteration results for the complex interpolation method. In particular, we will deduce the following example from Theorem \ref{thm:reiteration}:

\begin{example}\label{ex:thm:reiteration_complex}
Let $(\mc{X}_0,\mc{X}_1)$ be a compatible couple of  Ces\`aro bounded  sequentially structured Banach spaces such that $\mf{S}_0$ or $\mf{S}_1$ is Ces\`aro convergent.
Assume that for $j=0,1$
\begin{equation*}
\nrm{(\ee^{iks}x_{k})_{k \in \Z}}_{\mf{S}_j} \lesssim \nrm{\vec{x}}_{\mf{S}_j}, \qquad s \in \R, \,\vec{x} \in \mf{S}_j.
\end{equation*}
Let $0 < \theta_0 < \theta_1 < 1$ and $\theta \in (0,1)$ and set $\omega = (1-\theta)\theta_0 + \theta\theta_1$.
Then we have the reiteration identity
\begin{equation*}
\left[(\mc{X}_0,\mc{X}_1)_{\theta_0},(\mc{X}_0,\mc{X}_1)_{\theta_1}\right]_{\theta} = (\mc{X}_0,\mc{X}_1)_\omega.
\end{equation*}
In particular, we have  the reiteration identities:
\begin{align}
\notag \bracb{[X_0,X_1]_{\theta_0},[X_0,X_1]_{\theta_1}}_{\theta} &= [X_0,X_1]_{\omega}, \\
\notag \bracb{(X_0,X_1)_{\theta_0,\varepsilon},(X_0,X_1)_{\theta_1,\varepsilon}}_{\theta} &= (X_0,X_1)_{\omega,\varepsilon}, \\
\notag \bracb{(X_0,X_1)_{\theta_0,\gamma},(X_0,X_1)_{\theta_1,\gamma}}_{\theta} &= (X_0,X_1)_{\omega,\gamma}, \\
\notag \bracb{(X_0,X_1)_{\theta_0,\alpha},(X_0,X_1)_{\theta_1,\alpha}}_{\theta} &= (X_0,X_1)_{\omega,\alpha},
\end{align}
 where $\alpha$ is a global Euclidean structure.
\end{example}

\begin{remark}\label{rmk:ex:thm:reiteration_complex;concrete}
Concerning the complex interpolation of real interpolation spaces, it is known that
\begin{equation}\label{eq:rmk:ex:thm:reiteration_complex;concrete}
\bracb{(X_0,X_1)_{\theta_0,p_0},(X_0,X_1)_{\theta_1,p_1}}_{\theta} = (X_0,X_1)_{\omega,p}
\end{equation}
is valid for all $p,p_0,p_1 \in [1,\infty]$ with $\frac{1}{p}=\frac{1-\theta}{p_0}+\frac{\theta}{p_1}$ (see e.g.\ \cite[Theorem~4.7.2]{BL76}). However, this result is only contained in Example \ref{ex:thm:reiteration_complex} in the specific case that $p_0=p_1<\infty$. The finiteness assumption is just a technicality caused by our proof, whereas the assumption $p_0=p_1$ is inherent to our approach. Indeed, the assumed embeddings in Theorem \ref{thm:reiteration} only connect $p_0$ to $p$ and $p_1$ to $p$ separately, so they can not encode the relation between $p_0, p_1$ and $p$ required for \eqref{eq:rmk:ex:thm:reiteration_complex;concrete}.
It would be interesting to find a suitable extension of \eqref{eq:rmk:ex:thm:reiteration_complex;concrete} to the abstract setting of sequentially structured interpolation.
\end{remark}

They key to deduce Example \ref{ex:thm:reiteration_complex} from  Theorem~\ref{thm:reiteration} is
the following proposition, in which we prove the  embeddings in \eqref{eq:reitembedding} for the sequence structures associated with the complex interpolation method (see Example~\ref{example:interpolationmethods}\ref{it:ssicomplex}).

\begin{proposition}\label{proposition:complexreiterationabstract}
Let $(\mc{X}_0,\mc{X}_1)$ be a compatible couple of Ces\`aro bounded sequentially structured Banach spaces such that $\mf{S}_0$ or $\mf{S}_1$ is Ces\`aro convergent. Let $\theta \in (0,1)$ and let $a \in (1,\infty)$.
Assume that for $j=0,1$
\begin{equation}\label{ex:prop:class_J_complex_formulation}
\nrm{(\ee^{iks}x_{k})_{k \in \Z}}_{\mf{S}_j} \lesssim \nrm{\vec{x}}_{\mf{S}_j}, \qquad s \in \R,\, \vec{x} \in \mf{S}_j.
\end{equation}
Set $\lambda := \frac{\pi}{\log a} \in (0,\infty)$.
Then we have for $p \in [1,\infty)$
$$
\mf{S}_0(a^{-\theta}) \cap \mf{S}_1(a^{1-\theta})\hookrightarrow \widehat{L}^p(\T_\lambda;(\mc{X}_0,\mc{X}_1)_{\theta;a}) \hookrightarrow \mf{S}_0(a^{-\theta}) + \mf{S}_1(a^{1-\theta}).
$$
\end{proposition}

\begin{proof}
For the first embedding we will use the isometric isomorphism between ${\ms{H}}_\pi(\mathbb{S};\mc{X}_0,\mc{X}_1)$ and $\mf{S}_0(e^{-\theta}) \cap \mf{S}_1(e^{1-\theta})$ from Lemma~\ref{lemma:complexisometric}. Fix $f \in {\ms{H}}_\lambda(\mathbb{S};\mc{X}_0,\mc{X}_1)$ and note that for $t \in \R$ we have
\begin{align*}
\nrm{f(\,\cdot\,+it)}_{{\ms{H}}_\lambda(\mathbb{S};\mc{X}_0,\mc{X}_1)} &=
\max_{j = 0,1}\,\nrmb{\ms{F}_\lambda[ f_\theta(\,\cdot\,+t)]}_{\mf{S}_j(\ee^{j-\theta})}  \\
&= \max_{j = 0,1}\,\nrmb{(a^{ikt}\ms{F}_\lambda f_\theta (k))_{k \in \Z}}_{\mf{S}_j(\ee^{j-\theta})} \\
&\stackrel{\eqref{ex:prop:class_J_complex_formulation}}{\lesssim} \max_{j = 0,1}\,\nrmb{(\ms{F}_\lambda f_\theta (k))_{k \in \Z}}_{\mf{S}_j(\ee^{j-\theta})} \\
&= \nrm{f}_{{\ms{H}}_\lambda(\mathbb{S};\mc{X}_0,\mc{X}_1)}.
\end{align*}
Therefore, using the complex formulation in Proposition \ref{proposition:complexsimple}, we have $f_\theta(t) \in (\mc{X}_0,\mc{X}_1)_{\theta;a}$ with for any $t \in \T_\lambda$
\begin{equation*}
 \nrm{f_\theta(t)}_{(\mc{X}_0,\mc{X}_1)_{\theta;a}}\leq \nrm{f(\,\cdot\,+t)}_{{\ms{H}}_\lambda(\mathbb{S};\mc{X}_0,\mc{X}_1)}  \lesssim \nrm{f}_{{\ms{H}}_\lambda(\mathbb{S};\mc{X}_0,\mc{X}_1)}.
\end{equation*}
Therefore, $f_\theta \in L^p(\T_\lambda;(\mc{X}_0,\mc{X}_1)_{\theta})$ and
  \begin{equation*}
\nrm{f_\theta}_{{L^p}(\T_\lambda;(\mc{X}_0,\mc{X}_1)_{\theta;a})} \lesssim \nrm{f}_{\ms{H}_\lambda(\mathbb{S};\mc{X}_0,\mc{X}_1)}
\end{equation*}
In view of Lemma~\ref{lemma:complexisometric}, this proves the first embedding.

We will deduce the second embedding by duality. By Lemma \ref{lemma:c00dense}, we know that $\mc{X}_0^\circ$ and $\mc{X}_1^\circ$ are Ces\`aro convergent. Thus, by Lemma~\ref{lemma:dualstructure}, we know that $(\mc{X}_0^{\circ})^*$ and $(\mc{X}_1^\circ)^*$ are Ces\`aro bounded  sequence structures, for which one can easily check that for $j=0,1$
\begin{equation*}
\nrm{(\ee^{iks}x_{k}^*)_{k \in \Z}}_{(\mf{S}_j^\circ)^{*}} \lesssim \nrm{\vec{x}^*}_{(\mf{S}_j^\circ)^{*}}, \qquad s \in \R,\, \vec{x}^* \in (\mf{S}_j^\circ)^{*}.
\end{equation*}
Noting that the first part of the proof did not use the Ces\`aro convergence assumption and also works for $a^{-1} \in (0,1)$,  we have by Proposition \ref{proposition:duality}
\begin{equation}\label{eq:dualreitembedding}
\begin{aligned}
    (\mf{S}_0^\circ)^{*}(a^{\theta}) \cap (\mf{S}_1^\circ)^{*}(a^{\theta-1})&\hookrightarrow \widehat{L}^p(\T_\lambda;((\mc{X}_0^{\circ})^*,(\mc{X}_1^{\circ})^*)_{\theta;a})\\
  &\hookrightarrow \widehat{L}^p(\T_\lambda;(\mc{X}_0,\mc{X}_1)_{\theta;a})^*
\end{aligned}
\end{equation}
Take $\vec{x} \in c_{00}(\Z;X_0 \cap X_1)$ and let $\varepsilon > 0$. Using \eqref{eq:proposition:duality;dual_sum} we can find a $\vec{x}^* \in    (\mf{S}_0^\circ)^{*}(a^{\theta}) \cap (\mf{S}_1^\circ)^{*}(a^{\theta-1})$ of norm $1$ such that $$\nrm{\vec{x}}_{\mf{S}_0^\circ(a^{-\theta}) + \mf{S}^{\circ}_1(a^{1-\theta})} \leq |\ip{\vec{x},\vec{x}^*}|+\varepsilon.$$
Then, by the embedding in \eqref{eq:dualreitembedding}, we have
\begin{align*}
\nrm{\vec{x}}_{\mf{S}_0(a^{-\theta}) + \mf{S}_1(a^{1-\theta})} &\leq |\ip{\vec{x},\vec{x}^*}|+\varepsilon \\&\leq \nrm{\vec{x}}_{\widehat{L}^p(\T_\lambda;(\mc{X}_0,\mc{X}_1)_{\theta;a})} \nrm{\vec{x}^*}_{{\widehat{L}^p(\T_\lambda;(\mc{X}_0,\mc{X}_1)_{\theta;a})}^*}+\varepsilon \\&\lesssim \nrm{\vec{x}}_{\widehat{L}^p(\T_\lambda;(\mc{X}_0,\mc{X}_1)_{\theta;a})}+\varepsilon.
\end{align*}
Since $\varepsilon > 0$ was arbitrary and $c_{00}(\Z;X_0 \cap X_1)$ is dense in $\widehat{L}^p(\T_\lambda;(\mc{X}_0,\mc{X}_1)_{\theta;a})$ by Corollary \ref{corollary:denseX0X1}, the second embedding follows.
\end{proof}

Example \ref{ex:thm:reiteration_complex} now follows from Theorem \ref{thm:reiteration} using $a=\ee^{{1}/\ha{\theta_1-\theta_0}}$, $b=e$, $Y_j = (\mc{X}_0,\mc{X}_1)_{\theta_j}$ and
$$
\mf{T}_j = \widehat{L}^p(\T_\lambda;(\mc{X}_0,\mc{X}_1)_{\theta_j})
$$
for $\lambda := \frac{\pi}{\log a}$ and $p \in [1,\infty)$. Indeed, the assumed embeddings in Theorem \ref{thm:reiteration} follow from
Proposition \ref{proposition:complexreiterationabstract} and the conclusion of Theorem \ref{thm:reiteration} is equivalent to the conclusion of Example \ref{ex:thm:reiteration_complex} by Example \ref{ex:changing_base_torus_complex_int}.

\subsection{The $\gamma$-interpolation method}
To conclude our discussion of reiteration, we will discuss Theorem \ref{thm:reiteration} in the setting of the $\gamma$-interpolation method. In particular, will prove the following concrete example. For an introduction to nontrivial type and Pisier's contraction property we refer to \cite[Chapter 7]{HNVW17}.

\begin{example}\label{ex:cor:thm:reiteration}
Let $(X_0,X_1)$ be a compatible couple of Banach spaces with non-trivial type.
Let $0 < \theta_0 < \theta_1 < 1$ and $\theta \in (0,1)$ and set $\omega = (1-\theta)\theta_0 + \theta\theta_1$.
We have the reiteration identity
\begin{equation*}
\left([X_0,X_1]_{\theta_0},[X_0,X_1]_{\theta_1}\right)_{\theta,\gamma} = (X_0,X_1)_{\omega,\gamma}
\end{equation*}
and, if in addition $X_0$ and $X_1$ have Pisier's contraction property, we also have the reiteration identity
\begin{equation*}
\left((X_0,X_1)_{\theta_0,\gamma},(X_0,X_1)_{\theta_1,\gamma}\right)_{\theta,\gamma} = (X_0,X_1)_{\omega,\gamma}.
\end{equation*}
\end{example}

As before, to prove that Example \ref{ex:cor:thm:reiteration} follows from Theorem \ref{thm:reiteration}, we need to study the embeddings in \eqref{eq:reitembedding}.
We start with an interpolation result for weighted sequence structures, which is
an application of the Stein interpolation in Theorem~\ref{theorem:Steininterpolation1}.

\begin{lemma}\label{lemma:steinweightedss}
  Let $(\mc{X}_0,\mc{X}_1)$ be a compatible couple of Ces\`aro convergent sequentially structured Banach spaces, let $\mf{U}_j$ be a Ces\`aro convergent sequence structure on $\mf{S}_j$ and set $\mc{S}_j:=[\mf{S}_j,\mf{U}_j]$ for $j=0,1$. Suppose that
  \begin{equation*}
\nrmb{\big((x_{k,n-k})_{k \in \Z}\big)_{n \in \Z}}_{\mf{U}_j} \lesssim \nrm{\vec{x}}_{\mf{U}_j}, \qquad \vec{x} = \big((x_{k,n})_{k \in \Z}\big)_{n \in \Z} \in \mf{U}_j.
\end{equation*}
for $j=0,1$.   For  $\theta \in (0,1)$ and $a \in (1,\infty)$ we have
\begin{equation*}
 (\mc{S}_0,\mc{S}_1)_{\theta;a} = \big(\mc{S}_0(a^{-\theta}),\mc{S}_1(a^{1-\theta})\big)_{\theta;a},
\end{equation*}
where $\mc{S}_j(a^{j-\theta}) := [\mf{S}_j(a^{j-\theta}),\mf{V}_j]$ with
$$
\mf{V}_j := \left\{ ((x_{k,n})_{k \in \Z})_{n \in \Z} \in \ell^0(\Z;\mf{S}_j(a^{j-\theta})) : ((a^{(j-\theta)k}x_{k,n})_{k \in \Z})_{n \in \Z} \in \mf{U}_j \right\}.
$$
\end{lemma}

\begin{proof}
  Set $\lambda := \frac{\pi}{\log a}$.
Let $\breve{X} := c_{00}(\Z;X_0 \cap X_1)$ and let
$\cbrace{T(z)}_{z \in \overline{\mathbb{S}}}$ be the family of operators $\breve{X} \to \mf{S}_0(a^{-\theta})+\mf{S}_1(a^{1-\theta})$
given by
$$
T(z) \vec{x} :=(a^{k(\theta-z)}x_{k})_{k \in \Z}.
$$
Note that $\breve{X}$ is dense in $\mf{S}_0 \cap \mf{S}_1$ and note that $$z \mapsto T(z) \vec{x} \in \ms{H}_{\lambda}(\overline{\mathbb{S}};\mf{S}_0(a^{-\theta})+\mf{S}_1(a^{1-\theta})), \qquad \vec{x} \in \breve{X}.$$
For $\vec{x} = ((x_{n,k})_{k \in \Z})_{n \in \Z} \in c_{00}(\Z;\breve{X}) =  c_{00}(\Z^2;X_0 \cap X_1)$ we have
\begin{align*}
\ms{F}_\lambda \bracs{t \mapsto \sum_{n \in \Z}a^{int}T\ha{j+it}(x_{k,n})_{k \in \Z} }
&= \ms{F}_\lambda \bracs{ t \mapsto \sum_{n \in \Z}a^{int}(a^{k(\theta-j-it)}x_{k,n})_{k \in \Z}} \nonumber \\
&= \has{ \ms{F}_\lambda \bracb{t \mapsto \sum_{n \in \Z}a^{i(n-k)t}a^{(\theta-j)k}x_{k,n} } }_{k \in \Z} \nonumber \\
&= \ms{F}_\lambda \bracs{ t \mapsto \sum_{n \in \Z}a^{int}(a^{(\theta-j)k}x_{k,n-k})_{k \in \Z} } \nonumber \\
&= \big((a^{(\theta-j)k}x_{k,n-k})_{k \in \Z}\big)_{n \in \Z}, 
\end{align*}
so that
\begin{align*}
\nrms{ \ms{F}_\lambda \bracb{t \mapsto \sum_{n \in \Z} a^{int} T\ha{j+it}(x_{k,n})_{k \in \Z}} }_{\mf{V}_j}
&= \nrmb{\big((x_{k,n-k})_{k \in \Z}\big)_{n \in \Z}}_{\mf{U}_j} \lesssim \nrm{\vec{x}}_{\mf{U}_j}.
\end{align*}
Therefore, we can apply Theorem~\ref{theorem:Steininterpolation1} (see also Remark \ref{remark:stein}\ref{it:steinlambda}) and Corollary \ref{corollary:denseX0X1} to find that $T(\theta)=I$ extends to a bounded operator from $(\mc{S}_0,\mc{S}_1)_{\theta;a}$ to $\big(\mc{S}_0(a^{-\theta}),\mc{S}_1(a^{1-\theta})\big)_{\theta;a}$, which is equivalent to the embedding ``$\hookrightarrow$''.
The reverse embedding can be obtained in the same way.
\end{proof}

Using Lemma \ref{lemma:steinweightedss}, we can prove the embeddings in \eqref{eq:reitembedding} for the sequence structures associated with the $\gamma$-interpolation method introduced in Example~\ref{example:interpolationmethods}\ref{it:ssigaussian}. Example \ref{ex:cor:thm:reiteration} follows by combining the following proposition with Theorem \ref{thm:reiteration} and Proposition \ref{prop:change_basis}.

\begin{proposition}\label{ex:prop:interp_of_seq_structures}
Let $(X_0,X_1)$ be an compatible couple of Banach spaces with non-trivial type, let $\theta \in (0,1)$, $p \in [1,\infty)$ and let $a \in (1,\infty)$. Set $\mf{S}_j=\gamma^p(\Z;X_j)$ for $j=0,1$. Then we have
\begin{align*}
    \mf{S}_0(a^{-\theta}) \cap \mf{S}_1(a^{1-\theta}) \hookrightarrow
    \gamma^p(\Z;[X_0,X_1]_\theta)\hookrightarrow \mf{S}_0(a^{-\theta}) + \mf{S}_1(a^{1-\theta}),
  \intertext{and, if $X_0$ and $X_1$ have Pisier's contraction property, we have}
      \mf{S}_0(a^{-\theta}) \cap \mf{S}_1(a^{1-\theta}) \hookrightarrow
    \gamma^p(\Z;(X_0,X_1)_{\theta,\gamma})\hookrightarrow \mf{S}_0(a^{-\theta}) + \mf{S}_1(a^{1-\theta}).
\end{align*}
\end{proposition}
\begin{proof}
Set $\lambda := \frac{\pi}{\log a}$. For the first chain of embeddings, let  $\mf{U}_j := \widehat{L}^2(\T_\lambda;\mf{S}_j)$ and set $\mc{S}_j:= [\mf{S}_j,\mf{U}_j].$
By \cite[Proposition~7.1.3]{HNVW17}, we know that $[X_0,X_1]_\theta$ has non-trivial type. It thus follows from \cite[Theorem~7.1.14 and Theorem~7.4.23]{HNVW17} that $X_0$, $X_1$ and $[X_0,X_1]_\theta$ have finite cotype and are $K$-convex. We can therefore apply \cite[Corollary~7.2.10 and Theorem~7.4.16]{HNVW17} in combination with Example~\ref{ex:changing_base_torus_complex_int} to obtain that
\begin{equation*}
\begin{aligned}
(\mc{S}_0,\mc{S}_1)_{\theta;a} = \bracb{\gamma^p(\Z;X_0),\gamma^p(\Z;X_1)}_\theta &=
\bracb{\varepsilon^p(\Z;X_0),\varepsilon^p(\Z;X_1)}_\theta \\
&= \varepsilon^p(\Z;[X_0,X_1]_\theta) = \gamma^p(\Z;[X_0,X_1]_\theta).
\end{aligned}
\end{equation*}
Using Kahane's contraction principle (see \cite[Theorem~6.1.13]{HNVW17}) we find that, for all $\vec{x}=\big((x_{k,n})_{k \in \Z}\big)_{n \in \Z} \in \mf{U}_j$,
\begin{align*}
\nrmb{\big((x_{k,n-k})_{k \in \Z}\big)_{n \in \Z}}_{\mf{U}_j}
&= \nrmb{t \mapsto \big(a^{-ikt}\ms{F}_\lambda[(x_{k,n})_{n \in \Z}](t)\big)_{k \in \Z}}_{L^2(\T_\lambda;\gamma^p(\Z;X_j))} \\
&\leq \nrmb{t \mapsto \big(\ms{F}_\lambda[(x_{k,n})_{n \in \Z}](t)\big)_{k \in \Z}}_{L^2(\T_\lambda;\gamma^p(\Z;X_j))} = \nrm{\vec{x}}_{\mf{U}_j}.
\end{align*}
Therefore, by Lemma \ref{lemma:steinweightedss}, we have
$$
 (\mc{S}_0,\mc{S}_1)_{\theta;a}  = \big(\mc{S}_0(a^{-\theta}),\mc{S}_1(a^{1-\theta})\big)_{\theta;a}.
$$
Since, by Proposition~\ref{proposition:Banachembeddings}, we have
$$
    \mf{S}_0(a^{-\theta}) \cap \mf{S}_1(a^{1-\theta}) \hookrightarrow
     \big(\mc{S}_0(a^{-\theta}),\mc{S}_1(a^{1-\theta})\big)_{\theta;a} \hookrightarrow \mf{S}_0(a^{-\theta}) + \mf{S}_1(a^{1-\theta}),
$$
 this proves first chain of embeddings.

 The proof of the second chain of embeddings is analogous, using $\mf{U}_j := \gamma^p(\Z;\mf{S}_j)$ for $j=0,1$. Indeed, using \cite[Corollary~3.2 and Corollary~3.3]{SW06}, \cite[Corollary~7.2.10]{HNVW17} and Proposition~\ref{prop:change_basis} we find that
\begin{equation*}
\begin{aligned}
(\mc{S}_0,\mc{S}_1)_{\theta;a} &= \hab{\gamma^p(\Z;X_0),\gamma^p(\Z;X_1)}_{\theta,\gamma} =
\hab{\varepsilon^p(\Z;X_0),\varepsilon^p(\Z;X_1)}_{\theta,\gamma} \\
&= \varepsilon^p(\Z;(X_0,X_1)_{\theta,\gamma}) = \gamma^p(\Z;(X_0,X_1)_{\theta,\gamma}).
\end{aligned}
\end{equation*}
As $X_j$ has Pisier's contraction property, we have $\mf{U}_j =\gamma^p(\Z^2;X_j)$ (see \cite[Corollary~7.5.19]{HNVW17}), which yields that
  \begin{equation*}
\nrmb{\big((x_{k,n-k})_{k \in \Z}\big)_{n \in \Z}}_{\mf{U}_j} \lesssim \nrm{\vec{x}}_{\mf{U}_j}, \qquad \vec{x} = \big((x_{k,n})_{k \in \Z}\big)_{n \in \Z} \in \mf{U}_j.
\end{equation*}
The proof is finished using Lemma \ref{lemma:steinweightedss} and Proposition~\ref{proposition:Banachembeddings} as before.
\end{proof}

\section{Interpolation of intersections}\label{section:intersections}

In this final section we will study intersection representations for the sequentially structured interpolation method.
These results can be used to derive several intersection representations for anisotropic mixed-norm function spaces.
In future work we will for instance see that the intersection representation for anisotropic mixed-norm Triebel-Lizorkin spaces (cf.\ \eqref{eq:intro:space_BV})
\begin{equation*}
F^{(s,t)}_{(p,q),q}(\R^{d} \times \R^{k}) = F^{s}_{p,q}(\R^{d};L^q(\R^{k}) \cap L^p(\R^{d};F^{t}_{q,q}(\R^{k}))
\end{equation*}
from \cite{Li21} can be obtained from the
elementary  intersection representation for anisotropic mixed-norm Sobolev spaces
\begin{equation*}
W^{(m,n)}_{(p,q)}(\R^{d} \times \R^{k})
= W^{m}_{p}(\R^d;L^q(\R^k)) \cap L^{p}(\R^d;W^{n}_{q}(\R^k))
\end{equation*}
by means of $\ell^q$-interpolation.
Besides that this yields a tremendously simplified proof, it will also provide us valuable insight in the trace theory behind the maximal $L^p$-$L^q$-regularity approach to parabolic boundary value problems.

The following intersection representation is an extension of a result by Peetre \cite[Korollar~1.1]{Pe74} (also see \cite[1.12.1 and 1.12.2]{Tr78}) on the intersection of real interpolation spaces to the setting of sequentially structured interpolation. Other results on interpolation of intersections as well as sums can be found in \cite{Ba13,EPS03,Gr72a,Gr72b,Ha06c,Ma84,Ma86}.

\begin{theorem}\label{thm:interpol_intersection}
Let $\mc{X}=[X,\mf{S}]$, $\mc{Y}=[Y,\mf{T}]$ and $\mc{Z}=[Z,\mf{U}]$ be sequentially structured Banach spaces such that $(X,Y,Z)$ is a compatible triple of Banach spaces.
Assume that for each $k \in \Z$ there exist linear operators $S_k:X+Y \to X$ and $T_k:X+Y+Z \to Y+Z$ such that:
\begin{align}
\label{eq:thm:interpol_intersection;1}
S_k+T_k &= I_{X+Y}, && k \in \Z,\\
\label{eq:thm:interpol_intersection;2}   \nrm{(S_k v_{k})_{k \in \Z}}_{\mf{S}} &\lesssim \nrm{\vec{v}}_{\mf{S}+\mf{T}(\ee)}, && \vec{v} \in \mf{S}+\mf{T}(\ee),\\
\label{eq:thm:interpol_intersection;3}   \nrm{(T_k v_{k})_{k \in \Z}}_{\mf{T}(\ee)} &\lesssim \nrm{\vec{v}}_{\mf{S}+\mf{T}(\ee)}, && \vec{v} \in \mf{S}+\mf{T}(\ee),\\
\label{eq:thm:interpol_intersection;4}   \nrm{(T_k z_{k})_{k \in \Z}}_{\mf{U}} &\lesssim \nrm{\vec{z}}_{\mf{U}}, && \vec{z} \in \mf{U}.
\end{align}
Then, for all $\theta \in (0,1)$, we have
\begin{equation}\label{eq:thm:interpol_intersection;statement}
(\mc{X},\mc{Y})_\theta \cap (\mc{X},\mc{Z})_\theta = (\mc{X},\mc{Y} \cap \mc{Z})_\theta,
\end{equation}
where $\mc{Y} \cap \mc{Z} = [Y \cap Z,\mf{T} \cap \mf{U}]$.
\end{theorem}

\begin{remark}\label{rmk:thm:interpol_intersection}
The existence of linear operators $S_k:X+Y \to X$ and $T_k:X+Y \to Y$ for each $k \in \Z$ such that \eqref{eq:thm:interpol_intersection;1}, \eqref{eq:thm:interpol_intersection;2} and \eqref{eq:thm:interpol_intersection;3} are satisfied would be the natural way to define the quasilinearizability of $(\mc{X},\mc{Y})$, extending the notion of a quasilinearizable compatible couple of Banach spaces (see e.g.\ \cite[1.8.4]{Tr78}) to the setting of compatible couples of sequentially structured Banach spaces. Note that \eqref{eq:thm:interpol_intersection;2} and \eqref{eq:thm:interpol_intersection;3} hold true if and only if
\begin{align}
\label{eq:rmk:thm:interpol_intersection;1}
\nrm{(S_k)_{k \in \Z}}_{\mf{S} \to \mf{S}} < \infty, &\quad \nrm{(\ee^{-k}S_k)_{k \in \Z}}_{\mf{T} \to \mf{S}} < \infty, \\
\label{eq:rmk:thm:interpol_intersection;2}
\nrm{(\ee^k T_k)_{k \in \Z}}_{\mf{S} \to \mf{T}} < \infty, &\quad \nrm{(T_k)_{k \in \Z}}_{\mf{T} \to \mf{T}} < \infty,
\end{align}
hold true (cf.\ \cite[Definition~4.1]{Ku15}).
\end{remark}

\begin{proof}[Proof of Theorem~\ref{thm:interpol_intersection}]
It will be convenient to define the linear operators $S: \ell^0(\Z;X+Y) \to \ell^0(\Z;X)$ and $T:\ell^0(\Z;X+Y+Z) \to \ell^0(\Z;Y+Z)$ by $S\vec{v}:=(S_k v_k)_{k \in Z}$ and $T\vec{v}:=(T_k v_k)_{k \in Z}$.

Note that "$\hookleftarrow$" in \eqref{eq:thm:interpol_intersection;statement} follows by two applications of Proposition~\ref{proposition:embeddings}.
In order to establish the converse, let $v \in (\mc{X},\mc{Y})_\theta \cap (\mc{X},\mc{Z})_\theta \subseteq X + Y$.
Then, in particular, $v \in (\mc{X},\mc{Z})_\theta$, so that $\vec{v}:=(\ldots,v,v,v,\ldots) \in \mf{S}(\ee^{-\theta})+\mf{U}(\ee^{1-\theta})$ by Theorem~\ref{theorem:meanmethod}.
Therefore, there exist $\vec{x} \in \mf{S}(\ee^{-\theta})$ and $\vec{z} \in \mf{U}(\ee^{1-\theta})$ such that
$\vec{v}=\vec{x}+\vec{z}$ and
\begin{equation}\label{eq:thm:interpol_intersection;proof;1}
\nrm{\vec{x}}_{\mf{S}(\ee^{-\theta})}+\nrm{\vec{z}}_{\mf{U}(\ee^{1-\theta})}
\leq 2\nrm{\vec{v}}_{\mf{S}(\ee^{-\theta})+\mf{U}(\ee^{1-\theta})}.
\end{equation}
Defining
\begin{align*}
\vec{a} &:= \vec{v}-T\vec{z} \stackrel{\eqref{eq:thm:interpol_intersection;1}}{=}  S\vec{v}+T\vec{x} \stackrel{\eqref{eq:thm:interpol_intersection;1}}{=} S\vec{v}+\vec{x}-S\vec{x},\\
\vec{b} &:= T\vec{z} = T\vec{v}-T\vec{x},
\end{align*}
we have $\vec{v}=\vec{a}+\vec{b}$.
Furthermore, we have the following estimates:
\begin{align*}
\nrm{\vec{a}}_{\mf{S}(\ee^{-\theta})}
&\leq \nrm{S\vec{v}}_{\mf{S}(\ee^{-\theta})}   + \nrm{\vec{x}}_{\mf{S}(\ee^{-\theta})} + \nrm{S\vec{x}}_{\mf{S}(\ee^{-\theta})} \\
&\stackrel{\eqref{eq:thm:interpol_intersection;2}}{\lesssim} \nrm{\vec{v}}_{\mf{S}(\ee^{-\theta})+\mf{T}(\ee^{1-\theta})}   + \nrm{\vec{x}}_{\mf{S}(\ee^{-\theta})} + \nrm{S\vec{x}}_{\mf{S}(\ee^{-\theta})} \\
&\stackrel{\eqref{eq:rmk:thm:interpol_intersection;1}}{\lesssim} \nrm{\vec{v}}_{\mf{S}(\ee^{-\theta})+\mf{T}(\ee^{1-\theta})}   + \nrm{\vec{x}}_{\mf{S}(\ee^{-\theta})} \\ &\stackrel{\eqref{eq:thm:interpol_intersection;proof;1}}{\lesssim}\max\big\{\nrm{\vec{v}}_{\mf{S}(\ee^{-\theta})+\mf{T}(\ee^{1-\theta})},\nrm{\vec{v}}_{\mf{S}(\ee^{-\theta})+\mf{U}(\ee^{1-\theta})}\big\}, \intertext{and}
\nrm{\vec{b}}_{\mf{T}(\ee^{1-\theta})}
&\leq \nrm{T\vec{v}}_{\mf{T}(\ee^{1-\theta})} + \nrm{T\vec{x}}_{\mf{T}(\ee^{1-\theta})}  \\
&\stackrel{\eqref{eq:thm:interpol_intersection;3}}{\lesssim} \nrm{\vec{v}}_{\mf{S}(\ee^{-\theta})+\mf{T}(\ee^{1-\theta})} + \nrm{T\vec{x}}_{\mf{T}(\ee^{1-\theta})} \\
&\stackrel{\eqref{eq:rmk:thm:interpol_intersection;2}}{\lesssim} \nrm{\vec{v}}_{\mf{S}(\ee^{-\theta})+\mf{T}(\ee^{1-\theta})} + \nrm{\vec{x}}_{\mf{S}(\ee^{-\theta})} \\ &\stackrel{\eqref{eq:thm:interpol_intersection;proof;1}}{\lesssim}\max \big\{\nrm{\vec{v}}_{\mf{S}(\ee^{-\theta})+\mf{T}(\ee^{1-\theta})}, \nrm{\vec{v}}_{\mf{S}(\ee^{-\theta})+\mf{U}(\ee^{1-\theta})}\big\},
\intertext{and finally}
\nrm{\vec{b}}_{\mf{U}(\ee^{1-\theta})}
&=\nrm{T\vec{z}}_{\mf{U}(\ee^{1-\theta})}
\stackrel{\eqref{eq:thm:interpol_intersection;4}}{\lesssim} \nrm{\vec{z}}_{\mf{U}(\ee^{1-\theta})}
\stackrel{\eqref{eq:thm:interpol_intersection;proof;1}}{\leq} 2\nrm{\vec{v}}_{\mf{S}(\ee^{-\theta})+\mf{U}(\ee^{1-\theta})}.
\end{align*}
Therefore,
\begin{align*}
\nrm{\vec{v}}_{\mf{S}(\ee^{-\theta})+[\mf{T}\cap \mf{U}](\ee^{1-\theta})}
&\leq \nrm{\vec{a}}_{\mf{S}(\ee^{-\theta})} + \nrm{\vec{b}}_{[\mf{T}\cap \mf{U}](\ee^{1-\theta})} \\
&= \nrm{\vec{a}}_{\mf{S}(\ee^{-\theta})} + \max\{\nrm{\vec{b}}_{\mf{T}(\ee^{1-\theta})},\nrm{\vec{b}}_{\mf{U}(\ee^{1-\theta})}\}  \\
&\lesssim \max\big\{\nrm{\vec{v}}_{\mf{S}(\ee^{-\theta})+\mf{T}(\ee^{1-\theta})},\nrm{\vec{v}}_{\mf{S}(\ee^{-\theta})+\mf{U}(\ee^{1-\theta})}\big\}.
\end{align*}
By Theorem~\ref{theorem:meanmethod} we thus obtain that
\begin{equation*}
\nrm{v}_{(\mc{X},\mc{Y} \cap \mc{Z})_\theta} \lesssim \max\big\{\nrm{v}_{(\mc{X},\mc{Y})_\theta},\nrm{v}_{(\mc{X},\mc{Z})_\theta}\big\},
\end{equation*}
finishing the proof.
\end{proof}

As in \cite{Pe74}, we obtain the following corollary in the special case that $\mc{Y}$ is the sequentially structured Banach space induced from $\mc{X}$ by an operator $A$ on $X$, which is an extension of a classical result by Grisvard \cite{Gr72a} in the setting of real interpolation.

\begin{corollary}\label{cor:thm:interpol_intersection;operators}
Let $(\mc{X},\mc{Z})$ be a compatible couple of sequentially structured Banach and let $A$ be a closed, injective linear operator on $X$ with $\{\ee^{k} : k \in \Z\} \subseteq \rho(-A)$ such that $(\ee^{k}+A)^{-1}$ has an extension to a linear operator on $X+Z$ for which $Z$ is an invariant subspace for every $k \in \Z$.
Assume that $\big(e^{k}(e^{k}+A)^{-1}\big)_{k \in \Z}$ is both $\mf{S}$-bounded and $\mf{U}$-bounded.
Let $\mc{Y} = [\dot{\mrm{D}}(A),\mf{S}_A]$ with $$\mf{S}_A = \left\{ \vec{y} \in \ell^0(\Z;\dot{\mrm{D}}(A)) : (Ay_k)_{k \in \Z} \in \mf{S} \right\}.$$
Then, for all $\theta \in (0,1)$, we have the intersection representation
\eqref{eq:thm:interpol_intersection;statement}.
\end{corollary}
\begin{proof}
For $k \in \Z$ we define the linear operators $S_k:X+Y \to X$ and $T_k:X+Y+Z \to Y+Z$ by
\begin{align*}
S_k &:= I-\ee^{k}(\ee^{k}+A)^{-1} = A(\ee^{k}+A)^{-1}, \\
T_k &:= \ee^{k}(\ee^{k}+A)^{-1}.
\end{align*}
Then we clearly have \eqref{eq:thm:interpol_intersection;1}.
From the $\mf{S}$-boundedness of $\big(e^{k}(e^{k}+A)^{-1}\big)_{k \in \Z}$
it follows that \eqref{eq:rmk:thm:interpol_intersection;1} and \eqref{eq:rmk:thm:interpol_intersection;2} are satisfied, and thus by Remark~\ref{rmk:thm:interpol_intersection}, that \eqref{eq:thm:interpol_intersection;2} and \eqref{eq:thm:interpol_intersection;3} are satisfied. Furthermore, the $\mf{U}$-boundedness of $\big(e^{k}(e^{k}+A)^{-1}\big)_{k \in \Z}$ yields that \eqref{eq:thm:interpol_intersection;4} is satisfied.
We can thus apply Theorem~\ref{thm:interpol_intersection} to obtain that there is the intersection representation
\eqref{eq:thm:interpol_intersection;statement} for every $\theta \in (0,1)$.
\end{proof}

In \cite[Section~3]{Pe74} it is posed as an open problem to obtain a version of \cite[Korollar~1.1]{Pe74} for the complex interpolation functor.
On an abstract level, Theorem~\ref{thm:interpol_intersection} provides a solution to this problem in the full generality of sequentially structured interpolation, which contains the complex interpolation functor as a special case by Example~\ref{example:interpolationmethods}\ref{it:ssicomplex}.
However, it remains the question whether Theorem~\ref{thm:interpol_intersection} or Corollary~\ref{cor:thm:interpol_intersection;operators} has any practical relevance in this specific case.
Is it for instance possible to derive \cite[Lemma~9.5]{EPS03} from Corollary~\ref{cor:thm:interpol_intersection;operators}?

\bibliographystyle{alpha}
\bibliography{seqstructbib}

\newcommand{\etalchar}[1]{$^{#1}$}
\begin{thebibliography}{HNVW23}

\bibitem[AHL{\etalchar{+}}02]{AHLT02}
P.~Auscher, S.~Hofmann, M.~Lacey, A.~McIntosh, and P.~Tchamitchian.
\newblock The solution of the {K}ato square root problem for second order
  elliptic operators on {${\mathbb{R}}^n$}.
\newblock {\em Ann. of Math. (2)}, 156(2):633--654, 2002.

\bibitem[AKM23]{AKM20}
I.~Asekritova, N.~Kruglyak, and M.~Masty{\l}o.
\newblock Stability of the inverses of interpolated operators with application
  to the {S}tokes system.
\newblock {\em Rev. Mat. Complut.}, 36(1):163--206, 2023.

\bibitem[AV22a]{AV20a}
A.~Agresti and M.~Veraar.
\newblock Nonlinear parabolic stochastic evolution equations in critical spaces
  part {I}. {S}tochastic maximal regularity and local existence.
\newblock {\em Nonlinearity}, 35(8):4100--4210, 2022.

\bibitem[AV22b]{AV20b}
A.~Agresti and M.~Veraar.
\newblock Nonlinear parabolic stochastic evolution equations in critical spaces
  part {II}: {B}low-up criteria and instantaneous regularization.
\newblock {\em J. Evol. Equ.}, 22(2):Paper No. 56, 96, 2022.

\bibitem[Bag13]{Ba13}
A.~G. Bagdasaryan.
\newblock Some problems on the commutativity of interpolation functors.
\newblock {\em Izv. Nats. Akad. Nauk Armenii Mat.}, 48(3):3--24, 2013.

\bibitem[BGH19]{BGH19}
B.~{Barraza Martínez}, J.~{González Ospino}, and J.~{Hernández Monzón}.
\newblock {\em Vector-valued function and distribution spaces on the torus}.
\newblock Editorial Universidad del Norte, 1 edition, 2019.

\bibitem[BK91]{BK91}
Yu.A. Brudny\u{\i} and N.Ya. Krugljak.
\newblock {\em Interpolation functors and interpolation spaces. {V}ol. {I}},
  volume~47 of {\em North-Holland Mathematical Library}.
\newblock North-Holland Publishing Co., Amsterdam, 1991.

\bibitem[BL76]{BL76}
J.~Bergh and J.~L\"{o}fstr\"{o}m.
\newblock {\em Interpolation spaces. {A}n introduction}.
\newblock Springer-Verlag, Berlin-New York, 1976.
\newblock Grundlehren der Mathematischen Wissenschaften, No. 223.

\bibitem[BS88]{BS88}
C.~Bennett and R.~Sharpley.
\newblock {\em Interpolation of operators}, volume 129 of {\em Pure and Applied
  Mathematics}.
\newblock Academic Press, Inc., Boston, MA, 1988.

\bibitem[Cal64]{Ca64}
A.P. Calder{\'o}n.
\newblock Intermediate spaces and interpolation, the complex method.
\newblock {\em Studia Math.}, 24:113--190, 1964.

\bibitem[CDMY96]{CDMY96}
M.~Cowling, I.~Doust, A.~McIntosh, and A.~Yagi.
\newblock Banach space operators with a bounded {$H^\infty$} functional
  calculus.
\newblock {\em J. Austral. Math. Soc. Ser. A}, 60(1):51--89, 1996.

\bibitem[CJ84]{CJ84}
M.~Cwikel and S.~Janson.
\newblock Interpolation of analytic families of operators.
\newblock {\em Studia Math.}, 79(1):61--71, 1984.

\bibitem[CJMR89]{CJMR89}
M.~Cwikel, B.~Jawerth, M.~Milman, and R.~Rochberg.
\newblock Differential estimates and commutators in interpolation theory.
\newblock In {\em Analysis at {U}rbana, {V}ol. {II} ({U}rbana, {IL},
  1986--1987)}, volume 138 of {\em London Math. Soc. Lecture Note Ser.}, pages
  170--220. Cambridge Univ. Press, Cambridge, 1989.

\bibitem[CK95]{CK95}
M.~Cwikel and N.J. Kalton.
\newblock Interpolation of compact operators by the methods of {C}alder\'{o}n
  and {G}ustavsson-{P}eetre.
\newblock {\em Proc. Edinburgh Math. Soc. (2)}, 38(2):261--276, 1995.

\bibitem[CK17]{CK17}
R.~Chill and S.~Kr\'{o}l.
\newblock Real interpolation with weighted rearrangement invariant {B}anach
  function spaces.
\newblock {\em J. Evol. Equ.}, 17(1):173--195, 2017.

\bibitem[CKMR02]{CKMR02}
M.~Cwikel, N.~Kalton, M.~Milman, and R.~Rochberg.
\newblock A unified theory of commutator estimates for a class of interpolation
  methods.
\newblock {\em Adv. Math.}, 169(2):241--312, 2002.

\bibitem[CKS92]{CBS92}
F.~Cobos, T.~K\"{u}hn, and T.~Schonbek.
\newblock One-sided compactness results for {A}ronszajn-{G}agliardo functors.
\newblock {\em J. Funct. Anal.}, 106(2):274--313, 1992.

\bibitem[Con90]{Co90}
J.B. Conway.
\newblock {\em A course in functional analysis}, volume~96 of {\em Graduate
  Texts in Mathematics}.
\newblock Springer-Verlag, New York, second edition, 1990.

\bibitem[CS88]{CS88}
M.~Cwikel and Y.~Sagher.
\newblock Analytic families of operators on some quasi-{B}anach spaces.
\newblock {\em Proc. Amer. Math. Soc.}, 102(4):979--984, 1988.

\bibitem[Cwi78]{Cw78}
M.~Cwikel.
\newblock Complex interpolation spaces, a discrete definition and reiteration.
\newblock {\em Indiana Univ. Math. J.}, 27(6):1005--1009, 1978.

\bibitem[Cwi92]{Cw92}
Michael Cwikel.
\newblock Real and complex interpolation and extrapolation of compact
  operators.
\newblock {\em Duke Math. J.}, 65(2):333--343, 1992.

\bibitem[DHP07]{DHP07}
R.~Denk, M.~Hieber, and J.~Pr\"uss.
\newblock Optimal {$L^p$}-{$L^q$}-estimates for parabolic boundary value
  problems with inhomogeneous data.
\newblock {\em Math. Z.}, 257(1):193--224, 2007.

\bibitem[DK10]{DK10}
J.J. Duistermaat and J.~A.C. Kolk.
\newblock {\em Distributions}.
\newblock Cornerstones. Birkh\"{a}user Boston, Inc., Boston, MA, 2010.

\bibitem[EPS03]{EPS03}
J.~Escher, J.~Pr\"{u}ss, and G.~Simonett.
\newblock Analytic solutions for a {S}tefan problem with {G}ibbs-{T}homson
  correction.
\newblock {\em J. Reine Angew. Math.}, 563:1--52, 2003.

\bibitem[GKKT98]{GKKKT98}
J.~{Garcia-Cuerva}, K.S. Kazarian, V.I. Kolyada, and J.L. Torrea.
\newblock The {H}ausdorff-{Y}oung inequality with vector-valued coefficients
  and applications.
\newblock {\em Uspekhi Mat. Nauk}, 53(3(321)):3--84, 1998.

\bibitem[Gra14]{Gr14a}
L.~Grafakos.
\newblock {\em Classical {F}ourier analysis}, volume 249 of {\em Graduate Texts
  in Mathematics}.
\newblock Springer, New York, third edition, 2014.

\bibitem[Gri72a]{Gr72a}
P.~Grisvard.
\newblock Interpolation non commutative.
\newblock {\em Atti Accad. Naz. Lincei Rend. Cl. Sci. Fis. Mat. Nat. (8)},
  52:11--15, 1972.

\bibitem[Gri72b]{Gr72b}
P.~Grisvard.
\newblock Spazi di tracce e applicazioni.
\newblock {\em Rend. Mat. (6)}, 5:657--729, 1972.

\bibitem[Haa06]{Ha06c}
M.H.A. Haase.
\newblock Identification of some real interpolation spaces.
\newblock {\em Proc. Amer. Math. Soc.}, 134(8):2349--2358, 2006.

\bibitem[HL21]{HL19}
F.~{Hummel} and N.~{Lindemulder}.
\newblock {Elliptic and Parabolic Boundary Value Problems in Weighted Function
  Spaces}.
\newblock {\em Potential Anal.}, 2021.

\bibitem[HNVW16]{HNVW16}
T.P. Hyt\"{o}nen, J.M.A.M.~van Neerven, M.C. Veraar, and L.~Weis.
\newblock {\em Analysis in {B}anach spaces. {V}ol. {I}. {M}artingales and
  {L}ittlewood-{P}aley theory}, volume~63 of {\em Ergebnisse der Mathematik und
  ihrer Grenzgebiete.}
\newblock Springer, Cham, 2016.

\bibitem[HNVW17]{HNVW17}
T.P. Hyt\"{o}nen, J.M.A.M.~van Neerven, M.C. Veraar, and L.~Weis.
\newblock {\em Analysis in {B}anach spaces. {V}ol. {II}. Probabilistic methods
  and operator theory}, volume~67 of {\em Ergebnisse der Mathematik und ihrer
  Grenzgebiete.}
\newblock Springer, Cham, 2017.

\bibitem[HNVW23]{HNVW23}
T.P. Hyt\"onen, J.M.A.M.~van Neerven, M.C. Veraar, and L.~Weis.
\newblock {\em Analysis in {B}anach spaces. {V}olume {III}: Harmonic Analysis
  and Spectral Theory}, volume~76 of {\em Ergebnisse der Mathematik und ihrer
  Grenzgebiete.}
\newblock Springer, Cham, 2023.

\bibitem[Ivt12]{Iv12}
A.~Ivtsan.
\newblock Stafney's lemma holds for several ``classical'' interpolation
  methods.
\newblock {\em Proc. Amer. Math. Soc.}, 140(3):881--889, 2012.

\bibitem[Jan81]{Ja81b}
S.~Janson.
\newblock Minimal and maximal methods of interpolation.
\newblock {\em J. Functional Analysis}, 44(1):50--73, 1981.

\bibitem[JNP84]{JNP84}
S.~Janson, P.~Nilsson, and J.~Peetre.
\newblock Notes on {W}olff's note on interpolation spaces.
\newblock {\em Proc. London Math. Soc. (3)}, 48(2):283--299, 1984.
\newblock With an appendix by Misha Zafran.

\bibitem[Kal16]{Ka16b}
N.J. Kalton.
\newblock {\em Nigel {J}. {K}alton selecta. {V}ol. 2}.
\newblock Contemporary Mathematicians. Birkh\"{a}user/Springer, Cham, 2016.

\bibitem[KKW06]{KKW06}
N.~Kalton, P.C. Kunstmann, and L.~Weis.
\newblock Perturbation and interpolation theorems for the {$H^\infty$}-calculus
  with applications to differential operators.
\newblock {\em Math. Ann.}, 336(4):747--801, 2006.

\bibitem[KLW23]{KLW19}
N.J. Kalton, E.~Lorist, and L.~Weis.
\newblock Euclidean structures and operator theory in {B}anach spaces.
\newblock {\em Mem. Amer. Math. Soc.}, 288(1433):vi+156, 2023.

\bibitem[KU14]{KU14}
P.C. Kunstmann and A.~Ullmann.
\newblock {$\mathcal{R}_s$}-sectorial operators and generalized
  {T}riebel-{L}izorkin spaces.
\newblock {\em J. Fourier Anal. Appl.}, 20(1):135--185, 2014.

\bibitem[Kun15]{Ku15}
P.C. Kunstmann.
\newblock A new interpolation approach to spaces of {T}riebel-{L}izorkin type.
\newblock {\em Illinois J. Math.}, 59(1):1--19, 2015.

\bibitem[KW17]{KW17}
P.C. Kunstmann and L.~Weis.
\newblock New criteria for the {$H^\infty$}-calculus and the {S}tokes operator
  on bounded {L}ipschitz domains.
\newblock {\em J. Evol. Equ.}, 17(1):387--409, 2017.

\bibitem[Lin20]{Li17b}
N.~Lindemulder.
\newblock Maximal regularity with weights for parabolic problems with
  inhomogeneous boundary conditions.
\newblock {\em J. Evol. Equ.}, 20(1):59--108, 2020.

\bibitem[Lin21]{Li21}
N.~Lindemulder.
\newblock An intersection representation for a class of anisotropic
  vector-valued function spaces.
\newblock {\em J. Approx. Theory}, 264:105519, 2021.

\bibitem[LL22]{LL21c}
N.~Lindemulder and E.~Lorist.
\newblock Stein interpolation for the real interpolation method.
\newblock {\em Banach J. Math. Anal.}, 16(1):Paper No. 7, 18, 2022.

\bibitem[LN23]{LN23}
E.~Lorist and Z.~Nieraeth.
\newblock Banach function spaces done right.
\newblock {\em Nederl. Akad. Wetensch. Indag. Math.}, Online first, 2023.

\bibitem[Loz69]{Lo69}
G.Ya. Lozanovskii.
\newblock On some {B}anach lattices.
\newblock {\em Siberian Mathematical Journal}, 10(3):419--431, 1969.

\bibitem[LP64]{LP64}
J.-L. Lions and J.~Peetre.
\newblock Sur une classe d'espaces d'interpolation.
\newblock {\em Inst. Hautes \'{E}tudes Sci. Publ. Math.}, 19(1):5--68, 1964.

\bibitem[LT79]{LT79}
J.~Lindenstrauss and L.~Tzafriri.
\newblock {\em Classical {B}anach spaces. {II}}, volume~97 of {\em Ergebnisse
  der Mathematik und ihrer Grenzgebiete}.
\newblock Springer-Verlag, Berlin-New York, 1979.

\bibitem[LV20]{LV20}
N.~Lindemulder and M.C. Veraar.
\newblock The heat equation with rough boundary conditions and holomorphic
  functional calculus.
\newblock {\em J. Differential Equations}, 269(7):5832--5899, 2020.

\bibitem[Mal84]{Ma84}
L.~Maligranda.
\newblock The {$K$}-functional for symmetric spaces.
\newblock In {\em Interpolation spaces and allied topics in analysis ({L}und,
  1983)}, volume 1070 of {\em Lecture Notes in Math.}, pages 169--182.
  Springer, Berlin, 1984.

\bibitem[Mal86]{Ma86}
L.~Maligranda.
\newblock Interpolation between sum and intersection of {B}anach spaces.
\newblock {\em J. Approx. Theory}, 47(1):42--53, 1986.

\bibitem[McI86]{Mc86}
A.~McIntosh.
\newblock Operators which have an {$H_\infty$} functional calculus.
\newblock In {\em Miniconference on operator theory and partial differential
  equations ({N}orth {R}yde, 1986)}, volume~14 of {\em Proc. Centre Math. Anal.
  Austral. Nat. Univ.}, pages 210--231. Austral. Nat. Univ., Canberra, 1986.

\bibitem[{Mey}91]{MN91}
P.~{Meyer-Nieberg}.
\newblock {\em Banach lattices}.
\newblock Universitext. Springer-Verlag, Berlin, 1991.

\bibitem[Nil82]{Nil82}
P.~Nilsson.
\newblock Reiteration theorems for real interpolation and approximation spaces.
\newblock {\em Ann. Mat. Pura Appl. (4)}, 132:291--330 (1983), 1982.

\bibitem[Ovc84]{Ov84}
V.I. Ovchinnikov.
\newblock The method of orbits in interpolation theory.
\newblock {\em Math. Rep.}, 1(2):i--x and 349--515, 1984.

\bibitem[Pee69]{Pe69}
J.~Peetre.
\newblock Sur la transformation de {F}ourier des fonctions \`a valeurs
  vectorielles.
\newblock {\em Rend. Sem. Mat. Univ. Padova}, 42:15--26, 1969.

\bibitem[Pee71]{Pe71}
J.~Peetre.
\newblock Sur l'utilisation des suites inconditionellement sommables dans la
  th\'{e}orie des espaces d'interpolation.
\newblock {\em Rend. Sem. Mat. Univ. Padova}, 46:173--190, 1971.

\bibitem[Pee74]{Pe74}
J.~Peetre.
\newblock \"{U}ber den {D}urchschnitt von {I}nterpolationsr\"{a}umen.
\newblock {\em Arch. Math. (Basel)}, 25:511--513, 1974.

\bibitem[Pis16]{Pi16}
G.~Pisier.
\newblock {\em Martingales in {B}anach spaces}, volume 155 of {\em Cambridge
  Studies in Advanced Mathematics}.
\newblock Cambridge University Press, Cambridge, 2016.

\bibitem[PS16]{PS16}
J.~Pr\"{u}ss and G.~Simonett.
\newblock {\em Moving interfaces and quasilinear parabolic evolution
  equations}, volume 105 of {\em Monographs in Mathematics}.
\newblock Birkhh\"auser/Springer, Cham, 2016.

\bibitem[PSW18]{PSW18}
J.~Pr\"uss, G.~Simonett, and M.~Wilke.
\newblock Critical spaces for quasilinear parabolic evolution equations and
  applications.
\newblock {\em J. Differential Equations}, 264(3):2028--2074, 2018.

\bibitem[PW18]{PW18}
J.~Pr\"{u}ss and M.~Wilke.
\newblock On critical spaces for the {N}avier-{S}tokes equations.
\newblock {\em J. Math. Fluid Mech.}, 20(2):733--755, 2018.

\bibitem[{\v{S}}ne73]{Sn73}
I.Ja. {\v{S}}ne\u{\i}berg.
\newblock The solvability of linear equations in interpolational families of
  {B}anach spaces.
\newblock {\em Dokl. Akad. Nauk SSSR}, 212:57--59, 1973.

\bibitem[Ste56]{St56}
E.M. Stein.
\newblock Interpolation of linear operators.
\newblock {\em Trans. Amer. Math. Soc.}, 83:482--492, 1956.

\bibitem[SW06]{SW06}
J.~Su\'arez and L.~Weis.
\newblock Interpolation of {B}anach spaces by the {$\gamma$}-method.
\newblock In {\em Methods in {B}anach space theory}, volume 337 of {\em London
  Math. Soc. Lecture Note Ser.}, pages 293--306. Cambridge Univ. Press,
  Cambridge, 2006.

\bibitem[SW09]{SW09}
J.~Su\'arez and L.~Weis.
\newblock Addendum to ``{I}nterpolation of {B}anach spaces by the
  {$\gamma$}-method'' [{MR}2326391].
\newblock {\em Extracta Math.}, 24(3):265--269, 2009.

\bibitem[Tr{\`e}06]{Tr06}
F.~Tr{\`e}ves.
\newblock {\em Topological vector spaces, distributions and kernels}.
\newblock Dover Publications, Inc., Mineola, NY, 2006.
\newblock Unabridged republication of the 1967 original.

\bibitem[Tri78]{Tr78}
H.~Triebel.
\newblock {\em Interpolation theory, function spaces, differential operators},
  volume~18 of {\em North-Holland Mathematical Library}.
\newblock North-Holland Publishing Co., Amsterdam-New York, 1978.

\bibitem[Voi92]{Vo92}
J.~Voigt.
\newblock Abstract {S}tein interpolation.
\newblock {\em Math. Nachr.}, 157:197--199, 1992.

\bibitem[Wei02]{Wei02}
P.~Weidemaier.
\newblock Maximal regularity for parabolic equations with inhomogeneous
  boundary conditions in {S}obolev spaces with mixed {$L_p$}-norm.
\newblock {\em Electron. Res. Announc. Amer. Math. Soc.}, 8:47--51, 2002.

\bibitem[Wei06]{We06}
L.~Weis.
\newblock The {$H^\infty$} holomorphic functional calculus for sectorial
  operators---a survey.
\newblock In {\em Partial differential equations and functional analysis},
  volume 168 of {\em Oper. Theory Adv. Appl.}, pages 263--294. Birkh\"{a}user,
  Basel, 2006.

\bibitem[Wil71]{Wi71}
V.~Williams.
\newblock Generalized interpolation spaces.
\newblock {\em Trans. Amer. Math. Soc.}, 156:309--334, 1971.

\bibitem[Wol82]{Wo82}
T.H. Wolff.
\newblock A note on interpolation spaces.
\newblock In {\em Harmonic analysis ({M}inneapolis, {M}inn., 1981)}, volume 908
  of {\em Lecture Notes in Math.}, pages 199--204. Springer, Berlin-New York,
  1982.

\bibitem[Zaa67]{Za67}
A.C. Zaanen.
\newblock {\em Integration}.
\newblock North-Holland Publishing Co., Amsterdam; Interscience Publishers John
  Wiley \& Sons, Inc., New York, 1967.
\newblock Completely revised edition of An introduction to the theory of
  integration.

\bibitem[Zaf80]{Za80}
M.~Zafran.
\newblock Spectral theory and interpolation of operators.
\newblock {\em J. Functional Analysis}, 36(2):185--204, 1980.

\end{thebibliography}

\end{document}